\documentclass[10pt,abstract=true]{scrartcl}
\setkomafont{disposition}{\normalfont\rmfamily\bfseries}
\setkomafont{section}{\large}
\makeatletter
\renewcommand{\sectionlinesformat}[4]{
  \Ifstr{#1}{section}{\parbox{\linewidth}{\centering #3#4}}{\@hangfrom{\hskip #2#3}{#4}}%
}
\makeatother
\setkomafont{subsection}{\normalsize}
\setkomafont{author}{\normalsize}
\setkomafont{date}{\normalsize}
\RedeclareSectionCommand[beforeskip=-2.5ex plus -.6ex minus -.2ex, afterskip=1.3ex plus .2ex, tocbeforeskip=2pt]{section}
\RedeclareSectionCommand[beforeskip=-2ex plus -.5ex minus -.2ex, afterskip=.8ex plus .2ex]{subsection}
\usepackage{amssymb,mathrsfs,amsmath,amsthm}
\AddToHook{cmd/normalsize/after}{%
  \setlength{\abovedisplayskip}{7pt plus 2pt minus 3pt}%
  \setlength{\belowdisplayskip}{7pt plus 2pt minus 3pt}%
  \setlength{\abovedisplayshortskip}{0pt plus 2pt}%
  \setlength{\belowdisplayshortskip}{4pt plus 2pt minus 2pt}%
}
\numberwithin{equation}{section}
\usepackage[svgnames]{xcolor}
\usepackage{graphicx}
\usepackage{tikz,tikz-cd}
\usepackage[inline,shortlabels]{enumitem}
\setlist{topsep=2pt, itemsep=1pt, parsep=0pt, partopsep=0pt}
\usepackage[a4paper, margin=2.5cm]{geometry}
\usepackage[colorlinks,citecolor=brown,linkcolor=Purple,urlcolor=YellowGreen]{hyperref}
\newtheorem{theorem}{Theorem}[section]
\newtheorem{lemma}[theorem]{Lemma}

\newtheorem{proposition}[theorem]{Proposition}

\newtheorem{corollary}[theorem]{Corollary}

\theoremstyle{definition}
\newtheorem{definition}[theorem]{Definition}
\newtheorem{example}[theorem]{Example}

\newtheorem{construction}[theorem]{Construction}

\newtheorem*{claim*}{Claim}

\theoremstyle{remark}
\newtheorem{remark}[theorem]{Remark}

\def\simQ{\sim_{\mathbb Q}} 
\DeclareMathOperator\Alb{Alb}
\DeclareMathOperator\alb{alb}
\DeclareMathOperator\Proj{Proj}
\DeclareMathOperator\Pic{Pic}
\DeclareMathOperator\Tor{Tor}
\DeclareMathOperator\Supp{Supp}

\DeclareMathOperator\ord{ord}

\title{\Large Effective nonvanishing of $K$-trivial regular surfaces\\over imperfect fields}
\author{Jingshan Chen\thanks{School of Mathematics and Statistics, Hubei Minzu University, Enshi 445000, P.R.China.
Emails: \texttt{chjingsh@hbmzu.edu.cn} (Jingshan Chen) and \texttt{chongningwang@hbmzu.edu.cn} (Chongning Wang).},\quad
Chongning Wang\footnotemark[1]\quad and\quad Lei Zhang\thanks{School of Mathematical Science, University of Science and Technology of China, Hefei 230026, P.R.China.
Email: \texttt{zhlei18@ustc.edu.cn}.}}
\date{}

\begin{document}
\maketitle
\begingroup
\renewcommand{\thefootnote}{}
\footnotetext[0]{Date: \today.}
\endgroup
\vspace{-2em}
\begin{abstract}
  For any regular projective surface $X$ over an arbitrary field $k$ of characteristic $p>0$ with $K_X\equiv0$, we prove that
  \[
    60K_X\sim0 .
  \]
  This uniform bound is optimal.
  To show that the factor $5$ is necessary, we construct an example in characteristic $5$ for which $K_X$ has exact order five.

  If $X$ is geometrically normal, then $12K_X\sim0$. 
  Our main focus is the geometrically non-normal case, where we obtain more precise bounds according to the Albanese morphism.
  If it is nontrivial, then $p=2$ or $3$, and we prove $12K_X\sim0$ and $6K_X\sim0$, respectively; both bounds are optimal.
  If it is trivial, then $K_X\sim0$ for $p\ge7$, while $5K_X\sim0$, $3K_X\sim0$, and $4K_X\sim0$ for $p=5,3,2$, respectively.
\end{abstract}
\tableofcontents

\section{Introduction}\label{sec:introduction}
For projective normal varieties $X$ with mild singularities and
$K_X \equiv 0$, which we call \emph{$K$-trivial varieties}, the following
questions arise naturally:
\begin{itemize}
  \item (Nonvanishing) Is it true that $\kappa(X) \ge 0$? Equivalently, does $K_X \sim_{\mathbb{Q}} 0$ hold?
  \item (Effectivity) Does there exist a positive integer $N= N(d)$, depending only on the dimension $d$ of $X$, such that $N K_X \sim 0$ holds for every such $X$?
\end{itemize}

For regular $K$-trivial surfaces over an algebraically closed field, it is well known that $12K_X\sim0$; see \cite{B-DF,BHPV,BM-II,BM-III}.
The same result holds over every perfect field by descent.

Over an algebraically closed field of characteristic $0$, it is known that for terminal threefolds $X$ with $K_X \equiv 0$, we have
$2^5 \cdot 3^3 \cdot 5^2 \cdot 7 \cdot 11 \cdot 13 \cdot 17 \cdot 19 K_X \sim 0$ (see \cite{Beauville83Rmk,Kaw86,Morrison86,Oguiso93}).
One key idea in the proof is that when $q(X) := h^1(X,\mathcal O_X) > 0$, the Albanese morphism $X\to \Alb_X$ is an étale fiber bundle (\cite{Kawamata-1985-minimal-model}). This reduces the problem to a lower-dimensional case. 

More generally, the Index Conjecture predicts that for complex projective klt pairs $(X,D)$ with $K_X+D\sim_{\mathbb Q}0$ and coefficients of $D$ in a fixed DCC set $I \subset \mathbb Q$, the least positive integer $m$ with $m(K_X+D)\sim0$ is bounded in terms of $\dim X$ and $I$ (see \cite{Xu-2019} and the references therein). This is known in dimension at most three \cite{Xu-2019,Jiang-Liu-2021} and in dimension four when $D\ne0$ \cite{Xu-2019}.

Over an algebraically closed field of characteristic $p>0$, not much is known.
In \cite{CWZ-classification}, we obtain a concrete description of $K$-trivial threefolds $X$ with $q(X) > 0$. 
In the case $q=2$, we proved $2^4\cdot 3^3 K_X \sim 0$; for $q = 1$, we proved that the Albanese morphism $\alb_X\colon X \to \Alb_X$ is a fibration and $X \cong (E \times F)/G$, where $E$ is an elliptic curve, $F$ is a general fiber of $\alb_X$ and $G$ is a finite group scheme acting diagonally on $E \times F$.
Note that the surface $F$ is not necessarily normal. 
Equivalently, the generic fiber $X_\eta$ of $\alb_X$, which is a regular surface over the function field $K(\Alb_X)$ with $K_{X_\eta} \equiv 0$, is not necessarily geometrically normal.

More generally, fibrations whose generic fibers are regular $K$-trivial varieties arise naturally in birational geometry, including Iitaka fibrations defined by semi-ample canonical divisors and Albanese fibrations of $K$-trivial varieties.
Studying geometrically non-normal regular $K$-trivial varieties over imperfect fields is therefore important for understanding such fibrations.

We study effective nonvanishing for $K$-trivial regular surfaces over imperfect fields.
Tanaka proved nonvanishing in this setting, namely $K_X \sim_{\mathbb{Q}} 0$ (\cite[Theorem~1.1]{tanaka_abundance_2020}).
Our main result establishes effectivity as follows.

\begin{theorem}[{Propositions~\ref{prop:geo-normal}, \ref{prop:a_X-surjective}, and \ref{prop:Alb-trivial}; Corollary~\ref{cor:case-C-p-ge7}}]\label{thm:main}%
  Let $X$ be a regular projective surface over a field $k$ of characteristic $p>0$ with $H^0(X,\mathcal O_X) = k$ and $K_X \equiv 0$. 
  We have
  \begin{itemize}
    \item[\textup{(A)}]  If $X$ is geometrically normal, then $12 K_X \sim 0$.
    \item[\textup{(B)}]  If $X$ is geometrically non-normal and the Albanese morphism $X\to \Alb_{X/k}$ is nontrivial, then $p=2$ or $3$; moreover, $12K_X \sim 0$ if $p=2$, while $6K_X \sim 0$ if $p=3$.
    \item[\textup{(C)}]  If $X$ is geometrically non-normal and the Albanese morphism of $X$ is trivial, then 
      \begin{itemize}
        \item if $p \ge 7$, then $K_X \sim 0$;
        \item if $p = 5$, then $5 K_X \sim 0$;
        \item if $p = 3$, then $3 K_X \sim 0$;
        \item if $p = 2$, then $4 K_X \sim 0$.
      \end{itemize}
  \end{itemize}
\end{theorem}

\begin{corollary}\label{cor:uniform-bound}
  Let $X$ be a regular projective surface over an arbitrary field $k$ with $K_X\equiv0$.  Then
  \[
    60 K_X\sim0.
  \]
  Moreover, if $\operatorname{char} k\ne5$, then $12 K_X \sim 0$.
\end{corollary}

The bound in Theorem~\ref{thm:main}\textup{(C)} is sharp in characteristic $5$.
\begin{theorem}[{Example~\ref{ex:case-C-p5}}]\label{prop:order-five-intro}
  Let $k = \mathbf F_5(s_1,s_2)$, where $s_1,s_2$ are algebraically independent over $\mathbf F_5$.
  Then there exists a geometrically non-normal regular projective surface $X$ over $k$ with $H^0(X,\mathcal O_X) = k$ and $K_X\equiv0$ such that
  \[
    \ord(K_X) = 5.
  \]
\end{theorem}

As a byproduct of the analysis of Case (B), we obtain the following description of the Albanese morphism of regular $K$-trivial surfaces.
It can be viewed as a generalization of the (quasi-)bielliptic surface structure of \cite{BM-II,BM-III}.
\begin{theorem}[Combine Propositions~\ref{prop:geo-normal-alb} and~\ref{prop:a_X-surjective} with Lemma~\ref{lemm:sur-second-fibration}]\label{thm:albanese-structure-intro}
  Let $X$ be a regular projective surface over a separably closed field $k$ of characteristic $p>0$, with $H^0(X,\mathcal O_X) = k$ and $K_X \equiv 0$.
  The Albanese morphism $\alb_{X/k}\colon X\to A:=\Alb_{X/k}$ is a fibration. Moreover:
  \begin{enumerate}[label=\textup{(\arabic*)}]
    \item If $\dim A = 2$, then $\alb_{X/k}$ is an isomorphism.
    \item If $\dim A = 1$, then $X\simeq(A\times_k C)/G$, where $C$ is an Albanese fiber and $G\subset A$ is a finite subgroup scheme acting diagonally on the product.
  \end{enumerate}
\end{theorem}

Our argument also gives the following result for numerically trivial line bundles on surfaces with nef anti-canonical divisor.
\begin{theorem}[{Propositions~\ref{prop:Alb-trivial} and~\ref{prop:Alb-trivial-nef}; Corollary~\ref{cor:case-C-p-ge7}}]\label{thm:nef-anticanonical}
  Let $X$ be a regular projective surface over a field $k$ of characteristic $p > 0$ with $H^0(X,\mathcal O_X) = k$.
  Assume that $-K_X$ is nef, $\dim \Alb_{X/k} = 0$, and $X$ is geometrically non-normal.
  Then every numerically trivial line bundle $L$ on $X$ satisfies the following:
  \begin{enumerate}[label=\textup{(\arabic*)}]
    \item If $p \ge 7$, then $L \sim 0$.
    \item If $p\in\{3,5\}$, then $pL\sim0$.
    \item If $p = 2$, then $2^2L \sim 0$.
  \end{enumerate}
\end{theorem}

Section~\ref{sec:proof} establishes the bounds $pL\sim0$ for $p\ge3$ and $4L\sim0$ for $p=2$.
To obtain the stronger conclusion $L\sim0$ when $p\ge7$, we introduce the following construction.

\subsection*{The \texorpdfstring{$t$}{t}-twisted \texorpdfstring{$\mu_p$}{mu-p}-cover}
We first recall the simple $\mu_n$-cover (see for example \cite[\S0.2]{Cossec-Dolgachev-Liedtke-2025}).
Let $X$ be a variety over a field $k$, let $L$ be a line bundle on $X$, and let $n\ge1$.
A nonzero section $s\in H^0(X,L^{\otimes n})$ defines a finite flat cover of degree $n$
\[
  \pi_s\colon Y:=\operatorname{Spec}_X\!\left(\bigoplus_{i=0}^{n-1}L^{-i}\right)\longrightarrow X,
\]
where multiplication uses the map $L^{-n}\to\mathcal O_X$ given by evaluation on $s$.
The grading defines a $\mu_n$-action with quotient $X$.
If $s$ is nowhere vanishing, then $\pi_s$ is a $\mu_n$-torsor and $\pi_s^*L$ is trivial.

Now assume that $\operatorname{char}k=p>0$ and fix a trivialization $\tau\colon L^{\otimes p}\xrightarrow{\sim}\mathcal O_X$, with $s:=\tau^{-1}(1)$.
The $\mu_p$-cover associated to $(L,s)$ need not be normal, even when $X$ is regular.
To obtain a regular cover, we introduce a transcendental twist.

For $t$ transcendental over $k$, set $K=k(t)$ and $X_K=X\times_kK$.
The simple cover associated to $(L_K,ts_K)$ is called the \emph{$t$-twisted $\mu_p$-cover} and is denoted by $\pi\colon\widehat Y\to X_K$.
Introducing $t$ makes $\widehat Y$ regular whenever $X$ is regular.
This allows us to construct a geometrically non-reduced regular variety from some special $p$-torsion line bundle as follows.
\begin{theorem}[{Proposition~\ref{prop:generically-twisted-mu-p-cover}}]\label{thm:twisted-cover-intro}
  Let $X$ be a regular projective variety over a field $k$ of characteristic $p>0$ with $H^0(X,\mathcal O_X) = k$.
  Assume that $X_{\bar k}$ is integral, and let $\nu\colon X_{\bar k}^\nu\to X_{\bar k}$ be its normalization.
  Let $L$ be a line bundle on $X$ with $\ord(L) = p$ and $\nu^*L_{\bar k}\simeq\mathcal O_{X_{\bar k}^\nu}$.
  Then the $t$-twisted $\mu_p$-cover $\pi\colon\widehat Y\to X_K$ satisfies the following:
  \begin{enumerate}[label=\textup{(\arabic*)}]
    \item $\widehat Y$ is a regular projective variety over $K$ with $H^0(\widehat Y,\mathcal O_{\widehat Y}) = K$;
    \item $\omega_{\widehat Y}\simeq\pi^*\omega_{X_K}$ and $\pi^*L_K\simeq\mathcal O_{\widehat Y}$;
    \item $\widehat Y$ is geometrically non-reduced over $K$.
  \end{enumerate}
\end{theorem}

By \cite[Theorem~1.2]{WZ26}, a regular projective $K$-trivial surface over a field of characteristic $p\ge7$ is geometrically integral.
In Case~\textup{(C)} of Theorem~\ref{thm:main} with $p\ge7$, the proof of Proposition~\ref{prop:Alb-trivial} shows that every numerically trivial line bundle $L$ satisfies $L^{\otimes p}\simeq\mathcal O_X$ and becomes trivial on $X_{\bar k}^\nu$.
If $L$ were nontrivial, Theorem~\ref{thm:twisted-cover-intro} would produce a regular geometrically non-reduced surface over $K$ with numerically trivial canonical divisor, contradicting \cite[Theorem~1.2]{WZ26}.
Thus $L\sim0$, and in particular $K_X\sim0$ (Corollary~\ref{cor:case-C-p-ge7}).

\subsection*{Strategy of the proof}
We prove Case~\textup{(A)} by taking a geometric base change, resolving the resulting singularities, and then applying the classification of smooth surfaces and descent (Section~\ref{sec:caseA}).

In Cases~\textup{(B)} and~\textup{(C)}, the Frobenius base change $X\times_k k^{1/p}$ is non-normal.
Set $X_1 = (X\times_k k^{1/p})_{\mathrm{red}}^\nu$ to be the normalization, and let $\phi\colon X_1\to X$ be the natural morphism.
The canonical base-change formula (Theorem~\ref{th:geonon-normal}) gives
\[
  K_{X_1}+(p-1)\mathfrak C\sim\phi^*K_X\equiv0,
  \qquad \mathfrak C>0.
\]
Thus $K_{X_1}$ is not pseudo-effective.
Resolving the singularities of $X_1$ and running the minimal model program gives the diagram
\[
  \begin{tikzcd}[row sep=tiny]
    && \widetilde X_1 \ar[ld,"\mu"'] \ar[rd,"\epsilon"]\\
    X & X_1 \ar[l,"\phi"'] && Z \ar[d,"g"]\\[2ex]
      &&& B\rlap{\,.}
  \end{tikzcd}
\]
Here $\mu$ is the minimal resolution, $\epsilon$ is the MMP contraction, and $g$ is a Mori fiber space.
The pullback $\mu^*\phi^*\omega_X$ descends to a numerically trivial line bundle on the del Pezzo surface $Z$ if $\dim B = 0$, or to a degree-zero line bundle on $B$ if $\dim B = 1$.
Since $\phi$ is purely inseparable of height one, a relation $m\phi^*K_X\sim0$ for a positive integer $m$ yields $pmK_X\sim0$.

\smallskip
In Case~\textup{(B)}, the nontrivial Albanese morphism forces $B$ to be a curve of positive geometric genus.
Let $f_1\colon X_1\to B$ be the fibration induced by $g$.
Restricting the canonical formula to a general fiber gives $p = 2$ or $3$ and shows that $\mathfrak C$ contains an $f_1$-horizontal prime divisor $C_1$.
A key step is to prove that $C_1^2 = 0$ and $C_1$ is semi-ample (Lemma~\ref{lemm:C1^2}).
In higher dimensions, statements of this kind are proved via the positivity of direct images of pluricanonical sheaves \cite{EP23,Ejiri-2023-splitting,CWZ-classification}.
But since our study is restricted to surfaces, we apply a more classical treatment following \cite[Proposition~3.1]{tanaka_abundance_2020}, which goes back to Mumford \cite[\S2]{Mumford-69-I}, see also \cite[Lemma, p.~682]{Masek-1993}.

A sufficiently divisible multiple of $C_1$ therefore defines a second fibration $h_1\colon X_1\to D_1$ transversal to $f_1$:
\[
  \begin{tikzcd}
    X_1 \ar[r,"h_1"] \ar[d,"f_1"'] & D_1 \\
    B\rlap{\,.} &
  \end{tikzcd}
\]
Here $C_1$ is vertical for $h_1$.
On $X$, this yields a fibration transversal to the Albanese morphism, which is the analogue of the (quasi-)bielliptic structure in \cite{BM-II}.
Adjunction along the fibers of $h_1$ shows that $X_1$ is regular; calculations of normal bundles and descent then give $12K_X\sim0$ for $p = 2$ and $6K_X\sim0$ for $p = 3$ (Section~\ref{sec:caseB}).

\smallskip
In Case~\textup{(C)}, either $Z$ is del Pezzo or $B$ has geometric genus zero.
As explained before, we only need to show that $pK_X \sim 0$ in characteristic $p\ge3$ and $4K_X \sim 0$ in characteristic $2$.

For the del Pezzo case, every numerically trivial line bundle $L$ satisfies $L\sim0$ when $p\ge5$, $3L\sim0$ when $p=3$, and $4L\sim0$ when $p=2$ \cite[Theorem~4.1]{Bernasconi-Tanaka-22}.
This is weaker than we need when $p=2$ or $3$, and we improve the bound under additional assumptions (Lemmas~\ref{lemm:del-Pezzo-p3} and~\ref{lemm:del-Pezzo-p2}).

In the ruled case, the pullback $\mu^*\phi^*K_X$ descends to a degree-zero line bundle $L_B$ on $B$.
We analyze the horizontal components in the conductor divisor $\mathfrak C$ and use adjunction and positivity to bound the arithmetic genus of $B$ by $p_a(B) \le 1$.
The line bundle $L_B$ is trivial if $p_a(B) = 0$.
When $p_a(B) = 1$, $B$ is geometrically non-normal, and we can use the geometric singularities of $B$ to obtain some properties on the geometric singularities of $X$, see Subcase 1(a) in the proof of Proposition~\ref{prop:Alb-trivial} for example.

\medskip
\paragraph{Acknowledgements}
This research is partially supported by CAS Project for Young Scientists in Basic Research (No.~YSBR-032), NSFC (No.~12122116 and No.~12471495), and Quantum Science and Technology-National Science and Technology Major Project (QNMP, No.~2021ZD0302902). 
The first two authors are also supported by Hubei Minzu University under Grant No.~BS26011 and XN24040, respectively.

\medskip
\paragraph{AI Statement}
The authors wrote the first draft in 2025 and obtained a $p$-dependent bound on the torsion order of $K_X$.
To improve the results, AI models were used --- primarily GPT-5.6-Sol, GPT-6-astra and claude-opus-5.
Specifically, AI provided the $t$-twisted $\mu_p$-cover construction (\S\ref{sec:case-C-p-ge7}), improved the result in Case (B) when $p=3$ (Lemma~\ref{lemm:caseB-normal-bundle}), improved the result in Case (C) (part of \S\ref{sec:proof}) and constructed examples (\S\ref{sec:examples}).

\section{Preliminary}\label{sec:preliminary}
\subsection{Notation and conventions}\label{sec:notation-conventions}
\begin{itemize}
  \item Throughout this paper, unless otherwise stated, $k$ is a field of characteristic $p>0$.
    We fix an algebraic closure $\bar k$ and write $k^{1/p^n}:=\{a\in\bar k\mid a^{p^n}\in k\}$ for $n\geq0$.
    The field $k$ is called \emph{$F$-finite} if $[k:k^p]<\infty$.
  \item By a $k$-scheme, we mean a separated scheme of finite type over $k$.
    By a \emph{variety} over $k$ we mean an integral $k$-scheme.
    By a \textit{surface}, we mean a variety of dimension two.

  \item For a field extension $\ell/k$ and a $k$-scheme $X$, we write $X_\ell:=X\times_k\ell$.
  \item For a scheme $X$, we denote by $X_{\mathrm{red}}$ the induced reduced scheme and by $X_{\mathrm{red}}^\nu$ its normalization.
  \item A \emph{fibration} is a projective surjective morphism $f\colon X\to Y$ of varieties such that $f_*\mathcal O_X=\mathcal O_Y$.
  \item For a line bundle $L$ of finite order in $\Pic(X)$, we write $\ord(L)$ for its order; when $\omega_X$ is invertible, we write $\ord(K_X):=\ord(\omega_X)$.
  \item For an integral projective curve $C$ over $k$, put $k_C := H^0(C,\mathcal O_C)$.
    We write
    \[
      p_a(C) := \dim_{k_C}H^1(C,\mathcal O_C)
    \]
    for its \emph{arithmetic genus over the field of constants}.
    For a geometrically irreducible projective curve $C$, its \emph{geometric genus} is the genus of $(C_{\bar k})_{\mathrm{red}}^\nu$ over $\bar k$.

  \item For a scheme $X$ of characteristic $p>0$, let $F_X\colon X\to X$ denote the absolute Frobenius, $F_X^n$ its $n$-th iterate, and $F_k^n$ the corresponding morphism of $\operatorname{Spec}k$.

  \item Let $X$ be a $k$-scheme with structure morphism $\varphi\colon X\to\operatorname{Spec}k$, and set $X^{(p^n)}:=X\times_{\operatorname{Spec}k,F_k^n}\operatorname{Spec}k$.
The \emph{$n$-th relative Frobenius morphism} $F^n_{X/k}\colon X\to X^{(p^n)}$ is the unique morphism making the following diagram commute:
\[
  \begin{tikzcd}
    X \ar[rd,dashed,"F^n_{X/k}"] \ar[rrd,bend left=15,"F_X^n"] \ar[rdd,bend right=15,"\varphi"']\\
    & X^{(p^n)} \rar\dar\ar[rd,phantom,"\square"] & X \dar{\varphi} \\
    & \mathrm{Spec}\, k \rar{F_k^n} & \mathrm{Spec}\, k
  \end{tikzcd}
\]
  \item Let $\pi\colon X\to Y$ be a finite purely inseparable morphism.
We say that $\pi$ is of \emph{height $n$} if there exists a morphism $\rho\colon Y\to X$ such that $\rho\circ\pi=F_X^n$.
\end{itemize}

\subsection{Some basic results}\label{sec:singularities-field-extensions}%
\begin{lemma}\label{lemma-finite-Cartier}%
  Let $\pi\colon X\to Y$ be a finite surjective morphism of degree $d$ between normal varieties.
  \begin{itemize}
    \item[\textup{(1)}] If $D_Y$ is a Weil divisor on $Y$ such that
      $\pi^*D_Y\sim0$, then $dD_Y\sim0$.
    \item[\textup{(2)}] If $\pi$ is purely inseparable of height $n$ and
      $\pi^*D_Y\sim0$, then $p^nD_Y\sim0$.
    \item[\textup{(3)}] If $\pi$ is purely inseparable of height $n$ and $Y$ is
      regular, then $p^nD_X$ is Cartier for every Weil divisor $D_X$ on $X$.
  \end{itemize}
\end{lemma}
\begin{proof}
  See \cite[Lemma~6.8]{CWZ-classification} for (1) and (2).
  To show (3), we factor the $n$-th iterate absolute Frobenius $F^n_X$ into $X \buildrel \pi\over\to Y \buildrel \rho\over\to X$.
  Since $Y$ is regular, the pullback divisor $\rho^*D_X$ is Cartier.
  Therefore, $p^nD_X = F_X^{n*} D_X = \pi^*(\rho^*D_X)$ is Cartier.
\end{proof}

\begin{lemma}\label{lemm:geometric-genus-zero-torsion}
  Let $k$ be an $F$-finite separably closed field of characteristic $p>0$, and let $T$ be a regular projective curve over $k$ such that $H^0(T,\mathcal O_T)=k$ and $T$ has geometric genus zero.
  Let $M$ be a degree-zero line bundle on $T$.
  \begin{enumerate}[label=\textup{(\arabic*)}]
    \item The line bundle $M$ has $p$-power order.
    \item If $p_a(T)=0$, then $M\sim0$.
      If $p_a(T)=1$, then $p\in\{2,3\}$, the curve $T_1:=(T\times_k k^{1/p})_{\mathrm{red}}^\nu$ has arithmetic genus zero over its field of constants, and $pM\sim0$.
  \end{enumerate}
\end{lemma}
\begin{proof}
  The assertions follow from \cite[Remark~5.2 and the proof of Theorem~11.8]{tanaka_invariants_2021}, Riemann--Roch, and Lemma~\ref{lemma-finite-Cartier}\textup{(2)}.
\end{proof}

\begin{lemma}\label{lemma-smooth-Cartier-divisor}
  Let $X$ be a scheme of finite type over $k$, and let $Z\subset X$ be a closed subscheme smooth over $k$.
  If $Z$ is an effective Cartier divisor on $X$, then $X$ is smooth over $k$ on an open neighborhood of $Z$.
\end{lemma}
\begin{proof}
  Since $Z$ is an effective Cartier divisor on $X$, the immersion $Z\hookrightarrow X$ is regular.
  The assertion follows from \cite[Expos\'e~II, Corollary~4.16]{SGA1}.
\end{proof}

\begin{lemma}[Riemann--Hurwitz, see {\cite[Theorem~7.4.16]{Liu-AGAC}}]\label{lemma-canonical-curve-comparison}%
  Let $\pi\colon B\to C$ be a finite \emph{separable} morphism between regular projective curves over $k$.
  Then
  \[
    K_B\sim\pi^*K_C+\sum_x(e'_x-1)[x],
    \qquad e'_x\geq e_x\geq1,
  \]
  where $x$ runs over the closed points of $B$ and $e_x$ is the ramification index.
  Moreover, $e'_x=e_x$ precisely when $\pi$ is tamely ramified at $x$, and $e'_x=1$ precisely when $\pi$ is \'etale at $x$.
\end{lemma}

\subsection{Intersection numbers}\label{sec:intersection-numbers}%
  Let $X$ be a projective normal surface over $k$.
  Let $D$ be a Cartier divisor on $X$, and let $C$ be an integral curve on $X$.
  For a line bundle $L$ on $C$, write $\deg_k L := \chi_k(C,L)-\chi_k(C,\mathcal O_C)$, where $\chi_k$ denotes the Euler characteristic computed over $k$.
  Thus $\deg_k L=[k_C:k]\deg_{k_C}L$, where $k_C=H^0(C,\mathcal O_C)$.
  The intersection number $(D\cdot C)_k$ of $D$ on $C$ is defined to be
  \[
    (D\cdot C)_k := \deg_k\bigl(\mathcal O_X(D)|_C\bigr).
  \]
  For a finite surjective morphism $\nu\colon C'\to C$ of integral projective curves over $k$, the projection formula gives
  \[
    \deg_k\bigl(\nu^*(\mathcal O_X(D)|_C)\bigr)
    = [K(C'):K(C)](D\cdot C)_k.
  \]
  In particular, for the normalization $\nu\colon C^\nu\to C$, we have $(D\cdot C)_k=\deg_k\bigl(\nu^*(\mathcal O_X(D)|_C)\bigr)$.
  \medskip

Let $X$ be a geometrically irreducible variety over $k$, let $\ell/k$ be a finite field extension, and set $Y:=(X\times_k\ell)_{\mathrm{red}}^\nu$.
We define the \emph{thickening exponent of this base change} by
\[
  \varepsilon(X;\ell/k)
  :=\log_p\frac{[\ell:k]}{[K(Y):K(X)]}.
\]
It is a nonnegative integer, and it vanishes precisely when $X_\ell$ is reduced; cf.~\cite[Theorem~7.3\textup{(1)--(2)}]{tanaka_invariants_2021}.

\begin{lemma}\label{lemm:intersection-field-extension}%
  Let $X$ be a projective normal geometrically irreducible $k$-surface, $D$ a Cartier divisor on $X$, and $C\subset X$ an integral curve.
  For a finite extension $\ell/k$, let $f\colon Y:=(X\times_k\ell)_{\mathrm{red}}^\nu\to X$ be the natural morphism.
  Then
  \[
    (D\cdot C)_k
    =p^{\varepsilon(X;\ell/k)}(f^*D\cdot f^*C)_\ell,
  \]
  where $f^*C$ denotes the pullback as a Weil divisor.
\end{lemma}
\begin{proof}
  Set $d=[K(Y):K(X)]$.
  Since $f_*f^*C=dC$, the projection formula gives
  \[
    [\ell:k](f^*D\cdot f^*C)_\ell
    =(D\cdot f_*f^*C)_k
    =d(D\cdot C)_k.
  \]
  The assertion follows from $p^{\varepsilon(X;\ell/k)}=[\ell:k]/d$.
\end{proof}

\begin{corollary}\label{lemma-intersection-general}%
  Let $X$ be a projective normal $k$-surface with $H^0(X,\mathcal O_X)=k$, let $\ell/k$ be a finite purely inseparable extension, and set $Y:=(X_\ell)_{\mathrm{red}}^\nu$.
  For a $\mathbb Q$-Cartier integral curve $C_Y\subset Y$, its image $C_X\subset X$ is $\mathbb Q$-Cartier and
  \[
    C_Y^2>0\quad(\text{resp. }<0,\ =0)
    \quad\Longleftrightarrow\quad
    C_X^2>0\quad(\text{resp. }<0,\ =0).
  \]
\end{corollary}

\subsection{Adjunction formula}\label{sec:adjunction}%
\begin{proposition}[{\cite[Definition~2.34 and Proposition~2.35]{Kollar-Sing}}]\label{prop:adj}%
  Let $X$ be a normal surface and $C\subset X$ a reduced curve with normalization $C^\nu\to C$.
  Assume that $K_X + C$ is $\mathbb Q$-Cartier.
  Then there exists an effective $\mathbb Q$-divisor $\operatorname{Diff}_{C^\nu}$ on $C^\nu$, called the different, such that \[
        (K_X + C)|_{C^\nu} \simQ K_{C^\nu} + \operatorname{Diff}_{C^\nu}.
      \]
  Moreover, $\operatorname{Diff}_{C^\nu} = 0$ if and only if $C$ is regular and $X$ is regular along $C$.
\end{proposition}

\subsection{A semi-ampleness criterion}\label{sec:semi-ampleness}
\begin{proposition}[special case of {\cite[Proposition~3.1]{tanaka_abundance_2020}}]\label{prop:tanaka-abun-3.1}%
  Let $X$ be a projective normal surface over $k$.
  Let $B$ be an integral curve on $X$, and assume that $K_X$ and $B$ are $\mathbb Q$-Cartier.
  Assume that
  \begin{itemize}
    \item[\textup{(1)}] $K_X \cdot B = B^2 = 0$;
    \item[\textup{(2)}] there exists a positive integer $n$ such that $nB$ is Cartier and $\mathcal O_X(nB)|_B\cong\mathcal O_B$.
  \end{itemize}
  Then $B$ is semi-ample.
\end{proposition}

\subsection{Positivity}\label{sec:positivity}
\begin{lemma}[a special case of {\cite[Corollary~4.6]{Ejiri-2019}}]\label{lemm:Ejiri-base-curve}
  Let $k$ be an $F$-finite field of characteristic $p>0$.
  Let $f\colon S\to C$ be a surjective morphism from a regular projective surface to a regular projective curve such that $H^0(C,\mathcal O_C)=k$, and let $\Delta\ge0$ be an effective $\mathbb Q$-divisor whose Cartier index is not divisible by $p$.
  Assume that the geometric generic fiber $(S_{\bar\eta},\Delta_{\bar\eta})$ is $F$-pure.
  Then the following statements hold.
  \begin{enumerate}[label=\textup{(\arabic*)}]
    \item If $-(K_S+\Delta)$ is nef, then $\deg K_C \leq 0$.
    \item If $-(K_S+\Delta)$ is ample, then $\deg K_C < 0$.
    \item If $(S_{\bar\eta},\Delta_{\bar\eta})$ is strongly $F$-regular and $-(K_S+\Delta)$ is nef and big, then $\deg K_C < 0$.
  \end{enumerate}
\end{lemma}

\subsection{Behavior of canonical divisors under purely inseparable field extensions}\label{sec:geo-non-normal-base-change}%
Let $X$ be a normal variety over $k$ such that $k$ is algebraically closed in $K(X)$.
Following \cite[Definition~5.1]{tanaka_invariants_2021}, the \emph{Frobenius length of geometric non-normality} $\ell_F(X/k)$ is the least integer $n\geq0$ such that
\[
  (X\times_k{k^{1/p^n}})_{\mathrm{red}}^\nu
\]
is geometrically normal over $k^{1/p^n}$.
This integer exists by \cite[Remark~5.2]{tanaka_invariants_2021}.

\begin{theorem}[{\cite[Theorem~1.1]{tanaka_behavior_2018} and \cite[Theorem~1.1]{ji_structure_2021}}]\label{th:geonon-normal}%
  Let $X$ be a normal projective variety over a field $k$ with $H^0(X,\mathcal O_X) = k$.
  Set either $\ell = k^{1/p}$ or $\ell = \bar k$.
  Let $Y := (X \times_k \ell)_{\mathrm{red}}^\nu$.
  Then there exists an effective $\mathbb{Z}$-Weil divisor $\mathfrak C = \mathfrak M + \mathfrak F$ such that 
  \begin{equation}\label{eq:WR1Y}
    -K_{Y/X} \sim (p-1) \mathfrak C ,
  \end{equation}
  where
  \begin{itemize}
    \item $\mathfrak C = 0$ if and only if $X$ is geometrically normal;
    \item $\mathfrak{F}$ is supported on the conductor divisor of the normalization map $Y\to (X \times_k \ell)_{\mathrm{red}}$; and 
    \item $\mathfrak M$ is a movable divisor, and $\mathfrak M\ne0$ if and only if $X$ is geometrically non-reduced.
  \end{itemize}
\end{theorem}
The original statement of \cite[Theorem~1.1]{ji_structure_2021} assumes that $k$ is a function field, but this assumption is unnecessary as noted in \cite[page~3917]{Das-Waldron-2022}.

We recall how the divisor $\mathfrak C = \mathfrak M + \mathfrak F$ is constructed in \cite[\S3.2]{ji_structure_2021} in terms of foliations when $\ell = k^{1/p}$ and $k$ is $F$-finite.
Let $r := \dim_\ell\operatorname{Der}_k(\ell)$.
Let $\phi\colon Y\to X$ be the induced morphism, and let
\[
  \mathscr G := T_{Y/X} = \mathcal{H}om_{\mathcal O_Y}(\Omega_{Y/X},\mathcal O_Y)
\]
be the foliation associated with $\phi$.
Restricting local derivations in $\mathscr G$ to the subfield $\ell\subset K(Y)$ gives a natural inclusion
\[
  \mathscr G \hookrightarrow \operatorname{Der}_k(\ell)\otimes_\ell\mathcal O_Y \simeq \mathcal O_Y^{\oplus r}.
\]
Let $\mathscr G'$ be the saturation of $\mathscr G$ in this trivial bundle, and denote the inclusion by $\iota\colon\mathscr G\hookrightarrow\mathscr G'$.
Then we have the exact sequences
\begin{equation}\label{eq:foliation-inclusion-quotient}
  0\longrightarrow\mathscr G
  \xrightarrow{\iota}\mathscr G'
  \longrightarrow\operatorname{coker}(\iota)\longrightarrow0
\end{equation}
and
\begin{equation}\label{eq:foliation-saturation-quotient}
  0\longrightarrow\mathscr G'
  \longrightarrow\operatorname{Der}_k(\ell)\otimes_\ell\mathcal O_Y
  \longrightarrow\mathscr Q\longrightarrow0,
\end{equation}
where $\mathscr Q$ is torsion-free by definition of saturation.
The zero divisor of $\det(\iota)$ is the fixed divisor $\mathfrak F$ from Theorem~\ref{th:geonon-normal}; see \cite[Definition~3.2]{ji_structure_2021}.
If $X\times_k\ell$ is reduced (equivalently, $\mathfrak M=0$), then $\mathscr G'=\operatorname{Der}_k(\ell)\otimes_\ell\mathcal O_Y\simeq\mathcal O_Y^{\oplus r}$, so $\mathfrak F$ is the zero divisor of $\det$ of the inclusion $\mathscr G\hookrightarrow\mathcal O_Y^{\oplus r}$; see \cite[Remark~3.3]{ji_structure_2021}.

To describe $\mathfrak M$, we keep track of how the fiber of $\mathscr G'$ inside $V:=\operatorname{Der}_k(\ell)$ varies with the point of $Y$.
Put $q:=\operatorname{rank}\mathscr G$ and $s:=r-q=\operatorname{rank}\mathscr Q$.
At a general point $y\in Y$, the fiber of $\mathscr G'$ is a $q$-dimensional subspace of $V\otimes_\ell\kappa(y)$; sending $y$ to this subspace and composing with the Pl\"ucker embedding gives a rational map (cf.\ \cite[\S2.3]{ji_structure_2021})
\begin{equation}\label{eq:foliation-grassmannian-map}
  \gamma_0\colon Y\dashrightarrow\operatorname{Gr}(q,V)\hookrightarrow\mathbb P(\wedge^s V).
\end{equation}
We call $\gamma_0$ the Gauss map associated with $\mathscr G$, using the same notation for its composition with the Pl\"ucker embedding.
Its domain is the open locus $U\subseteq Y$ where $\mathscr G'$ is a subbundle of $V\otimes_\ell\mathcal O_Y$, i.e.\ where $\mathscr Q$ is locally free, whose complement has codimension at least $2$ since $Y$ is normal and $\mathscr Q$ is torsion-free.
Taking $s$-th exterior powers of $V\otimes_\ell\mathcal O_Y\twoheadrightarrow\mathscr Q$ gives a map $\wedge^s V\otimes_\ell\mathcal O_Y\to\det\mathscr Q$, which is surjective on $U$, where it is the pullback by $\gamma_0$ of $\wedge^s V\otimes\mathcal O\twoheadrightarrow\mathcal O(1)$.
Hence the sections of $\det\mathscr Q$ coming from $\wedge^s V$ span the linear system $\mathcal M = \gamma_0^*|\mathcal O(1)|$, whose base locus is exactly $Y\setminus U$, of codimension at least $2$; in particular, $\mathcal M$ is movable.
Moreover, $\det\mathscr G'\simeq(\det\mathscr Q)^\vee$ on $U$, hence on $Y$ by reflexivity.
By \cite[Definition~5.1 and Lemma~5.2]{ji_structure_2021}, the divisor $\mathfrak M$ of Theorem~\ref{th:geonon-normal} is a member of $\mathcal M$, so
\[
  \mathcal M\subseteq|\mathfrak M|,\qquad
  \mathcal O_Y(\mathfrak M)\simeq(\det\mathscr G')^\vee\simeq\det\mathscr Q.
\]
In particular, $\gamma_0$ is a morphism along any curve $C\subseteq Y$ that is disjoint from some member of $\mathcal M$.
Since $\mathfrak F$ is the zero divisor of $\det(\iota)$, the isomorphism $\mathcal O_Y(\mathfrak M)\simeq(\det\mathscr G')^\vee$ gives
\begin{equation}\label{eq:foliation-chern-class}
  c_1(\mathscr G')=-\mathfrak M,\qquad
  c_1(\mathscr G)=c_1(\mathscr G')-\mathfrak F=-\mathfrak C.
\end{equation}

\begin{lemma}\label{lemm:foliation-evaluation}
  In the above setting with $\ell=k^{1/p}$, let $a\in\ell\setminus k$, and let $D$ be an effective Weil divisor on $Y$.
  Assume that $v_P(\partial(a))\geq\operatorname{mult}_PD$ for every prime divisor $P\subseteq\operatorname{Supp}D$ and every $\partial$ in the stalk of $\mathscr G$ at the generic point of $P$.
  Then $\mathfrak C\sim D+E$ for some effective divisor $E$ on $Y$.
\end{lemma}
\begin{proof}
  Since $a\in\ell\subseteq H^0(Y,\mathcal O_Y)$ and local sections of $\mathscr G$ are derivations of $\mathcal O_Y$, the function $\partial(a)$ is regular for every local section $\partial$ of $\mathscr G$.
  The hypothesis, normality of $Y$, and reflexivity of $\mathcal O_Y(-D)$ therefore give a map \cite[Tag~\href{https://stacks.math.columbia.edu/tag/0AVC}{0AVC}]{StacksProject}
  \[
    \operatorname{ev}_a\colon\mathscr G\longrightarrow\mathcal O_Y(-D),
    \qquad \partial\longmapsto\partial(a).
  \]
  This map is nonzero since $a\notin K(X)$.
  Extend $a$ to a $p$-basis $a_1=a,a_2,\ldots,a_r$ of $\ell/k$.
  For $2\leq i\leq r$, define
  \[
    \operatorname{ev}_i\colon\mathscr G\longrightarrow\mathcal O_Y,
    \qquad \partial\longmapsto\partial(a_i).
  \]
  The inclusion $\mathscr G\hookrightarrow\operatorname{Der}_k(\ell)\otimes_\ell\mathcal O_Y$ implies that $\operatorname{ev}_a,\operatorname{ev}_2,\ldots,\operatorname{ev}_r$ span the dual of $\mathscr G$ at the generic point.
  After renumbering $a_2,\ldots,a_r$, we may therefore assume that $\operatorname{ev}_a,\operatorname{ev}_2,\ldots,\operatorname{ev}_q$ are linearly independent there.
  Taking the exterior product $\operatorname{ev}_a\wedge\operatorname{ev}_2\wedge\cdots\wedge\operatorname{ev}_q$ and reflexive hulls, and using \eqref{eq:foliation-chern-class}, gives a nonzero map
  \[
    \mathcal O_Y(-\mathfrak C)\simeq\det\mathscr G
      \longrightarrow\mathcal O_Y(-D).
  \]
  Hence $\mathfrak C\sim D+E$ for some effective divisor $E$.
\end{proof}

\subsection{Del Pezzo surfaces and Mori fiber spaces}\label{sec:del-pezzo-mori-fiber-spaces}%
\begin{definition}
  Let $X$ be a regular projective $k$-surface with $H^0(X,\mathcal O_X) = k$.
  We say that $X$ is a \textit{del Pezzo surface} if $-K_X$ is ample.
  We say that a fibration $f\colon X\to B$ from $X$ is a \textit{Mori fiber space} if $\dim B< \dim X$, $\rho(X/B) = 1$ and $- K_X$ is relatively ample over $B$.
\end{definition}

In this paper, we need the following results.
\begin{proposition}[{\cite[Theorem~4.1]{Bernasconi-Tanaka-22}}]\label{prop-ell_F-del-Pezzo}%
  Let $X$ be a normal projective surface over $k$ such that $X$ has at most canonical singularity and $-K_X$ is nef and big. 
  Let $L\in \mathrm{Pic}(X)$ such that $L \equiv 0$.
  Then:
  \begin{itemize}
    \item[\textup{(1)}] if $p\geq5$, then $L\sim0$;
    \item[\textup{(2)}] if $p=3$, then $3L\sim0$;
    \item[\textup{(3)}] if $p=2$, then $2^2L\sim0$.
  \end{itemize}
\end{proposition}

\begin{proposition}[see {\cite[Proposition~2.18]{Bernasconi-Tanaka-22}}]\label{prop:Mori-f-BT22}%
  Let $f\colon X\to B$ be a $K_X$-Mori fibre space from a regular projective  surface to a curve with $H^0(B,\mathcal O_B) = k$.
  Let $b\in B$ be a (not necessarily closed) point.
  Then the fibre $X_b$ is an integral $\kappa(b)$-conic with $\kappa(b) = H^0(X_b, \mathcal O_{X_b})$.
\end{proposition}

\begin{definition}\label{def:quasilinear}
  Let $K$ be a field of characteristic two.
  A conic in $\mathbb P^2_K$ is \textit{quasilinear} if it is defined by an equation $sx^2+ty^2+rz^2=0$ with $s,t,r\in K$.
\end{definition}
By \cite[Theorem~9.10\,(5),\,(6)]{tanaka_invariants_2021}, a regular conic $F$ over a field $K$ with $H^0(F,\mathcal O_F)=K$ is either smooth over $K$, or $\operatorname{char}K=2$ and $F$ is quasilinear.

\subsection{Minimal model program of surfaces}\label{sec:mmp}%
\begin{definition}
  Let $X$ be a regular projective surface over $k$.
  We call an integral curve $C \subset X$ a \emph{$(-1)$-curve} if
  \[
    (K_X\cdot C)_{k_C} = (C)^2_{k_C} = -1
  \]
  where $k_C = H^0(C, \mathcal O_C)$.
  Note that, for such a curve $C$, we have $H^1(C,\mathcal O_C) = 0$.
\end{definition}

\begin{theorem}[Castelnuovo's contraction criterion (special case of {\cite[Theorem~27.1]{Lipman69}})]%
  Let $C \subset X$ be a $(-1)$-curve on a regular projective surface $X$.
  Then there exists a birational contraction $f\colon X\to Z$ to a regular surface $Z$ such that $C$ is the exceptional curve.
\end{theorem}

\begin{theorem}[minimal model program {(\cite[page 97]{Shafarevich-1966-MMP}, \cite[page 5]{Tanaka-2018-MMP-excelent-surface})}]\label{thm:MMP}%
  Let $X$ be a regular projective surface over $k$.
  Then every sequence of contractions of $(-1)$-curves terminates at a regular surface $Z$ such that either $K_Z$ is nef or $Z$ admits a Mori fiber space structure.
\end{theorem}

\begin{lemma}\label{lemm:height-n}
  Let $\pi\colon X\to Y$ be a purely inseparable morphism of height $n$ between normal surfaces, and assume that $X$ has rational singularities.
  Then every exceptional curve $E$ of a resolution $\mu\colon \widetilde Y\to Y$ admits a purely inseparable cover $C\to E$ of height $n$ such that $H^1(C,\mathcal O_C) = 0$.
\end{lemma}
\begin{proof}
  Let $W$ be the normalization of $\widetilde Y$ in $K(X)$.
  The induced morphisms fit into a commutative diagram
  \[
    \begin{tikzcd}[column sep=large,row sep=small]
      W \ar[r,"\rho"] \ar[d,"\nu"'] & \widetilde Y \ar[d,"\mu"] \\
      X \ar[r,"\pi"] & Y\rlap{\,,}
    \end{tikzcd}
  \]
  where $\rho$ is purely inseparable of height $n$ and $\nu$ is proper birational.
  Let $y=\mu(E)$, let $x$ be the unique point of $X$ over $y$, and set $F=\nu^{-1}(x)$, $C=\bigl(\rho^{-1}(E)\bigr)_{\mathrm{red}}$.
  Thus, $C$ is an irreducible component of $F$, and $C\to E$ is a purely inseparable cover of height $n$.
  Since $X$ has rational singularities, $H^1(F,\mathcal O_{F}) = 0$ \cite[Proposition~1.2]{Lipman69}.
  Since $F$ is one-dimensional, the surjection $\mathcal O_F\twoheadrightarrow\mathcal O_C$ induces a surjection on $H^1$, so $H^1(C,\mathcal O_C)=0$.
\end{proof}

\begin{construction}[MMP of the base change]\label{cons:MMP}
Let $X$ be a regular projective $k$-surface with $H^0(X,\mathcal O_X)=k$, let $\ell$ be $k^{1/p}$ or $\bar k$, and let $\phi\colon Y:=(X\times_k\ell)_{\mathrm{red}}^\nu\to X$ be the natural morphism.
Let $\mu\colon \widetilde Y\to Y$ be a minimal resolution.
By Theorem~\ref{thm:MMP}, any sequence
\[
    \widetilde Y=:Z_0 \buildrel\epsilon_0\over\to Z_1 \buildrel\epsilon_1\over\to Z_2 \to\cdots\to Z_n
\]
of contractions of $(-1)$-curves terminates at $Z:=Z_n$, where
\begin{itemize}
  \item either $K_Z$ is nef; 
  \item or $Z$ admits a Mori fiber space structure $g\colon Z\to B$.
\end{itemize}
Let $\epsilon\colon\widetilde Y\to Z$ be the composite contraction and, in the second case, set $f:=g\circ\epsilon$.
In the Mori fiber space case, the morphisms fit into the diagram
\begin{equation}\label{eq:MMP-diagram-prelim}
  \begin{tikzcd}
    & \widetilde Y \ar[d,"\mu"'] \ar[r,"\epsilon"]
      \ar[dr,"f"'] & Z \ar[d,"g"] \\
    X & Y \ar[l,"\phi"'] & B
  \end{tikzcd}
\end{equation}
By Theorem~\ref{th:geonon-normal}, we have
\[
  K_Y+(p-1)\mathfrak C=\phi^*K_X,
\]
where $\mathfrak C=0$ if and only if $X$ is geometrically normal.
Then
\begin{equation}\label{eq:JJXM}
  \mu^*\phi^*K_X=K_{\widetilde Y}+(p-1)\widetilde{\mathfrak C}+\sum m_iE_i,
\end{equation}
where $\widetilde{\mathfrak C}$ is the strict transform of $\mathfrak C$ and the $E_i$ are the $\mu$-exceptional curves.
The coefficients $m_i$ are nonnegative integers: nonnegativity follows from the minimality of $\mu$ and the negativity lemma, while integrality follows because $K_X$ is Cartier.
\end{construction}
 
\begin{lemma}\label{lemm:descent-div}%
  Consider a commutative diagram
  \[
    \begin{tikzcd}[column sep=large,row sep=small]
       Y \arrow[r,"\epsilon"] \arrow[dr,"f"'] & Z \arrow[d,"g"] \\
      & B,
    \end{tikzcd}
  \]
  where $ Y$ and $Z$ are regular projective $k$-surfaces, $\epsilon$ is a composition of contractions of $(-1)$-curves, $g$ is a Mori fiber space, and $f=g\circ\epsilon$.
  Let $M\in\operatorname{Pic}( Y)$ satisfy $M\equiv0$.
  Then the following statements hold.
  \begin{itemize}
    \item[\textup{(1)}] If $\dim B=0$, then there exists $M_Z\in \Pic(Z)$ such that $M\cong\epsilon^*M_Z$.
    \item[\textup{(2)}] If $\dim B=1$, then there exists $M_B\in \Pic(B)$ such that $M\cong f^*M_B$.
  \end{itemize}
\end{lemma}
\begin{proof}
  Both assertions follow by applying \cite[Theorem~4.4\textup{(3)}]{Tanaka-2018-MMP-excelent-surface} successively to the contractions in $\epsilon$ and, when $\dim B=1$, to $g$.
\end{proof}

\subsection{Albanese variety over an arbitrary field}\label{sec:Alb}%
The existence of the Albanese morphism over an algebraically closed field was established by Serre \cite{Serre59-Alb}, and later by Grothendieck using a more advanced approach (cf.~\cite[\S9.5]{FGAexplained}).
This was extended to geometrically integral varieties over arbitrary fields in \cite[Appendix~A]{Wittenberg-2008}; see also \cite{Conrad-Alb} and \cite[\S1]{Laurent-2022}.
More recently, \cite{Laurent-Schroer-2024}, \cite{schroer_albanese_2024} and \cite[\S2.2]{AC-MV-25} established a more general result that applies to geometrically non-reduced varieties.
\begin{proposition}[special case of {\cite[Theorem~10.2]{Laurent-Schroer-2024}} or {\cite[Theorem~5.4]{schroer_albanese_2024}}]
  Let $X$ be a proper variety over $k$ with $H^0(X,\mathcal O_X) = k$.
  Then there exists a {\em para-abelian variety} $\Alb_{X/k}$ over $k$ (by which we mean that $\Alb_{X/k} \times_k \bar k$ admits an abelian variety structure) and a morphism $\alb_{X/k}\colon X\to \Alb_{X/k}$ verifying the following universal property:
  \begin{itemize}
    \item for every para-abelian variety $A$ and every morphism $g\colon X\to A$, there exists a unique morphism $\phi\colon \Alb_{X/k}\to A$ such that $g = \phi \circ \alb_{X/k}$.
  \end{itemize}
\end{proposition}

Every para-abelian variety carries a canonical simply transitive action of an abelian variety.
If $k$ is separably closed, then $\Alb_{X/k}$ has a $k$-point, hence it is an abelian variety.

\begin{lemma}\label{prop:Alb-eq-dim}%
  Let $X$ be a proper variety over $k$ with $H^0(X,\mathcal O_X) = k$.
  \begin{itemize}
    \item[\textup{(1)}] If $X$ is geometrically connected and geometrically reduced, 
      then for any field extension $k \subset k'$, there is a natural isomorphism $\Alb_{X_{k'}/k'} \buildrel\sim\over\to \Alb_{X/k} \otimes_k k'$.
    \item[\textup{(2)}] Let $Y := (X\times_k k^{1/p})_{\mathrm{red}}^\nu$. Then $\dim(\Alb_{Y/k^{1/p}}) = \dim(\Alb_{X/k})$.
    \item[\textup{(3)}] Let $Y := (X\times_k \bar k)_{\mathrm{red}}^\nu$.  Then $\dim\Alb_{Y/\bar k} = \dim\Alb_{X/k}$.
  \end{itemize}
\end{lemma}
\begin{proof}
  (1) See \cite[Theorem~3.2]{AC-MV-25}.

  (2)
  Set $k_1 = k^{1/p}$.
  We have $H^0(Y,\mathcal O_{Y}) = k_1$, and thus the Albanese variety $\Alb_{Y/k_1}$ exists.
  Consider the following commutative diagram \[
    \begin{tikzcd}[row sep = large]
      X^{\frac1p} \ar[rrrr,bend left=15,"F_X"]\rar{g}\dar{\alb_{X^{\frac1p}/k_1}} & Y \rar{\nu}\dar{\alb_{Y/k_1}} & X_{k_1} \ar[rr]\dar{\alb_{X_{k_1}/k_1}}\ar[dr,"(\alb_{X/k})_{k_1}"] && X \dar \\
      \Alb_{X^{\frac1p}/k_1}\rar{g'} & \Alb_{Y/k_1} \rar{\nu'} & \Alb_{X_{k_1}/k_1} \rar{u} & \Alb_{X/k} \otimes_k k_1 \ar[r] & \Alb_{X/k}
    \end{tikzcd}
  \]
  where $X^{\frac1p}$ denotes $X$ regarded as a $k_1$-scheme through the isomorphism $k_1\to k$, $a\mapsto a^p$, the morphisms $g'$ and $\nu'$ are induced by the functoriality of $\Alb$ (\cite[Corollary~10.4]{Laurent-Schroer-2024}), and $u$ is induced by the universal property of $\Alb_{X_{k_1}/k_1}$.
  By the uniqueness of $\Alb$ (\cite[Corollary~10.4]{Laurent-Schroer-2024}), the composition of the second line in the above diagram is the absolute Frobenius of $\Alb_{X/k}$, which is a homeomorphism.
  Hence $g'$ has finite fibers.
  Since $g$ is surjective and $\alb_{Y/k_1}\circ g=g'\circ\alb_{X^{\frac1p}/k_1}$, the image of $g'$ contains $\alb_{Y/k_1}(Y)$, which generates $\Alb_{Y/k_1}$, so $g'$ is surjective.
  Therefore $\dim\Alb_{Y/k_1}=\dim\Alb_{X^{\frac1p}/k_1}=\dim\Alb_{X/k}$.

  (3) Since there exists a finite purely inseparable field extension $k \subset k'\subseteq\bar k$ such that $(X \times_k k')_{\mathrm{red}}^\nu$ is geometrically integral, the result of (3) follows from (1) and (2).
\end{proof}

Over an algebraically closed field, the following result was proved for smooth varieties in \cite[Theorem~0.3]{HPZ19} and extended to normal varieties in \cite[Proposition~2.3]{EP23}.
\begin{proposition}\label{prop:GV0U}%
  Let $X$ be a normal projective variety over a field $k$ such that $H^0(X,\mathcal O_X) = k$ and $-K_X$ is nef.
  If $X$ is of maximal Albanese dimension (i.e., $\dim X = \dim\alb_{X/k}(X)$), then $\alb_{X/k}\colon X\to \Alb_{X/k}$ is an isomorphism.
\end{proposition}
\begin{proof}
  Let $Y:= (X \times_k \bar k)^\nu_{\mathrm{red}}$ and let $\phi\colon Y\to X$ be the induced morphism.
  Then $K_Y + (p-1) \mathfrak C = \phi^*K_X$ for some effective divisor $\mathfrak C$, where $\mathfrak C=0$ if and only if $X$ is geometrically normal.
  By Lemma~\ref{prop:Alb-eq-dim}, $Y$ is of maximal Albanese dimension.
  Thus $\kappa(Y,K_Y) \ge 0$ (see \cite[Proposition~2.9 (1)]{CWZ-cbf}).
  Since $-K_X$ is nef, this implies that $\mathfrak C = 0$ and $K_Y \equiv \phi^*K_X \equiv 0$.
  Consequently, $X$ is geometrically normal and $Y \cong X\times_k \bar k$.
  Since $Y$ is of maximal Albanese dimension and $K_Y \equiv 0$, it follows from \cite[Proposition~2.3]{EP23} that $Y$ is isomorphic to an abelian variety.
  Thus, $X$ is a para-abelian variety which is equivalent to the assertion.
\end{proof}

\begin{lemma}\label{lemm:sur-second-fibration}
  Let $X$ be a regular projective surface over a field $k$ with $H^0(X,\mathcal O_X)=k$ and $-K_X$ nef.
  Suppose that $A:=\Alb_{X/k}$ is an elliptic curve, and write $\alb:=\alb_{X/k}$.
  Let $h\colon X\to B$ be a fibration onto a regular projective curve whose generic fiber $\Gamma$ satisfies $p_a(\Gamma)=1$, and such that $\alb$ does not factor through $h$.
  \[
  \begin{tikzcd}
    X \rar{h} \dar{\alb} & B \\
    A
  \end{tikzcd}
  \]
  Then the following statements hold.
  \begin{enumerate}[label=\textup{(\arabic*)}]
    \item The Albanese morphism $\alb$ is a fibration.
    \item After choosing an origin on $A$ and setting $C:=\alb^{-1}(0)$, there exist an integer $n\geq1$ and an action of $A[n]$ on $C$. Let $A[n]$ act on $A$ by translations and on $A\times_k C$ diagonally. Then
      \[
        X\simeq(A\times_k C)/A[n],
      \]
  \end{enumerate}
\end{lemma}
\begin{proof}
  We adapt the argument of \cite[\S3, pp.~34--35]{BM-II}.

  Let $K= K(B)$, $J = \Pic^0_{\Gamma/K}$, and $\alpha = \alb|_{\Gamma}$.
  The morphism $\alpha$ is nonconstant, since otherwise its image would be a $K$-point of $A_K$, giving a rational map $B\dashrightarrow A$ that extends to a morphism $\beta$ with $\alb=\beta\circ h$.
  Hence $\alpha$ is dominant.
  The generic fiber $\Gamma$ is regular, has $H^0(\Gamma,\mathcal O_\Gamma)=K$ and $K_\Gamma\sim0$, and dominates the elliptic curve $A_K$.
  Proposition~\ref{prop:GV0U} therefore shows that $\Gamma$ is a smooth genus-one curve, with its natural simply transitive $J$-action $\mu\colon J\times_K\Gamma\longrightarrow\Gamma$.
  After a finite Galois extension $L/K$, choose $x_0\in\Gamma(L)$.
  The identification $J_L\simeq\Gamma_L$ given by $j\mapsto\mu(j,x_0)$ defines a morphism
  \[
    u_L\colon J_L\longrightarrow A_L,\qquad
    u_L(j):=\alpha_L(\mu(j,x_0))-\alpha_L(x_0).
  \]
  Since \(u_L(0)=0\), \(u_L\) is a homomorphism of abelian varieties.
  Replacing $x_0$ by $\mu(j_0,x_0)$ changes the right-hand side to $u_L(j + j_0) - u_L(j_0) = u_L(j)$.
  Thus $u_L$ is independent of $x_0$, hence Galois-invariant, and descends to a homomorphism
  \[
    u\colon J\to A_K .
  \]
  Since $\alpha$ is dominant, $u$ is nonzero and hence an isogeny.
  The construction gives \[
    \alpha(\mu(j,x))=u(j)+\alpha(x)
  \]
  for every $K$-scheme $S$ and all $j\in J(S)$ and $x\in\Gamma(S)$.
  For $n:=\deg u$, the dual isogeny $v\colon A_K\to J$ satisfies $uv=[n]$.
  Define the $A_K$-action on $\Gamma$ by
  \begin{equation}\label{eq:action}
    t\cdot x:=\mu(v(t),x)
  .
  \end{equation}
  Then
  \begin{equation}\label{eq:translation}
    \alpha(t\cdot x)=\alpha(\mu(v(t),x))
    =u(v(t))+\alpha(x)=\alpha(x) + [n]t.
  \end{equation}

  Every fiber $F_b:=h^*b$ satisfies $\chi_{\kappa(b)}(F_b,\mathcal O_{F_b})=0$ by flatness.
  The genus formula then gives $K_X\cdot F_b=0$.
  Since $-K_X$ is nef, every component of $F_b$ has nonpositive intersection with $K_X$, hence each intersection is zero.
  Thus $K_{X/B}$ is $h$-nef.
  By \cite[Theorem~7]{Kollar-2025-birational-actions}, the $A_K$-action (\ref{eq:action}) extends to an action of $A\times_kB$ on $X$ over $B$, equivalently an action
  \[
    A\times_kX\longrightarrow X,\quad (t,x)\longmapsto t\cdot x, \quad\text{satisfying  }h(t\cdot x)=h(x).
  \]
  Equation~(\ref{eq:translation}) then extends to 
  \begin{equation}\label{eq:second-fibration-equivariance}
    \alb(t\cdot x)=\alb(x)+[n]t,
  \end{equation}
  since each side is a morphism $A\times_kX\to A$ and they are equal on a dense subset.

  Set $C:=\alb^{-1}(0)$ and form the Cartesian diagram
  \[
    \begin{tikzcd}
      X':=A\times_{[n],A,\alb}X \ar[r,"\pi"] \ar[d,"\alb'"']
        & X \ar[d,"\alb"]\\
      A \ar[r,"{[n]}"'] & A.
    \end{tikzcd}
  \]
  Equation~\eqref{eq:second-fibration-equivariance} gives an isomorphism
  \begin{equation}\label{eq:second-fibration-product}
    A\times_kC\xrightarrow{\ \sim\ }X',
    \qquad (t,z)\longmapsto(t,t\cdot z),
  \end{equation}
  with inverse $(t,x)\mapsto(t,(-t)\cdot x)$.
  To see that this inverse is well-defined, note that $(t,x)\in X'$ satisfies $\alb(x)=[n]t$, so \eqref{eq:second-fibration-equivariance} gives
  \[
    \alb((-t)\cdot x)=\alb(x)-[n]t=0,
  \]
  and hence $(-t)\cdot x\in C$.
  Write $X\xrightarrow{f}B_0\xrightarrow{g}A$ for the Stein factorization of $\alb$.
  Then $B_0$ is a regular curve with $H^0(B_0,\mathcal O_{B_0})=k$, and $g$ is finite flat.
  By flat base change, the Stein factorization of $\alb'$ is
  \[
    X'\longrightarrow B':=A\times_{[n],A,g}B_0\longrightarrow A.
  \]
  Via \eqref{eq:second-fibration-product}, $\alb'$ is the projection
  $A\times_kC\to A$, whose Stein factorization is
  \[
    A\times_kC\longrightarrow A\times_kT\longrightarrow A,
    \qquad
    T:=g^{-1}(0)\simeq\operatorname{Spec}H^0(C,\mathcal O_C).
  \]
  Hence $B'\simeq A\times_kT$.
  The scheme $T$ is a zero-dimensional Cartier divisor on the regular curve $B_0$, hence it is Gorenstein.
  Its dualizing sheaf is therefore invertible, and hence trivial because $T$ is finite over $k$.
  Hence $\omega_{B'/k}\simeq\mathcal O_{B'}$.
  Let $\pi_0\colon B'\to B_0$ be the projection.
  The isogeny $[n]\colon A\to A$ has trivial relative dualizing sheaf, so flat base change and transitivity of duality give
  \[
    \pi_0^*\omega_{B_0/k}\simeq\omega_{B'/k}\simeq\mathcal O_{B'}.
  \]
  Taking the norm along the finite locally free morphism $\pi_0$, of degree $n^2$, yields $\omega_{B_0/k}^{\otimes n^2}\simeq\mathcal O_{B_0}$, hence $p_a(B_0)=1$.
  Since $g\colon B_0\to A$ is finite and dominant, Proposition~\ref{prop:GV0U} shows that $B_0$ is a smooth curve of genus one.
  The Albanese universal property shows that $g$ is an isomorphism, so $\alb$ is a fibration.

  Finally, $\pi$ is an $A[n]$-torsor, being the base change of $[n]$.
  Under \eqref{eq:second-fibration-product}, its action becomes
  \[
    s\cdot(t,z)=(t+s,(-s)\cdot z),
    \qquad s\in A[n].
  \]
  This is a free diagonal action with quotient $X$.
\end{proof}

\section{The geometrically normal case}\label{sec:caseA}
\begin{proposition}\label{prop:geo-normal-alb}%
  Let $X$ be a regular projective surface over a field $k$ with $H^0(X,\mathcal O_X)=k$ and $K_X\equiv0$.
  Assume that $X$ is geometrically normal, and let $\mu\colon\widetilde Y\to Y:=X_{\bar k}$ be the minimal resolution.
  Then either $\widetilde Y$ is rational, or $\mu$ is crepant.
\end{proposition}
\begin{proof}
  We may assume that $k$ is separably closed.
  Write
  \[
  \mu^*K_Y=K_{\widetilde Y}+\sum_i m_iE_i\equiv0, \quad\text{where the $E_i$ are the $\mu$-exceptional curves and } m_i\ge 0.
  \]
  Recall that $\mu$ is \emph{crepant} if $K_{\widetilde Y}=\mu^*K_Y$, that is, if $m_i=0$ for all $i$.
  Suppose that $\mu$ is not crepant, so that some $m_i>0$.
  Then $K_{\widetilde Y}$ is not pseudo-effective, so the MMP ends with a Mori fiber space $g\colon Z\to B$.
  The resulting morphisms fit into the diagram as in Construction~\ref{cons:MMP}:
  \[
    \begin{tikzcd}
      & \widetilde Y \ar[d,"\mu"'] \ar[r,"\epsilon"]
        \ar[dr,"f"'] & Z \ar[d,"g"] \\
      X & Y \ar[l,"\phi"'] & B \rlap{\,.}
    \end{tikzcd}
  \]
  If $\dim B=0$, then $Z$ is a smooth del Pezzo surface and hence rational, so $\widetilde Y$ is rational too.
  Suppose now that $\dim B=1$.
  After descending the resolution to $k^{1/p^n}$ for some $n$, we apply Lemma~\ref{lemm:height-n} to the relative Frobenius $X^{1/p^n}\to X_{k^{1/p^n}}$ to conclude that every $E_i$ is rational.
  If $p_a(B)>0$, this implies that every $E_i$ is vertical over $B$, so pushing forward $K_{\widetilde Y}+\sum_i m_iE_i\equiv0$ along $\epsilon$ gives $K_Z+V\equiv0$ for some vertical divisor $V$, which contradicts the relative ampleness of $-K_Z$ over $B$.
  Therefore $p_a(B)=0$, so $\widetilde Y$ is rational.
\end{proof}

\begin{proposition}\label{prop:geo-normal}
  Let $X$ be a regular projective surface over $k$ with $H^0(X,\mathcal O_X)=k$ and $K_X \equiv 0$.
  If $X$ is geometrically normal, then $12 K_X \sim 0$.
\end{proposition}
\begin{proof}
  By assumption, the base change $Y := X \times_k \bar k$ is a normal surface.
  Since the pullback $\operatorname{Pic}(X)\to\operatorname{Pic}(Y)$ is injective by \cite[Tag~\href{https://stacks.math.columbia.edu/tag/0CC5}{0CC5}]{StacksProject}, it is enough to show that $12 K_Y \sim 0$.
  Let $\mu\colon\widetilde Y\to Y$ be the minimal resolution.
  By Proposition~\ref{prop:geo-normal-alb}, either $\mu$ is crepant or $\widetilde Y$ is rational.
  If $\mu$ is crepant, then $K_{\widetilde Y} \equiv 0$, and it follows that $12 K_{\widetilde Y} \sim 0$ (\cite{BM-II,BM-III}).
  In turn, we have $12 K_Y \sim 0$.
  If $\widetilde Y$ is rational, then $\mu^*K_Y\sim 0$, since a smooth rational surface has no nontrivial numerically trivial line bundles.
  The equality $\mu_*\mathcal O_{\widetilde Y}=\mathcal O_Y$ and the projection formula then give $K_Y\sim 0$.
\end{proof}

\section{Albanese fibrations and effective nonvanishing}\label{sec:caseB}\label{sec:caseB-proof}%
The main result of this section is the following proposition, whose proof occupies the rest of the section.

\begin{proposition}\label{prop:a_X-surjective}%
  Let $X$ be a regular, geometrically non-normal, projective surface over $k$ such that $H^0(X,\mathcal O_X) = k$, $K_X \equiv 0$ and $\dim \Alb_{X/k} > 0$.
  Then:
  \begin{itemize}
    \item[\textup{(1)}] $p=2$ or $3$;
    \item[\textup{(2)}] we have $6K_X \sim 0$ if $p=3$, and $12K_X \sim 0$ if $p=2$;
    \item[\textup{(3)}] $\dim \Alb_{X/k} = 1$, and the Albanese morphism $\alb_{X/k}\colon X\to\Alb_{X/k}$ is a fibration; and
    \item[\textup{(4)}] after a finite separable field extension $k'/k$ and a choice of origin $0$ on $E:=(\Alb_{X/k})_{k'}$, there exists an integer $n\geq1$ such that
      \[
        X_{k'}\simeq(E\times_{k'}C)/E[n],
        \qquad C:=(\alb_{X/k})_{k'}^{-1}(0),
      \]
      where $E[n]$ acts on $E$ by translations and on $E\times_{k'}C$ diagonally.
  \end{itemize}
\end{proposition}

We now begin the proof. 
First, we may assume, without loss of generality, that $k$ is $F$-finite and separably closed.
Set $k_1:=k^{1/p}$ and $X_1:=(X\times_k k_1)_{\mathrm{red}}^\nu$, and let $\phi\colon X_1\to X$ be the natural morphism.
Since $X$ is geometrically non-normal, it is not para-abelian, so Proposition~\ref{prop:GV0U} gives $\dim\alb_{X/k}(X)\ne2$.
Hence $\dim\alb_{X/k}(X)=1$.
By \cite[Theorem~1.1]{tanaka_abundance_2020}, we have $K_X\simQ0$.
By Theorem~\ref{th:geonon-normal}, there is an effective divisor $\mathfrak C>0$ such that
\begin{equation}\label{eq:canonical-caseB}
  \phi^*K_X=K_{X_1}+(p-1)\mathfrak C.
\end{equation}
Thus $K_{X_1}$ is not pseudo-effective.
Therefore, the MMP on the minimal resolution $\widetilde X_1$ in Construction~\ref{cons:MMP} terminates in a Mori fiber space
\begin{equation}\label{eq:mmp-alb}
  \begin{tikzcd}[row sep=tiny]
    && \widetilde X_1 \ar[ld,"\mu"'] \ar[rd,"\epsilon"]\\
    X & X_1 \ar[l,"\phi"'] && Z \ar[d,"g"]\\[2ex]
      &&& B_1\rlap{\,.}
  \end{tikzcd}
\end{equation}
Note that $H^0(X_1,\mathcal O_{X_1})=H^0(\widetilde X_1,\mathcal O_{\widetilde X_1})=H^0(Z,\mathcal O_Z)=H^0(B_1,\mathcal O_{B_1})=k_1$. 
By Lemma~\ref{lemm:height-n} and Proposition~\ref{prop:Mori-f-BT22}, the exceptional curves of $\mu$ and $\epsilon$, as well as the fibers of $g$, have trivial Albanese image.
Hence the Albanese morphism of $X_1$ factors as
\[
  X_1\xrightarrow{f_1}B_1\longrightarrow\Alb_{X_1/k_1}.
\]
Since the Albanese morphism is nonconstant, it follows that $p_a(B_1)>0$.
Let $X\to B_0\to\Alb_{X/k}$ be the Stein factorization of $\alb_{X/k}$. Then we have the following commutative diagram
\[
  \begin{tikzcd}
    X_1 \ar[r,"f_1"] \ar[d,"\phi"'] & B_1 \ar[r] \ar[d] & \Alb_{X_1/k_1} \ar[d] \\
    X \ar[r] & B_0 \ar[r] & \Alb_{X/k}\rlap{\,.}
  \end{tikzcd}
\]

Let $F$ be the generic fiber of $f_1$, which is a regular projective curve over $k_F:=K(B_1)$ because $X_1$ is normal.
Since $F$ is birational over $k_F$ to the generic fiber of $g$, the two are isomorphic, so $F$ is an integral conic over $k_F$ by Proposition~\ref{prop:Mori-f-BT22}.
Restricting \eqref{eq:canonical-caseB} to $F$ and applying the adjunction formula, we obtain
\begin{equation}\label{eq:adjunction-F}
  0=\deg_{k_F}K_F+(p-1)\mathfrak C\cdot_{k_F}F =-2+(p-1)\mathfrak C\cdot_{k_F}F.
\end{equation}
Thus $p=2$ or $3$, which proves Proposition~\ref{prop:a_X-surjective}\,(1).  Moreover, one of the following holds:
\begin{itemize}
  \item[(a)] $\mathfrak C$ has a component $C_1$  satisfying
    $C_1\cdot_{k_F}F=1$;
  \item[(b)] $p=2$ and the horizontal part of $\mathfrak C$ is an
    irreducible curve $C_1$ satisfying $C_1\cdot_{k_F}F=2$.
\end{itemize}
In Case~(a), the induced morphism $\sigma\colon C_1\to B_1$ is birational and we write
\begin{equation}\label{eq:boundary-case-a}
  (p-1)\mathfrak C=C_1+C_1'+\mathfrak C_v,
\end{equation}
where $\mathfrak C_v$ is the vertical part and $C_1+C_1'$ is the horizontal part (possibly with
$C_1'=C_1$).  In Case~(b), the induced morphism $\sigma\colon C_1\to B_1$ is of degree two and we write
$\mathfrak C=C_1+\mathfrak C_v$.

Denote by $C_0$ the image of $C_1$ in $X$.
Since $\phi$ is a homeomorphism, $\phi^*C_0=\alpha C_1$ for some integer $\alpha\geq1$.

\begin{lemma}\label{lemm:C1^2}
  Under the notation above:
  \begin{itemize}
    \item[\textup{(i)}] $C_1$ is smooth over $k_{C_1}:=H^0(C_1,\mathcal O_{C_1})$ with $p_a(C_1)=1$, and $C_1\cdot\mathfrak C_v=0$.
      In Case~(a), $\sigma$ is an isomorphism and $k_{C_1}=k_1$.
    \item[\textup{(ii)}] $C_1^2=0$ and $C_1$ is semi-ample.
  \end{itemize}
\end{lemma}
\begin{proof}
  First, we claim that $C_1^2\geq0$. Otherwise $C_0^2<0$, and restricting $K_X+C_0\equiv C_0$ to $C_0^\nu$ and applying the adjunction formula would then give
  \[
    \deg K_{C_0^\nu}+\deg\operatorname{Diff}_{C_0^\nu}=C_0^2<0,
  \]
  hence $p_a(C_0^\nu)=0$, which is impossible because $C_0^\nu$ dominates the curve $B_0$ of positive geometric genus.

  Restricting \eqref{eq:canonical-caseB} to $C_1^\nu$ and applying the adjunction formula, we obtain
  \[
  K_{C_1^\nu}+C_1'|_{C_1^\nu}+\mathfrak C_v|_{C_1^\nu}
  +\operatorname{Diff}_{C_1^\nu}\simQ0 ~\text{in Case~(a)},
  \]
  and 
  \[
  K_{C_1^\nu}+\mathfrak C_v|_{C_1^\nu}
  +\operatorname{Diff}_{C_1^\nu}\simQ0  ~\text{in Case~(b)}.
  \]
  Since $C_1$ dominates $B_1$, it follows that $\deg K_{C_1^\nu}\geq0$. All the other terms in the above equalities have non-negative degrees. We conclude that 
  \[
    \deg K_{C_1^\nu}=0,\qquad \operatorname{Diff}_{C_1^\nu}=0,
    \qquad C_1\cdot\mathfrak C_v=0 \qquad\text{in either case},
  \]
  and $C_1'\cdot C_1=0$ in case (a).

  By Proposition~\ref{prop:adj}, $\operatorname{Diff}_{C_1^\nu}=0$ implies that $C_1$ is regular and that $X_1$ is
  regular along $C_1$. Moreover, we have $p_a(C_1)=1$ by $\deg K_{C_1^\nu}=0$.
  Since the composite $C_1\to B_1\to\Alb_{X_1/k_1}$ is nonconstant, it follows that $C_1$ is smooth over $k_{C_1}$.
  In Case~(a), since $C_1$ and $B_1$ are normal, the finite birational morphism $\sigma\colon C_1\to B_1$ is an isomorphism, so $k_{C_1}\simeq H^0(B_1,\mathcal O_{B_1})=k_1$, which completes the proof of (i).

  \smallskip
  Next, we prove (ii).
  We first deal with Case~(a). If $C_1'=C_1$, then we have shown $C_1^2=0$.
  Thus, the image $C_0$ of $C_1$ on $X$ also satisfies $C_0^2=0$. Since $K_X\simQ0$, by adjunction formula, we have 
  \[
   0\simQ (K_X+C_0)|_{C_0^\nu}\simQ K_{C_0^\nu}+\operatorname{Diff}_{C_0^\nu}.
  \]
  The curve $C_0^\nu$ has nontrivial Albanese dimension, as does $C_1$.  Thus $p_a(C_0^\nu)\ge 1$ and $\deg K_{C_0^\nu}\ge 0$.
  Thus $\operatorname{Diff}_{C_0^\nu}=0$, $C_0$ is regular, and
  $(K_X+C_0)|_{C_0}\sim K_{C_0}\sim0$.  Since $K_X\simQ0$, we have
  $C_0|_{C_0}\simQ0$ and $K_{X}\cdot C_0=0$. By Proposition~\ref{prop:tanaka-abun-3.1}, we see that $C_0$ is semi-ample.
  Since $\phi^*C_0 = \alpha C_1$, the curve $C_1$ is semi-ample as well.

  If $C_1'\neq C_1$, we have $C_1\cdot C_1'=0$. Moreover, we have shown that $C_1^2\geq0$ and similarly $C_1'^2\geq0$.
  The Hodge index theorem shows that either $C_1^2=0$ or $C_1'^2=0$.
  By the same argument as above, either $C_1$ or $C_1'$ is semi-ample.
  Then $C_1\cdot C_1'=0$ implies that  $C_1^2=C_1'^2=0$ and both $C_1$ and $C_1'$ are semi-ample. 
  
  \smallskip
  We now deal with Case~(b).
  As in Case~(a), it suffices to show that $C_1^2=0$.
  We take the base change
  \begin{equation}\label{eq:KN13}
    \begin{tikzcd}
      X_2:=X_1\times_{B_1}C_1 \ar[r,"\phi_2"] \ar[d,"f_2"']
        & X_1 \ar[d,"f_1"]\\
      B_2:=C_1 \ar[r,"\sigma"'] & B_1.
    \end{tikzcd}
  \end{equation}
  Since $K(B_1)$ is algebraically closed in
  $K(X_1)$, the fiber product is integral (if $\sigma$ is separable, this is clear; if $\sigma$ is inseparable, this follows from
  \cite[Lemma~3.3]{tanaka_behavior_2018}).  If $\sigma$ is separable, the pullback of $C_1$ is
  $C_2+C_2'$, and the induced morphism $C_2\to B_2$ is birational.

  Suppose first that $\sigma$ is separable.
  Then $k_{C_1}=k_1$ since $k_1$ is separably closed.
  The Riemann--Hurwitz formula (Lemma~\ref{lemma-canonical-curve-comparison}) gives $K_{C_1}\sim\sigma^*K_{B_1}+R$ with $R\geq0$.
  As $\deg K_{C_1}=0$ and $\deg K_{B_1}\geq0$, we get $R=0$ and $p_a(B_1)=1$.
  Hence $\sigma$ is an \'etale double cover and $K_{B_1}\sim0$.
  It follows that $\phi_2$ is \'etale, so $X_2$ is normal, and that
    \[
      K_{X_2}+C_2+C_2'+\phi_2^*\mathfrak C_v\sim\phi_2^*(K_{X_1}+C_1+\mathfrak C_v)\simQ0.
    \]
  Restricting to $C_2$ and applying the adjunction formula, we obtain \[
    K_{C_2}+C_2'|_{C_2}+\phi_2^*\mathfrak C_v|_{C_2}\simQ0.
  \]
  Since $C_2\simeq B_2=C_1$ has arithmetic genus one, $\deg K_{C_2}=0$. As $\phi_2^*\mathfrak C_v$ is $f_2$-vertical, we have $\phi_2^*\mathfrak C_v|_{C_2}\ge 0$. 
  Therefore, we have $C_2\cdot C_2'=0$.
  Let $\iota$ be the nontrivial automorphism of $C_1$ over $B_1$. Then $\mathrm{id}_{X_1}\times\iota$ is an automorphism of $X_2$ exchanging $C_2$ and $C_2'$, so $C_2^2=C_2'^2$.
  Since $2C_1^2=(\phi_2^*C_1)^2=(C_2+C_2')^2=2C_2^2$, we get $C_2^2=C_1^2\geq0$, and the Hodge index theorem gives $C_2^2=0$.
  Hence $C_1^2=0$.

  Suppose now that $\sigma$ is inseparable.
  Let $\nu\colon X_2^\nu\to X_2$ be the normalization.
  Since $\phi_2$ is flat, we have $K_{X_2}\sim\phi_2^*K_{X_1}+f_2^*(K_{C_1}-\sigma^*K_{B_1})$.
  Hence $X_2$ is Gorenstein in codimension one, and $K_{X_2^\nu}+\mathfrak D\sim\nu^*K_{X_2}$ for the conductor divisor $\mathfrak D\geq0$ of $\nu$.
  Write $\nu^*\phi_2^*C_1=mC_2$ and $\mathfrak D=dC_2+\mathfrak D'$ with $C_2\not\subset\Supp\mathfrak D'$.
  As $K_{C_1}\sim0$ and $K_{X_1}\simQ-C_1-\mathfrak C_v$, we get
  \begin{equation}\label{eq:caseB-insep}
    (K_{X_2^\nu}+C_2)+(m+d-1)C_2+\mathfrak D'+\nu^*\phi_2^*\mathfrak C_v+\nu^*f_2^*\sigma^*K_{B_1}\simQ0.
  \end{equation}
  The coefficient $m+d-1$ is positive, since $m\ge1$ and if $d=0$, then $X_2$ is normal at the generic point of $\nu(C_2)$, so $m\geq2$ because the scheme-theoretic preimage $X_2\times_{X_1}C_1=C_1\times_{B_1}C_1$ of $C_1$ is non-reduced.
  Restricting \eqref{eq:caseB-insep} to $C_2^\nu$ and applying the adjunction formula, we obtain
  \[
    K_{C_2^\nu}+\operatorname{Diff}_{C_2^\nu}+(m+d-1)C_2|_{C_2^\nu}+(\mathfrak D'+\nu^*\phi_2^*\mathfrak C_v+\nu^*f_2^*\sigma^*K_{B_1})|_{C_2^\nu}\simQ0.
  \]
  Every term has nonnegative degree, since $C_2^\nu$ has nontrivial Albanese morphism, $m^2C_2^2=2C_1^2\geq0$, $\nu^*\phi_2^*\mathfrak C_v$ is vertical over $B_2$ while $C_2$ is not, and $p_a(B_1)>0$.
  Hence $C_2^2=0$, so $C_1^2=0$.
\end{proof}

\begin{proposition}\label{prop:caseB-regularity}%
  Under the assumptions of Proposition~\ref{prop:a_X-surjective}, the surface $X_1 := (X \times_k k^{1/p})^\nu_{\mathrm{red}}$ is regular. Moreover, we have $\mathfrak C_v=0$.
\end{proposition}
\begin{proof}
  Let $h_1\colon X_1\to D_1$ be the fibration induced by the semi-ample curve $C_1$ from Lemma~\ref{lemm:C1^2}\,\textup{(ii)}.
  Let $G$ be an arbitrary fiber of $h_1$. Since $G\equiv rC_1$ for some $r>0$ and $C_1\cdot F>0$, we may write $G=nT+G'$, where $T$ is an $f_1$-horizontal prime divisor and $G'$ is an effective divisor such that $\Supp G'\not \supset T$.
  On the other hand, we may write $(p-1)\mathfrak C=bT+\mathfrak D$, with $T\not\subset\operatorname{Supp}\mathfrak D$ and $b\ge0$.

  Since $G|_{T^\nu}\sim 0$ and $\phi^*K_X\simQ 0$, we have \[0\simQ (\phi^*K_X+G)|_{T^\nu}=(K_{X_1}+(p-1)\mathfrak C+G)|_{T^\nu}=(K_{X_1}+T+(n+b-1)T+\mathfrak D+G')|_{T^\nu}.\] Applying the adjunction formula, we obtain
  \begin{equation}\label{eq:adj-X1-reg}
    0\simQ K_{T^\nu}+\operatorname{Diff}_{T^\nu} +(n+b-1)T|_{T^\nu}+(\mathfrak D+G')|_{T^\nu}.
  \end{equation}
  The argument proving $C_1^2\ge 0$ in Lemma~\ref{lemm:C1^2} also gives $T^2\geq0$.
  Therefore, all the terms in \eqref{eq:adj-X1-reg} have nonnegative degrees. We conclude that
  $\operatorname{Diff}_{T^\nu}=0$ and $G'\cdot T=0$.  Since the fiber $G$ is connected, it follows from $G'\cdot T=0$ that $G'=0$. Hence, every fiber of $h_1$ has irreducible support. By Proposition~\ref{prop:adj}, $\operatorname{Diff}_{T^\nu}=0$ implies that $X_1$ is regular along $T$.
  Since $G$ was an arbitrary fiber and the fibers of $h_1$ cover $X_1$, the surface $X_1$ is regular. 
  
  Finally, it remains to show that $\mathfrak C_v=0$. Suppose, to the contrary, that $\mathfrak C_v\neq 0$. 
  Since $C_1$ is semi-ample and its induced fibration $h_1$ has irreducible fibers, $C_1\cdot \mathfrak C_v=0$ implies that $\mathfrak C_v$ is numerically equivalent to $mC_1$ for some $m>0$. Then $\mathfrak C_v\cdot F=mC_1\cdot F>0$, contradicting $\mathfrak C_v \cdot F=0$. 
  
\end{proof}

\begin{lemma}\label{lemm:caseB-normal-bundle}%
  In the $p=3$ case, we have
  \[
    2C_1|_{C_1}\sim0.
  \]
\end{lemma}
\begin{proof}
	As $\mathfrak C_v=0$ by Proposition~\ref{prop:caseB-regularity}, we have
	\[
	\phi^*K_X=K_{X_1}+2C_1.
	\]
	By Lemma~\ref{lemm:C1^2}, the curve $C_1\simeq B_1$ is smooth over $k_1$ of arithmetic genus one.
	As $X_1$ is regular (Proposition~\ref{prop:caseB-regularity}), $C_1$ is an effective Cartier divisor, and Lemma~\ref{lemma-smooth-Cartier-divisor} shows that $X_1$ is smooth over $k_1$ along $C_1$.
	Moreover, $H^0(X_1,\mathcal O_{X_1})=k_1$ and $C_1^2=0$. 
	
  Put $N:=\mathcal O_{C_1}(C_1)$, so $N\in\operatorname{Pic}^0(C_1)$.
  If $C_1$ is movable, there is an effective divisor $D\sim C_1$ that does not contain $C_1$.
  The equality $C_1\cdot D=C_1^2=0$ shows that $C_1$ and $D$ are disjoint, and hence $N\simeq\mathcal O_{C_1}(D)\simeq\mathcal O_{C_1}$.
  We may therefore assume that $C_1$ is not movable.
  The decomposition $C_1 = \mathfrak F + \mathfrak M$ in Theorem~\ref{th:geonon-normal} then gives $\mathfrak M = 0$ and $\mathfrak F = C_1$, so $X\times_k k_1$ is reduced.
  Put $r := \dim_{k_1}\operatorname{Der}_k(k_1)$.
  Let $\mathscr G:=T_{X_1/X}$ be the foliation associated with the height-one purely inseparable morphism $\phi$, let $\mathscr G'$ be its saturation in $\operatorname{Der}_k(k_1)\otimes_{k_1}\mathcal O_{X_1}$, and denote the inclusion by $\iota\colon\mathscr G\hookrightarrow\mathscr G'$.
  Since $X\times_k k_1$ is reduced, Subsection~\ref{sec:geo-non-normal-base-change} (see \cite[Remark~3.3]{ji_structure_2021}) gives
  \[
    \mathscr G' = \operatorname{Der}_k(k_1)\otimes_{k_1}\mathcal O_{X_1}
      \simeq\mathcal O_{X_1}^{\oplus r},
  \]
  and the zero divisor of $\det(\iota)$ is $\mathfrak F = C_1$.

  Choose an open neighborhood $U$ of $C_1$ over which $X_1$ is smooth over $k_1$.
  Since $U$ is smooth over $k_1$, the cotangent sequence for $U\to\operatorname{Spec}k_1\to\operatorname{Spec}k$ is exact:
  \[
    0\longrightarrow\Omega_{k_1/k}\otimes_{k_1}\mathcal O_U
    \longrightarrow\Omega_{U/k}
    \longrightarrow\Omega_{U/k_1}
    \longrightarrow0.
  \]
  Since $\Omega_{U/k_1}$ is locally free, this sequence splits locally.
  Dualizing and using $\operatorname{Hom}_{k_1}(\Omega_{k_1/k},k_1)\cong\operatorname{Der}_k(k_1)$ and $\mathscr G'|_U = \operatorname{Der}_k(k_1)\otimes_{k_1}\mathcal O_U$, we obtain  the locally split exact sequence
  \begin{equation}\label{eq:tangent-field-extension}
    0\longrightarrow T_{X_1/k_1}|_U
    \longrightarrow T_{X_1/k}|_U
    \xrightarrow{\mathrm{res}}\mathscr G'|_U
    \longrightarrow0.
  \end{equation}
  The map $\mathrm{res}$ sends a local $k$-derivation $D$ to its restriction $D|_{k_1}$.
  The composite
  \[
    \mathscr G|_U\hookrightarrow T_{X_1/k}|_U
    \xrightarrow{\mathrm{res}}\mathscr G'|_U
  \]
  equals $\iota|_U$.
  Dualizing the cotangent sequence for $X_1\xrightarrow{\phi}X\to\operatorname{Spec}k$, we obtain an exact sequence \cite[Tag~\href{https://stacks.math.columbia.edu/tag/01UX}{01UX}]{StacksProject}
  \[
    0\longrightarrow\mathscr G
    \longrightarrow T_{X_1/k}
    \longrightarrow\mathcal{H}om_{\mathcal O_{X_1}}
      \bigl(\phi^*\Omega_{X/k},\mathcal O_{X_1}\bigr).
  \]
  Since $X_1$ is integral, the last sheaf is torsion-free.
  The kernel of a morphism into a torsion-free sheaf is saturated, since the quotient by the kernel embeds into the target.
  Hence, $\mathscr G$ is saturated in $T_{X_1/k}$.

  Since $X_1$ is a regular surface, the dual sheaf $\mathscr G = \mathcal{H}om_{\mathcal O_{X_1}}(\Omega_{X_1/X},\mathcal O_{X_1})$ is reflexive and hence locally free on $X_1$ \cite[Tags~\href{https://stacks.math.columbia.edu/tag/0AY4}{0AY4} and~\href{https://stacks.math.columbia.edu/tag/0B3N}{0B3N}]{StacksProject}.
  Fix $x\in C_1$, let $\mathfrak m\subset\mathcal O_{X_1,x}$ be the maximal ideal, and let $t$ be a local equation of $C_1$ at $x$.
  Since $\mathcal O_{X_1,x}$ and $\mathcal O_{X_1,x}/(t) = \mathcal O_{C_1,x}$ are regular local rings of dimensions differing by one, we have $t\notin\mathfrak m^2$.
  Let $M$ be an $r\times r$ matrix over $\mathcal O_{X_1,x}$ representing $\iota_x$. Then $\det M$ is a unit times $t$.
  If $r = 1$, then, after rescaling, $M=(t)$.
  If $r > 1$, then not all entries of $M$ lie in $\mathfrak m$ (otherwise, $\det M\in\mathfrak m^r\subseteq\mathfrak m^2$, a contradiction). 
  Using elementary row and column operations to clear a unit entry and inducting on $r$, we reduce $M$ to $\operatorname{diag}(1,\ldots,1,t)$.
  In either case, the cokernel of $\iota$ is a line bundle $A$ on $C_1$, which induces an exact sequence near $C_1$:
  \begin{equation}\label{eq:foliation-elementary-transform}
    0\longrightarrow\mathscr G
    \buildrel\iota\over\longrightarrow\mathscr G'
    \longrightarrow i_*A\longrightarrow0,
    \qquad i\colon C_1\hookrightarrow X_1.
  \end{equation}
  Consider the exact sequence
  \[
    0\longrightarrow\mathcal O_{X_1}(-C_1)
    \longrightarrow\mathcal O_{X_1}
    \longrightarrow i_*\mathcal O_{C_1}
    \longrightarrow0.
  \]
  Locally, the first map is the multiplication by a local equation of $C_1$, which annihilates $i_*A$.
  Tensoring this sequence with $i_*A$, we obtain an isomorphism of sheaves on $C_1$:
  \[
    \operatorname{Tor}_1^{\mathcal O_{X_1}}(\mathcal O_{C_1},i_*A)
    \simeq A\otimes\mathcal O_{C_1}(-C_1)
    =A\otimes N^{-1}.
  \]
  Tensoring \eqref{eq:foliation-elementary-transform} with $\mathcal O_{C_1}$ and using the fact that $\mathscr G'$ is locally free, we obtain
  \begin{equation}\label{eq:PGA}
    0 \longrightarrow \Tor_1^{\mathcal O_{X_1}} (i_*A, \mathcal O_{C_1}) \cong A\otimes N^{-1}
    \longrightarrow\mathscr G|_{C_1}
    \longrightarrow\mathscr G'|_{C_1}\simeq\mathcal O_{C_1}^{\oplus r}
    \longrightarrow A\longrightarrow0.
  \end{equation}
  In particular, $A$ is globally generated.
  Put $P:=A\otimes N^{-1}$.
  Since $\mathscr G$ is saturated in $T_{X_1/k}$, the restriction $\mathscr G|_{C_1}\to T_{X_1/k}|_{C_1}$ is injective.
  The restriction of \eqref{eq:tangent-field-extension} to $C_1$ and \eqref{eq:PGA} fit into the following commutative diagram with exact rows:
  \[
    \begin{tikzcd}
      0 \ar[r] 
        & P \ar[r] \ar[d,hook,"\theta"']
        & \mathscr G|_{C_1} \ar[r,"\iota|_{C_1}"] \ar[d,hook]
        & \mathscr G'|_{C_1} \ar[r] \ar[d,equal]
        & A \ar[r] \ar[d]
        & 0 \\
      0 \ar[r]
        & T_{X_1/k_1}|_{C_1} \ar[r]
        & T_{X_1/k}|_{C_1} \ar[r,"\mathrm{res}"]
        & \mathscr G'|_{C_1} \ar[r]
        & 0
        & {}
    \end{tikzcd}
  \]
  This yields an injection $\theta\colon P \to T_{X_1/k_1}|_{C_1}$.

  Since $C_1$ is a smooth curve of  genus one, we have $T_{C_1/k_1}\simeq\mathcal O_{C_1}$, and the tangent--normal exact sequence is
  \[
    0\longrightarrow T_{C_1/k_1}\simeq\mathcal O_{C_1}
    \longrightarrow T_{X_1/k_1}|_{C_1}
    \buildrel\rho\over\longrightarrow N = \mathcal O_{C_1}(C_1)
    \longrightarrow0.
  \]
  According as $\rho\circ\theta$ is nonzero or zero, $\theta$ induces a nonzero map from $P$ to $N$ or to $\mathcal O_{C_1}$.
  Since both targets have degree zero, $\deg P\leq0$.
  On the other hand, $\deg P=\deg A\geq0$, since $A$ is globally generated.
  Hence, we have $A\simeq\mathcal O_{C_1}$ and $P\simeq N^{-1}$.
  The same nonzero map between degree-zero line bundles is now an isomorphism. Therefore, we have either $N^{-1}\cong N$ or $N\cong\mathcal O_{C_1}$.
  Thus, we have $N^{-1}\cong N$, and hence $N^{\otimes2}\simeq\mathcal O_{C_1}$ in both cases.
\end{proof}

\begin{proposition}\label{prop:caseB-torsion}
	Under the assumptions of Proposition~\ref{prop:a_X-surjective}, we have $6K_X \sim 0$ if $p=3$, and $12K_X \sim 0$ if $p=2$.
\end{proposition}

\begin{proof}
	Since $X_1$ is regular, we have $\mu=\mathrm{id}$, so $f_1=g\circ\epsilon$. We have the following commutative diagram
	\begin{equation*}
		\begin{tikzcd}[row sep=normal]
			X & X_1 \ar[l,"\phi"']\ar[dr,"f_1"] \\[2ex]
			&C_1\ar[u,hook]\ar[r,"\sigma"] & B_1\rlap{\,.}
		\end{tikzcd}
	\end{equation*}
	
	We first consider the $p=3$ case. In this case, $\sigma$ is an isomorphism. Since $C_1$ is a regular curve of arithmetic genus one, we have $(K_{X_1}+C_1)|_{C_1}\sim K_{C_1}\sim0$. Therefore, we have 
	\[
	\phi^*K_X|_{C_1} =(K_{X_1}+2C_1)|_{C_1} \sim C_1|_{C_1}.
	\]
	By Lemma~\ref{lemm:caseB-normal-bundle}, we have $2C_1|_{C_1}\sim0$, hence $2\phi^*K_X|_{C_1}\sim0$.
	By Lemma~\ref{lemm:descent-div}\,\textup{(2)}, there is a line bundle $M$ on $B_1$ such that $2\phi^*K_X\simeq f_1^*M$.
	Restricting to $C_1$ gives $\sigma^*M\sim0$, so $M\sim0$ because $\sigma$ is an isomorphism, and hence $2\phi^*K_X\sim0$.
	By Lemma~\ref{lemma-finite-Cartier}(2), we obtain $6K_X\sim0$.
	
    \smallskip
	Next we consider the $p=2$ case. 
	In Case~(a), we have $\phi^*K_X=K_{X_1}+C_1+C_1'$ and $\sigma$ is an isomorphism. 
	If $C_1\neq C'_1$, we obtain $C_1\cap C'_1=\emptyset$ by $C_1\cdot C'_1=0$. Since $C_1$ is a regular curve of arithmetic genus one, it follows that
	\[
	\phi^*K_X|_{C_1}
	=(K_{X_1}+C_1+C_1')|_{C_1}=K_{C_1}\sim0.
	\]
	As above, Lemma~\ref{lemm:descent-div}\,\textup{(2)} gives $\phi^*K_X\simeq f_1^*M$ for a line bundle $M$ on $B_1$, and $\sigma^*M\sim0$ forces $M\sim0$.
	Hence $\phi^*K_X\sim0$, and Lemma~\ref{lemma-finite-Cartier}(2) gives $2K_X\sim0$.
	
	Next we consider the case $C_1= C'_1$. 
	We have the following commutative diagram:
	\begin{equation*}
		\begin{tikzcd}[row sep=normal]
			X & X_1 \ar[l,"\phi"'] \\[2ex]
			C_0\ar[u,hook]&C_1\ar[u,hook]\ar[l,"\psi"]
		\end{tikzcd}
	\end{equation*}	
	We have shown that $C_1$ and its image $C_0$ in $X$ are regular curves of arithmetic genus one. We have
	$(K_{X_1}+C_1)|_{C_1}=K_{C_1}\sim0$ and $(K_X+C_0)|_{C_0}=K_{C_0}\sim0.$
	On the other hand, we have $\phi^*C_0=\alpha C_1$ with $\alpha\in\{1,2\}$. 
	Therefore, it follows that \[
		\phi^*(K_X+C_0)|_{C_1}=(K_{X_1}+(\alpha+2)C_1)|_{C_1}=\psi^*((K_X+C_0)|_{C_0})\sim 0. 
    \]
	Thus $(\alpha+1)C_1|_{C_1}\sim0$, and hence
	$6\phi^*K_X|_{C_1}\sim0$. 
	As above, writing $6\phi^*K_X\simeq f_1^*M$ by Lemma~\ref{lemm:descent-div}\,\textup{(2)}, we get $\sigma^*M\sim0$, hence $M\sim0$ and $6\phi^*K_X\sim0$.
	By Lemma~\ref{lemma-finite-Cartier}(2), we obtain $12K_X\sim0$.
	
	In Case~(b), we have $\phi^*K_X=K_{X_1}+C_1$ and $\sigma$ is a morphism of degree two. Since $C_1$ is a regular curve of genus one, we obtain $\phi^*K_X|_{C_1}=(K_{X_1}+C_1)|_{C_1}=K_{C_1}\sim0$. 
	Writing $\phi^*K_X\simeq f_1^*M$ by Lemma~\ref{lemm:descent-div}\,\textup{(2)}, we get $\sigma^*M\sim0$, so $2M\sim0$ by Lemma~\ref{lemma-finite-Cartier}\,\textup{(1)} since $\sigma$ has degree two.
	Hence $2\phi^*K_X\sim0$.
	Therefore, we obtain $4K_X\sim0$ by Lemma~\ref{lemma-finite-Cartier}(2).

    We conclude that $12K_X\sim0$ if $p=2$.
\end{proof}

\begin{lemma}\label{lemm:caseB-second-fibration}%
  We have $\dim\Alb_{X/k}=1$.
  Moreover, there is a fibration $h\colon X\to D$ onto a regular projective curve whose generic fiber $\Gamma$ satisfies $p_a(\Gamma)=1$, and $\alb_{X/k}$ does not factor through $h$.
\end{lemma}
\begin{proof}
  Set $A:=\Alb_{X/k}$ and $\alb:=\alb_{X/k}$.
  The factorization of the Albanese morphism through the fibration $f_1\colon X_1\to B_1$ identifies $\Alb_{X_1/k_1}$ with $\Alb_{B_1/k_1}$ by the universal property. The normalization of $(B_1\otimes_{k_1}\bar k)_{\mathrm{red}}$ is an elliptic curve, since it is dominated by the elliptic curve $(C_1\otimes_{k_1}\bar k)_{\mathrm{red}}$ and maps nonconstantly to an abelian variety. Hence $\dim\Alb_{B_1/k_1}=1$ by Lemma~\ref{prop:Alb-eq-dim}\,\textup{(3)}. 
  By Lemma~\ref{prop:Alb-eq-dim}, we see that $\dim A=1$.

  By the proof of Lemma~\ref{lemm:C1^2}, the curve $C_0$ satisfies $C_0^2=0$ and is semi-ample.
  Let $h\colon X\to D$ be the fibration induced by $C_0$.
  Since $C_0\neq0$ and $C_0^2=0$, the base $D$ is a curve, and it is regular because $X$ is normal.
  The generic fiber $\Gamma$ of $h$ is a regular curve with $H^0(\Gamma,\mathcal O_\Gamma)=K(D)$, and $p_a(\Gamma)=1$ because $K_X\equiv0$ and the adjunction formula give $\deg K_\Gamma=0$.
  Finally, $\alb$ does not factor through $h$, since the $h$-vertical curve $C_0$ has nonconstant image under $\alb$.
\end{proof}

\begin{proof}[Proof of Proposition~\ref{prop:a_X-surjective}]
  Assertion~(1) follows from \eqref{eq:adjunction-F}, while assertion~(2) follows from Proposition~\ref{prop:caseB-torsion}.
  By Lemma~\ref{lemm:caseB-second-fibration}, we have $\dim\Alb_{X/k}=1$.
  Since $k$ is assumed to be separably closed, we may choose an origin on $\Alb_{X/k}$, so that it is an elliptic curve.
  Assertions~(3) and~(4) now follow from Lemmas~\ref{lemm:caseB-second-fibration} and~\ref{lemm:sur-second-fibration}.
\end{proof}

\section{Torsion bounds with trivial Albanese variety}\label{sec:proof}
We study numerically trivial line bundles on regular, geometrically non-normal surfaces with trivial Albanese variety.
When $-K_X$ is nef, we prove $pL\sim0$ for $p\geq3$ and $4L\sim0$ for $p=2$.

\subsection{Auxiliary results}\label{sec:caseC-auxiliary}
For a numerically trivial line bundle $L$ on a regular del Pezzo surface $Z$, Proposition~\ref{prop-ell_F-del-Pezzo} gives $3L\sim0$ in characteristic $3$ and $4L\sim0$ in characteristic $2$.
The following two lemmas strengthen these bounds under additional hypotheses.

\begin{lemma}\label{lemm:del-Pezzo-p3}
  Let $Z$ be a regular projective del Pezzo surface over a field $k$ of characteristic $3$, with $H^0(Z,\mathcal O_Z) = k$.
  Assume that one of the following holds:
  \begin{itemize}
    \item[\textup{(a)}] $Z$ is geometrically integral
    \item[\textup{(b)}] $Z$ admits a non-reduced effective Cartier divisor $D\equiv-K_Z$.
  \end{itemize}
  Then every numerically trivial line bundle on $Z$ is trivial.
\end{lemma}
\begin{proof}
  Under assumption~(b), we first show that $Z$ is geometrically normal.
  Let $A$ be a prime component of $D$ with multiplicity at least two, and put $R:=D-2A\geq0$.
  Suppose the contrary that $Z$ is geometrically non-normal.
  Set
  \[
    W := (Z\times_k\overline{k})_{\mathrm{red}}^\nu,
    \qquad \psi\colon W\to Z,
  \]
  where $\bar k$ is the algebraic closure of $k$.
  By \cite[Theorem~4.1]{Patakfalvi-Waldron-2022}, there is an effective integral divisor $C_W$ such that
  \[
    K_W + 2C_W \sim \psi^*K_Z,
  \]
  and $(W,C_W)$ is either $(\mathbb P^2,L)$ or $(S_d,F)$ for $d\geq2$.
  Here $L$ is a line in $\mathbb P^2$.
  In the second case, let $\sigma\colon\mathbb F_d\to S_d\simeq\mathbb P(1,1,d)$ contract the negative section $E$, and set $F:=\sigma(\widetilde F)$ for a fiber $\widetilde F$ of $\mathbb F_d\to\mathbb P^1$.
  Note that $\operatorname{Cl}(S_d)=\mathbb Z[F]$.
  Taking $d=1$ in the first case, we may therefore write $W=\mathbb P(1,1,d)$ for $d\geq1$ and let $H$ denote the generator represented by $L$ or $F$, respectively.
  Then
  \[
    K_W \sim -(d+2)H,\qquad C_W \sim H,\qquad
    \Pic(W) = \mathbb Z[dH],\qquad \psi^*(-K_Z) \sim dH.
  \]
  Since $\psi^*A$ is a nonzero effective Cartier divisor and $\psi^*R$ is an effective Cartier divisor, we have
  \[
    \psi^*A \sim adH,\qquad \psi^*R \sim rdH
  \]
  for integers $a\geq1$ and $r\geq0$.
  Since $-K_Z\equiv D=2A+R$, comparison of numerical classes after pullback shows that $1=2a+r\geq2$, a contradiction.

  Let $L$ be a numerically trivial line bundle on $Z$.
  In both cases of the hypothesis, $Z$ is geometrically integral.
  Hence $H^1(Z,\mathcal O_Z)=0$ by \cite[Proposition~4.10]{Bernasconi-Martin-2024-Bounding}.
  Serre duality and ampleness of $-K_Z$ give
  \[
    H^2(Z,L) = H^0(Z,\omega_Z\otimes L^{-1})^\vee = 0,
    \qquad H^2(Z,\mathcal O_Z)=0.
  \]
  Riemann--Roch yields $\chi(Z,L)=\chi(Z,\mathcal O_Z)=1$, so $H^0(Z,L)\neq0$.
  Thus $L\simeq\mathcal O_Z$.
\end{proof}

\begin{lemma}\label{lemm:del-Pezzo-p2}
  Let $Z$ be a regular del Pezzo surface over a field $k$ of characteristic $2$, with $H^0(Z,\mathcal O_Z)=k$.
  If there exists an effective Cartier divisor $A>0$ with $A\equiv-K_Z$, then every numerically trivial line bundle $L$ on $Z$ satisfies $2L \sim 0$.
\end{lemma}
\begin{proof}
  Let $L$ be a numerically trivial line bundle on $Z$, and set
  \[
    \psi\colon W:=(Z\times_k\bar k)_{\mathrm{red}}^\nu\longrightarrow Z.
  \]
  Then $W$ is a rational surface and thus $\psi^*L\sim 0$.
  By \cite[Theorems~3.7]{Bernasconi-Tanaka-22}, we have $\ell_F(Z/k)\leq2$.
  If $\ell_F(Z/k)\leq1$, then $2 L \sim 0$ holds.
  We may therefore assume that $\ell_F(Z/k)=2$.
  By \cite[Theorem~3.5\textup{(4)}]{Bernasconi-Martin-2024-Bounding} and \cite[Theorem~4.6\textup{(3)}]{Tanaka-2024-Boundedness},
  \[
    W\simeq\mathbb P(1,1,d),\qquad d\in\{1,2,4\},\qquad
    \psi^*(-K_Z)\sim dH,
  \]
  where $H$ is the positive generator of the Weil divisor class group of $W$.

  We distinguish the following two cases:
  \begin{enumerate}[label=\textup{(\alph*)}]
    \item There exists a Cartier divisor $M\equiv2K_Z$ with $H^1(Z,M)\ne0$.
    \item We have $H^1(Z,M)=0$ for every Cartier divisor $M\equiv2K_Z$.
  \end{enumerate}

  \par\smallskip
  \noindent\textbf{Case~\textup{(a)}.}
  Fix a Cartier divisor $M$ as in \textup{(a)}.
  By Serre vanishing,
  \[
    H^1(Z,2^nM)
    \simeq H^1(Z,K_Z-2^nM)^\vee=0
    \quad\text{for } n\gg0.
  \]
  Since $H^1(Z,M) \ne 0$, we may choose $q=2^e$ with $e \geq 0$ maximal such that $H^1(Z,qM)\ne0$.
  Set $\mathscr L = \mathcal O_Z(qM)$.
  Following \cite[I, (1.5)]{ekedahl_canonical_1988}, define $\alpha_{\mathscr L}$ by the exact sequence of fppf sheaves
  \begin{equation}\label{eq:del-Pezzo-p2-alpha}
    0\longrightarrow\alpha_{\mathscr L}
      \longrightarrow\mathscr L
      \xrightarrow{F}\mathscr L^{\otimes2}
      \longrightarrow0,
  \end{equation}
  where $F$ is the relative Frobenius morphism.
  The group scheme $\alpha_{\mathscr L}$ is finite flat of degree two and locally isomorphic to $\alpha_2=\operatorname{Spec}(\mathcal O_Z[\epsilon]/(\epsilon^2))$.
  An $\alpha_{\mathscr L}$-torsor is a scheme with an $\alpha_{\mathscr L}$-action that is fppf locally isomorphic to the group acting on itself by translation.
  Such torsors are classified by $H^1_{\mathrm{fppf}}(Z,\alpha_{\mathscr L})$.
  The long exact cohomology sequence of \eqref{eq:del-Pezzo-p2-alpha} gives
  \[
    H^0(Z,\mathscr L^{\otimes2})
      \longrightarrow H^1_{\mathrm{fppf}}(Z,\alpha_{\mathscr L})
      \longrightarrow H^1(Z,\mathscr L)
      \xrightarrow{F} H^1(Z,\mathscr L^{\otimes2})=0.
  \]
  Thus a nonzero class in $H^1(Z,\mathscr L)$ lifts to a nontrivial $\alpha_{\mathscr L}$-torsor $\pi\colon T\to Z$.
  Over an affine open $U=\operatorname{Spec}R$ trivializing $\mathscr L$, this torsor has the form $T|_U=\operatorname{Spec}(R[t]/(t^2-b))$ for some $b\in R$; see \cite[I, p.~106]{ekedahl_canonical_1988}.
  Hence $\pi$ is finite flat and purely inseparable of degree two, and $T$ is Gorenstein because $Z$ is regular.
  As in \cite[I, Proposition~1.7]{ekedahl_canonical_1988}, $T$ is reduced and thus integral, and we have
  \begin{equation}\label{eq:del-Pezzo-p2-torsor}
    K_T\sim\pi^*(K_Z+qM).
  \end{equation}

  Put $\mathscr L_W:=\psi^*\mathscr L\simeq\mathcal O_W(-2qdH)$.
  Let $h\colon\mathbb P^2\to W$ be the finite morphism induced by the graded free extension
  \[
    \bar k[x,y,z]\hookrightarrow\bar k[x,y,t],\qquad
    z\longmapsto t^d,\qquad \deg z=d,\quad\deg x=\deg y=\deg t=1.
  \]
  Then $\mathcal O_W$ is a direct summand of $h_*\mathcal O_{\mathbb P^2}$ and $h^*\mathscr L_W\simeq\mathcal O_{\mathbb P^2}(-2qd)$.
  Since the inclusion $\mathcal O_W\hookrightarrow h_*\mathcal O_{\mathbb P^2}$ splits, we have
  \[
      H^1(W,\mathscr L_W) \hookrightarrow 
      H^1(W, \mathscr L_W \otimes h_*\mathcal O_{\mathbb P^2})
      \simeq
        H^1(W,h_*h^*\mathscr L_W)
        \simeq H^1(\mathbb P^2,h^*\mathscr L_W) = 0.
  \]
  Thus $H^1(W,\mathscr L_W) = 0$.
  Since also $H^0(W,\mathscr L_W^{\otimes2})=0$,
  the cohomology sequence of \eqref{eq:del-Pezzo-p2-alpha}, after base change to $W$, gives $H^1_{\mathrm{fppf}}(W,\alpha_{\mathscr L_W})=0$, so the pulled-back torsor is trivial.
  Hence $(T\times_ZW)_{\mathrm{red}}\simeq W$, and its finite map to $(T\times_k\bar k)_{\mathrm{red}}$ is birational.
  Since $W$ is normal, it is the normalization of $(T\times_k\bar k)_{\mathrm{red}}$.
  The resulting morphisms fit into the commutative diagram
  \[
    \begin{tikzcd}[column sep=large,row sep=large]
      (T\times_k\bar k)_{\mathrm{red}}^\nu \ar[r,"\sim"] \ar[d]
        & (T\times_ZW)_{\mathrm{red}} \ar[r,"\sim"] \ar[dl]
        & W=(Z\times_k\bar k)_{\mathrm{red}}^\nu \ar[d,"\psi"] \\
      T \ar[rr,"\pi"'] && Z\rlap{\,.}
    \end{tikzcd}
  \]
  Let $\nu\colon T^\nu\to T$ be the normalization and put $k_T=H^0(T^\nu,\mathcal O_{T^\nu})$.
  The morphism $T^\nu\to Z$ has height one, so $T^\nu$ is $\mathbb Q$-factorial.
  For $a\in k_T$, we have $a^2\in K(Z)$ algebraic over $k$, hence $a^2\in k$; thus $k_T/k$ is purely inseparable of height at most one.
  Identifying $W$ with the normalization of $(T^\nu\times_{k_T}\bar k)_{\mathrm{red}}$, let $\theta\colon W\to T^\nu$ denote the natural morphism.
  The conductor formula \cite[Lemma~2.14]{Patakfalvi-Waldron-2022} and the canonical base-change formula \cite[Theorem~3.16]{tanaka_invariants_2021} give
  \[
    K_{T^\nu}+\Gamma\sim\nu^*K_T,\qquad
    K_W+D\simQ\theta^*K_{T^\nu},\qquad \Gamma,D\ge0,
  \]
  where $D$ is nonzero if $T^\nu/k_T$ is not geometrically normal.
  Since $K_W\sim-(d+2)H$, equation~\eqref{eq:del-Pezzo-p2-torsor} gives
  \begin{equation}\label{eq:del-Pezzo-p2-correction}
    D+\theta^*\Gamma\simQ(2-2qd)H.
  \end{equation}
  Since $D$ and $\Gamma$ are effective, \eqref{eq:del-Pezzo-p2-correction} shows that $2-2qd\ge0$, so $q=d=1$.
  The effective divisor $D+\theta^*\Gamma$ is then numerically trivial, hence zero, so $D=\Gamma=0$.
  The conductor is zero in codimension one, and $T$ is $S_2$ because it is Gorenstein; hence $T=T^\nu$ is normal.
  The condition $D = 0$ implies that $T$ is geometrically normal over $k_T$, so $T\times_{k_T}\bar k\simeq W \simeq \mathbb P^2$.
  Therefore $\pi^*L$ is trivial for every numerically trivial $L$, and thus $2 L \sim 0$.

  \par\smallskip
  \noindent\textbf{Case~\textup{(b)}.}
  In this case $H^1(Z,M)=0$ for every Cartier divisor $M\equiv2K_Z$.
  By assumption, there exists an effective Cartier divisor $A$ with $A\equiv-K_Z$.
  By Serre duality, $H^1(Z,\mathcal O_Z(A))= H^1(Z, \mathcal O(K_Z - A))^\vee = 0$ since $K_Z - A \equiv 2K_Z$, and $H^2(Z,\mathcal O_Z(A))= H^0(Z,\mathcal O(K_Z - A))^\vee = 0$.
  Hence
  \[
    \chi(\mathcal O_Z(A))=h^0(Z,\mathcal O_Z(A))>0.
  \]
  For a numerically trivial line bundle $L$ on $Z$, Riemann--Roch and Serre duality give
  \[
    \chi(\omega_Z^{-1}\otimes L)=\chi(\mathcal O_Z(A))>0,
    \qquad H^2(Z,\omega_Z^{-1}\otimes L)=0.
  \]
  Hence $h^0(Z,\omega_Z^{-1}\otimes L)\geq\chi(\omega_Z^{-1}\otimes L)>0$.
  Choose $A_L\in|\omega_Z^{-1}\otimes L|$.
  Each prime component of $A_L$ is Cartier because $Z$ is regular, and its pullback is linearly equivalent to a positive integral multiple of $dH$.
  Since $\psi^*A_L\sim dH$, the divisor $A_L$ consists of a single prime component with multiplicity one, hence is nonzero and integral.
  Adjunction and Riemann--Roch give, respectively,
  \begin{equation}\label{eq:del-Pezzo-p2-adjunction}
    L|_{A_L}\simeq\omega_{A_L},\qquad
    \chi(\mathcal O_{A_L}) =-\frac12A_L\cdot(A_L+K_Z)=0.
  \end{equation}
  By Serre duality and $\chi(\mathcal O_{A_L}) = 0$, we have
  \[
    h^0(A_L,\omega_{A_L}) = h^1(A_L,\mathcal O_{A_L})
      = h^0(A_L,\mathcal O_{A_L}) \geq 1.
  \]
  Thus $\omega_{A_L}$ has degree zero and a nonzero section, so it is trivial; in particular $L|_{A_L} \cong \mathcal O_{A_L}$.
  Let $2A_L\subset Z$ denote the closed subscheme defined by the ideal $\mathcal O_Z(-2A_L)$.
  The closed immersion $A_L\hookrightarrow 2A_L$ induces the restriction homomorphism $\Pic(2A_L)\to\Pic(A_L)$, sending the class of a line bundle $\mathcal N$ to that of $\mathcal N|_{A_L}$.
  The ideal sheaf of $A_L$ in $2A_L$ is
  \[
    \mathcal I = \mathcal O_Z(-A_L)/\mathcal O_Z(-2A_L)
      \subset \mathcal O_{2A_L}.
  \]
  Locally, if $A_L$ is defined by $f\in R$, then $\mathcal I$ corresponds to the ideal $(f)/(f^2)$ in $R/(f^2)$, whose square is zero.
  Let $\mathcal N$ be a line bundle on $2A_L$ whose restriction to $A_L$ is trivial.
  Fix a trivialization $\mathcal N|_{A_L}\simeq\mathcal O_{A_L}$.
  Choose local trivializations of $\mathcal N$ extending this trivialization, so that its transition functions have the form $1+u_{ij}$ with $u_{ij}$ local sections of $\mathcal I$.
  Since $\mathcal I^2 = 0$ and the characteristic is two, these transition functions satisfy
  \[
    (1+u_{ij})^2 = 1+2u_{ij}+u_{ij}^2 = 1.
  \]
  Thus $\mathcal N^{\otimes2}\simeq\mathcal O_{2A_L}$.
  Taking $\mathcal N=L|_{2A_L}$, whose restriction to $A_L$ is trivial as obtained before, we obtain $L^{\otimes2}|_{2A_L}\simeq\mathcal O_{2A_L}$.
  Finally, the exact sequence
  \begin{equation}\label{eq:del-Pezzo-p2-double}
    0\longrightarrow\mathcal O_Z(2K_Z)
      \longrightarrow L^{\otimes2}
      \longrightarrow\mathcal O_{2A_L}\longrightarrow0
  \end{equation}
  and $H^1(Z,\mathcal O_Z(2K_Z))=0$ give a nonzero section of $L^{\otimes2}$ and hence $2 L \sim 0$.
\end{proof}

\begin{lemma}\label{lemm:genus-one-p3-forms}
  Let $k$ be a separably closed field of characteristic $3$, and let $C$ be a regular, geometrically integral projective curve over $k$ with $H^0(C,\mathcal O_C)=k$ and $p_a(C)=1$.
  If $C$ is not smooth over $k$, then
  \[
    C\simeq\Proj k[x,y,z]/(y^2z-x^3-az^3)
    \qquad\text{for some }a\in k\setminus k^3.
  \]
\end{lemma}
\begin{proof}
  Since $C$ is regular with $p_a(C)=1$, we have $\deg\omega_C=0$ and $h^0(C,\omega_C)=h^1(C,\mathcal O_C)=1$; a nonzero section of a degree-zero line bundle on an integral projective curve vanishes nowhere, so $\omega_C\simeq\mathcal O_C$.

  The smooth locus of $C$ is a nonempty open subscheme, and it has a $k$-rational point $P$ because $k$ is separably closed.
  By \cite[Corollary~11.12\textup{(1)}]{tanaka_invariants_2021}, the linear system $|3P|$ embeds $C$ into $\mathbb P^2_k$ as a cubic.
  In Weierstrass coordinates, its equation has the form
  \[
    y^2z+b_1xyz+b_3yz^2=x^3+b_2x^2z+b_4xz^2+b_6z^3,
    \qquad b_i\in k.
  \]
  Since $2$ is invertible in $k$, set $Y=y+\frac12(b_1x+b_3z)$.
  The identity
  \[
    y^2z+(b_1x+b_3z)yz
      =Y^2z-\frac14(b_1x+b_3z)^2z
  \]
  eliminates the terms linear in $Y$.
  Renaming $Y$ as $y$ and the resulting coefficients as $a_2,a_4,a_6$, we obtain
  \[
    C=\{y^2z=x^3+a_2x^2z+a_4xz^2+a_6z^3\} .
  \]
  The point $(0:1:0)$ is smooth, and a point of $C\cap\{z\neq0\}$ is singular over $\overline{k}$ if and only if $y=0$ and $x$ is a multiple root of $f(x):=x^3+a_2x^2+a_4x+a_6$.
  Such a root exists because $C$ is not smooth.
  If it were not a triple root, it would be the root of the greatest common divisor of $f$ and $f'$ in $k[x]$, hence would lie in $k$, and $C$ would have a singular $k$-rational point, contradicting regularity.
  Therefore $f=(x-x_0)^3=x^3-x_0^3$, so $a_2=a_4=0$ and $a:=a_6$; moreover $a\notin k^3$, since otherwise the singular point would be $k$-rational.
\end{proof}

\subsection{The case \texorpdfstring{$K_X\equiv0$}{K\_X numerically trivial}}\label{sec:caseC-numerically-trivial}
\begin{proposition}\label{prop:Alb-trivial}
  Let $X$ be a regular, geometrically non-normal projective surface over a field $k$ of characteristic $p>0$ such that $H^0(X,\mathcal O_X)=k$, $K_X\equiv0$ and $\dim\Alb_{X/k}=0$.
  Then every numerically trivial line bundle $L$ on $X$ satisfies
  \[
    pL\sim0\quad\text{if }p\geq3,\qquad
    4L\sim0\quad\text{if }p=2.
  \]
\end{proposition}
\begin{proof}
  By descending $(X,L)$ to a model satisfying the same hypotheses over a finitely generated subfield and then passing to its separable closure, we may assume that $k$ is $F$-finite and separably closed.
  Set $k_1=k^{1/p}$ and $X_1=(X\times_k k_1)_{\mathrm{red}}^\nu$.
  Construction~\ref{cons:MMP} gives
  \[
    \begin{tikzcd}[row sep=tiny]
      && \widetilde X_1 \ar[ld,"\mu"'] \ar[rd,"\epsilon"] \ar[rdd,"f"'] \\
      X & X_1 \ar[l,"\phi"'] && Z \ar[d,"g"] \\[2ex]
        &&& B\rlap{\,.}
    \end{tikzcd}
  \]
  Here $\mu$ is the minimal resolution, $\epsilon$ is a $K_{\widetilde X_1}$-MMP contraction, $g$ is a Mori fiber space, and $f=g\circ\epsilon$.
  Note that
  \[
    H^0(X_1,\mathcal O_{X_1})=H^0(\widetilde X_1,\mathcal O_{\widetilde X_1})
      =H^0(Z,\mathcal O_Z)=H^0(B,\mathcal O_B)=k_1.
  \]
  In this proof, intersection numbers are computed over $k$ on $X$ and over $k_1$ on $X_1$, $\widetilde X_1$, $Z$, and their further birational models.
  Theorem~\ref{th:geonon-normal} and \eqref{eq:JJXM} give
  \begin{equation}\label{eq:caseC-boundary}
      K_{X_1}+(p-1)\mathfrak C\sim\phi^*K_X,\quad \Delta:=(p-1)\widetilde{\mathfrak C}+\sum_i m_iE_i, \quad K_{\widetilde X_1}+\Delta\sim\mu^*\phi^*K_X\equiv0,
  \end{equation}
  where $\mathfrak C>0$ is an effective integral divisor, $\widetilde{\mathfrak C}$ is its strict transform, the $E_i$ are $\mu$-exceptional, and $m_i\in\mathbb Z_{\geq0}$.
  Lemma~\ref{lemma-finite-Cartier}\textup{(2)} reduces the assertion to
  \begin{equation}\label{eq:goal-C}
    \phi^*L\sim0\quad\text{if }p\geq3,\qquad
    2\phi^*L\sim0\quad\text{if }p=2.
  \end{equation}

  \par\smallskip
  \noindent\textbf{The del Pezzo case.}
  Suppose that $\dim B=0$.
  By Lemma~\ref{lemm:descent-div}\textup{(1)}, there is a numerically trivial line bundle $L_Z$ on $Z$ such that
  \begin{equation}\label{eq:caseC-descent-Z}
    \mu^*\phi^*L\simeq\epsilon^*L_Z.
  \end{equation}
  By injectivity of $\mu^*\colon \Pic(X_1)\to \Pic(\widetilde X_1)$, to prove \eqref{eq:goal-C} it suffices to show that $L_Z\sim0$ for $p\geq3$ and $2L_Z\sim0$ for $p=2$.

  For $p\geq5$, Proposition~\ref{prop-ell_F-del-Pezzo} gives $L_Z\sim0$.

  Suppose that $p=3$.
  If $Z$ is geometrically integral, then $L_Z \sim 0$ by Lemma~\ref{lemm:del-Pezzo-p3}.
  If $Z$ is not geometrically integral, then $Z$, and hence $X$, is geometrically non-reduced, so $\mathfrak C=\mathfrak M+\mathfrak F$ has a nonzero movable part.
  In this case, set
  \[
    A=\epsilon_*\widetilde{\mathfrak M}>0,\qquad
    R=\epsilon_*\left(2\widetilde{\mathfrak F}+\sum_i m_iE_i\right)\geq0.
  \]
  Pushing forward \eqref{eq:caseC-boundary} to $Z$ gives $-K_Z\equiv2A+R$.
  So in this case, we also have $L_Z\sim0$ by Lemma~\ref{lemm:del-Pezzo-p3}.

  Suppose now that $p=2$.
  Set $A:=\epsilon_*\Delta$.
  Since $Z$ is regular, $A$ is an effective Cartier divisor, and pushing forward \eqref{eq:caseC-boundary} gives $A\equiv-K_Z$.
  The ampleness of $-K_Z$ implies $A>0$.
  Lemma~\ref{lemm:del-Pezzo-p2} therefore gives $L_Z^{\otimes2}\simeq\mathcal O_Z$.

  \par\smallskip
  \noindent\textbf{The ruled case.}
  Suppose that $\dim B=1$.
  Since $\Alb_{X/k}$ is trivial, $B$ has geometric genus zero.
  By Lemma~\ref{lemm:descent-div}\textup{(2)}, there is a degree-zero line bundle $L_B$ on $B$ such that
  \begin{equation}\label{eq:caseC-descent-B}
    \mu^*\phi^*L\simeq f^*L_B.
  \end{equation}
  We must prove $L_B\sim0$ for $p\geq3$ and $2L_B\sim0$ for $p=2$.
  Lemma~\ref{lemm:geometric-genus-zero-torsion}\textup{(1)} shows that $L_B$ has $p$-power order.
  We first settle the case $p\geq5$ using the horizontal components of $\Delta$, and then treat $p=2,3$ by distinguishing components of coefficient two and one.
  Let $F$ be the generic fiber of $f$.
  It is a regular conic over $K(B)$, and \eqref{eq:caseC-boundary} gives
  \[
    \deg_{K(B)}\Delta|_F=2.
  \]

  Suppose first that a horizontal component $\Gamma$ of $\Delta$ is $\mu$-exceptional.
  The morphism $h\colon\Gamma^\nu\to B$ has degree $d\in\{1,2\}$.
  Since $\phi^*L$ is a line bundle on $X_1$, it is trivial on a neighbourhood of the point $\mu(\Gamma)$.
  Thus $(\mu^*\phi^*L)|_{\Gamma^\nu}\simeq\mathcal O_{\Gamma^\nu}$, and restriction of \eqref{eq:caseC-descent-B} gives $h^*L_B\sim0$.
  Lemma~\ref{lemma-finite-Cartier}\,\textup{(1)} applied to $h$ yields $dL_B\sim0$.
  For $p\geq3$, coprimality of $d$ and $p$ gives $L_B\sim0$; for $p=2$, we obtain $2L_B\sim0$.
  We may therefore assume that every horizontal component of $\Delta$ is nonexceptional.
  Such a component has coefficient divisible by $p-1$.
  If $p\geq5$, this contradicts $\deg_{K(B)}\Delta|_F=2$, completing that case.
  Hence $p\in\{2,3\}$, and either $\Delta_{\mathrm{hor}}=2\Gamma$ with $\Gamma^\nu\simeq B$, or $p=2$ and $\Delta_{\mathrm{hor}}$ is reduced.

  \smallskip
  As in Subsection~\ref{sec:geo-non-normal-base-change}, set $r=\dim_{k_1}\operatorname{Der}_k(k_1)$ and $\mathscr G=T_{X_1/X}$.
  Restriction of derivations to $k_1$ gives an inclusion $\mathscr G\hookrightarrow\operatorname{Der}_k(k_1)\otimes_{k_1}\mathcal O_{X_1} \simeq\mathcal O_{X_1}^{\oplus r}$.
  Let $\mathscr G'$ be the saturation of this image.

  \par\smallskip
  \noindent\textbf{Case 1: $\Delta_{\mathrm{hor}}=2\Gamma$.}
  Let $C=\mu(\Gamma)$, so $C^\nu\simeq B$, and write
  \[
    (p-1)\mathfrak C=2C+C',\qquad
    C'\geq0,\qquad C\not\subset\Supp C'.
  \]
  If $(C^2)_{k_1}<0$, its image $D=\phi(C)_{\mathrm{red}}$ satisfies $(D^2)_k<0$ by Corollary~\ref{lemma-intersection-general}.
  The adjunction formula gives
  \[
    \deg_kK_{D^\nu}+\deg_k\operatorname{Diff}_{D^\nu}
      =((K_X+D)\cdot D)_k=(D^2)_k<0.
  \]
  Thus $D^\nu$ has arithmetic genus zero, and $L|_{D^\nu}$ is trivial.
  Pullback to $\Gamma^\nu\simeq B$ trivializes $L_B$.
  Suppose that $(C^2)_{k_1}\geq0$.
  By restricting $K_{X_1}+2C+C'\equiv0$ to $C^\nu \simeq B$ and applying the adjunction formula, we obtain
  \begin{equation}\label{eq:caseC-adjunction}
       0=\deg_{k_1}K_B+\deg_{k_1}\operatorname{Diff}_{C^\nu}+(C^2)_{k_1}+(C'\cdot C)_{k_1}.
  \end{equation}
  The last three terms are nonnegative, so $\deg_{k_1}K_B\leq0$ and hence $p_a(B)\leq1$.
  If $p_a(B)=0$, then $L_B\sim0$.
  If $p=2$ and $p_a(B)=1$, then Lemma~\ref{lemm:geometric-genus-zero-torsion}\textup{(2)} gives $2L_B\sim0$.

  It remains in this case to consider $p=3$ and $p_a(B)=1$.
  Equation~\eqref{eq:caseC-adjunction} gives
  \begin{equation}\label{eq:caseC-p3-residual}
    C\simeq B,\qquad (C^2)_{k_1}=0,\qquad ((\mathfrak C-C)\cdot C)_{k_1}=0,
  \end{equation}
  and $X_1$ is regular along $C$.
  In particular, $C$ is nef and $\mathfrak C\cdot C=0$.
  The curve $B$ is not smooth over $k_1$, since its geometric genus is zero.
  We consider separately whether $B$ is geometrically integral or geometrically non-reduced, and rule out both cases.

  \par\smallskip
  \noindent\textbf{Subcase 1(a): The base $B$ is geometrically integral.}
  We will show that $\mathfrak C\sim\mathfrak D+E$ with $E\ge0$ and $\mathfrak D=\mu_*f^*Q$ for a closed point $Q\in B$, contradicting $\mathfrak C\cdot C=0$.

  Lemma~\ref{lemm:genus-one-p3-forms} allows us to write $B=\{y^2z=x^3+az^3\}$ with $a\in k_1\setminus k_1^3=k_1\setminus k$.
  Put $U:=x/z$ and $Y:=y/z$, so that $Y^2=U^3+a$ in $K(B)\subseteq K(X_1)$.
  The zero divisor of $Y$ on $B$ is the unique non-smooth point of $B/k_1$,
  \[
    Q=\operatorname{Spec}\bigl(k_1[U]/(U^3+a)\bigr),\qquad [\kappa(Q):k_1]=3,
  \]
  with multiplicity one, since $Y$ generates the maximal ideal of $\mathcal O_{B,Q}$.
  Let $\mathfrak D$ be the zero divisor of $Y$ on $X_1$, so $\mathfrak D=\mu_*f^*Q$.
  Since $\mu$ is an isomorphism near $C$ and $\Gamma\to B$ is a section,
  \begin{equation}\label{eq:foliation-vanishing-degree}
    (\mathfrak D\cdot C)_{k_1}=(f^*Q\cdot\Gamma)_{k_1}=\deg_{k_1}Q=3.
  \end{equation}

  The equation of $B$ forces $\partial(a)$ to vanish along $\mathfrak D$ for every local section $\partial$ of $\mathscr G$.
  Indeed, any local section $\partial$ of $\mathscr G$ is a $k$-derivation, so in characteristic three
  \[
    \partial(a)=2Y\partial(Y).
  \]
  At a prime divisor $P$ where $Y$ vanishes, both $Y$ and $\partial(Y)$ lie in $\mathcal O_{X_1,P}$, so $v_P(\partial(a))\geq v_P(Y)$.
  Lemma~\ref{lemm:foliation-evaluation}, applied with $D=\mathfrak D$, therefore gives $\mathfrak C\sim\mathfrak D+E$ for some effective divisor $E$, as required.


  \par\smallskip
  \noindent\textbf{Subcase 1(b): The base $B$ is geometrically non-reduced.}
  We will show that $\mathfrak M\ne0$ and use it to construct a second pencil.
  This construction will yield $K(X)\cap k_1\supsetneq k$, contradicting the fact that $k$ is algebraically closed in $K(X)$.

  Put $\mathbb K=K(B)$, $V=\operatorname{Der}_k(k_1)$, $q = \operatorname{rank} \mathscr G$ and $s = r-q$.
  The section $\Gamma$ gives a $\mathbb K$-rational point of the generic fiber $F$ of $f$, so
  \[
    F\simeq\mathbb P^1_{\mathbb K}\quad \text{and} \quad K(X_1)=\mathbb K(u), 
  \]
  for a transcendental element $u$.
  Since $B$ is geometrically non-reduced, $X_1$ is geometrically non-reduced too, and so is $X$.
  Theorem~\ref{th:geonon-normal} therefore gives $\mathfrak M\ne0$.

  Since $C$ is nef and $\mathfrak C\cdot C = 0$, we have
  \[
    \mathfrak M\cdot C = \mathfrak F\cdot C = 0 .
  \]
  Consider the Gauss map from \eqref{eq:foliation-grassmannian-map},
  \[
    \gamma_0\colon X_1\dashrightarrow\operatorname{Gr}(q,V)\hookrightarrow\mathbb P(\wedge^s V).
  \]
  It records the subspaces $\mathscr G'(x) \subset V \otimes_{k_1} \kappa(x)$, and its hyperplanes define the movable system $\mathcal M = \gamma_0^*|\mathcal O(1)|$ which contains $\mathfrak M$.
  Because $\mathfrak M\ne0$ and $\mathcal M$ has no fixed components, $\gamma_0$ is nonconstant.
  Choose a member of $\mathcal M$ not containing $C$.
  Its intersection with $C$ is zero, so it is disjoint from $C$.
  Consequently $\gamma_0$ is defined along $C$.

  Resolve the indeterminacy of $\gamma_0\circ\mu$ by a birational morphism $\rho\colon\widehat X_1\to\widetilde X_1$ with $\widehat X_1$ regular.
  Since $\gamma_0$ is a morphism along $C$ and $\mu$ is an isomorphism near $C$, we can take $\rho$ to be an isomorphism near $\Gamma$.
  Write
  \[
    \gamma\colon\widehat X_1\longrightarrow\operatorname{Gr}(q,V)\hookrightarrow\mathbb P(\wedge^s V),
    \qquad A=\gamma^*\mathcal O(1).
  \]
  Keeping the notation $\Gamma$ and $F$ for their transforms, we have
  \[
    (A\cdot\Gamma)_{k_1}=(\Gamma^2)_{k_1}=0.
  \]
   Since the generic fiber $F$ is unchanged by the resolution of the base points, we have \[
     0\leq\deg_{\mathbb K}(A|_F) \le \deg_{\mathbb K}(((\mu\circ\rho)^*\mathfrak M)|_F) \le \frac12 \deg_{\mathbb K} (\Delta|_{F})=1.
  \]
  Since $(A\cdot\Gamma)_{k_1}=(\Gamma^2)_{k_1}=0$ and $\Gamma\not\equiv0$, the Hodge index theorem gives $(A^2)_{k_1}\leq0$. 
  Then the global generation of $A$ shows $(A^2)_{k_1}=0$, and the nonconstant map $\gamma$ has a curve as its image.

  Let $T$ be the normalization of this image, and factor $\gamma$ through $h\colon \widehat X_1\to T$.
  Then $A=h^*H$, where $H$ is the pullback of $\mathcal O(1)$ to $T$.
  The two pencils fit into
  \[
    \begin{tikzcd}[column sep=large,row sep=large]
      & \widehat X_1 \ar[dl,"f\circ\rho"'] \ar[d,"\mu\circ\rho"] \ar[dr,"h"] & \\
      B & X_1 \ar[l,dashed] \ar[r,dashed] & T.
    \end{tikzcd}
  \]
  Since $A\cdot\Gamma=0$, the image $h(\Gamma)$ is a point, and its residue field embeds into $H^0(\Gamma,\mathcal O_\Gamma)=k_1$.
  Thus $T$ has a $k_1$-rational point.
  The regular curve $T$ is smooth at this point, hence generically smooth; as $k_1$ is separably closed, $T$ is geometrically integral.

  If $\deg_{\mathbb K}(A|_F)=0$, then $h$ is constant on $F$ and factors through a nonconstant morphism $B\to T$, contradicting the constancy of $h|_\Gamma$.
  It follows that $\deg_{\mathbb K}(A|_{F}) = 1$ and the induced morphism \[
    h_F\colon\mathbb P^1_{\mathbb K}\to T_{\mathbb K}
  \]
  is finite and satisfies
  \[
    1 = \deg_{\mathbb K}(h_F^*H_{\mathbb K}) = \deg(h_F)\deg_{k_1}H.
  \]
  Thus $\deg(h_F) = \deg_{k_1}H = 1$.
  Since $H$ is globally generated of degree one, it induces a isomorphism $T\to\mathbb P^1_{k_1}$.
  If $t$ is its rational coordinate, then
  \begin{equation}\label{eq:caseC-p3-second-pencil}
    K(X_1)=\mathbb K(t).
  \end{equation}

  Let $\eta$ be the generic point of $X_1$.
  Saturation does not change the generic stalk, so $\mathscr G_\eta = \mathscr G'_\eta$ as subspaces of $V\otimes_{k_1}K(X_1)$.
  Since $\gamma$ factors through $T$ and $K(T)=k_1(t)$, this subspace is defined over $k_1(t)$.
  More precisely, the generic point of $T$ maps to a $k_1(t)$-point of $\operatorname{Gr}(q,V)$, that is, to a $q$-dimensional subspace
  \[
    W(t)\subset V\otimes_{k_1}k_1(t)
    \quad\text{with}\quad
    \mathscr G'_\eta=W(t)\otimes_{k_1(t)}K(X_1).
  \]
  Thus, in this product description, the subspace of foliation directions depends only on $t$.

  \smallskip
  We now use this description of $\mathscr G'_\eta$ to construct an element $a\in k_1\setminus k$ annihilated by every derivation in $\mathscr G$, which will contradict the algebraic closedness of $k$ in $K(X)$.
  Restricting derivations to $k_1$ gives an exact sequence of $\mathbb K$-vector spaces
  \[
    0\longrightarrow\operatorname{Der}_{k_1}(\mathbb K)
    \longrightarrow\operatorname{Der}_k(\mathbb K)
    \xrightarrow{\mathrm{res}}\operatorname{Der}_k(k_1,\mathbb K) \cong V \otimes_{k_1}\mathbb K,
  \]
  and we define $\mathfrak h\subset V\otimes_{k_1}\mathbb K$ to be the image of $\mathrm{res}$.
  Geometrically, $\mathrm{res}$ is the differential $T_{B/k}\to V\otimes_{k_1}\mathcal O_B$ of the structure morphism $B\to\operatorname{Spec}k_1$ at the generic point of $B$.
  Thus $\mathfrak h$ consists of the $k$-derivations $k_1\to\mathbb K$ that extend to $k$-derivations of $\mathbb K$, that is, the tangent directions of $\operatorname{Spec}k_1$ over $k$ that lift to rational vector fields on $B$.

  Restriction of derivations is dual to the natural map \[
    \alpha\colon\mathbb K\otimes_{k_1}\Omega_{k_1/k}\to\Omega_{\mathbb K/k}
  \]
  So $\mathfrak h = \operatorname{Ann}(\ker \alpha)$

  Note that $\Omega_{k_1/k}=\Omega_{k_1/\mathbb F_3}$ and $\Omega_{\mathbb K/k}=\Omega_{\mathbb K/\mathbb F_3}$, so $\alpha$ is also the natural map of differentials over $\mathbb F_3$.
  The extension $\mathbb K/k_1$ is not separably generated, because $B$ is geometrically non-reduced.
  Hence the differential criterion for separability \cite[Tag~\href{https://stacks.math.columbia.edu/tag/031X}{031X}]{StacksProject} shows that $\alpha$ is not injective.
  Consequently $\mathfrak h\ne V\otimes_{k_1}\mathbb K$.
  We regard $V$ as a $k_1$-subspace of $V\otimes_{k_1}\mathbb K$ via $v\mapsto v\otimes1$.
  Under the identification $\operatorname{Der}_k(k_1,\mathbb K)\cong V\otimes_{k_1}\mathbb K$, this sends $\delta\in\operatorname{Der}_k(k_1)$ to its composite with the inclusion $k_1\subset\mathbb K$.
  Since $V$ spans $V\otimes_{k_1}\mathbb K$ over $\mathbb K$, it follows that
  \[
    \mathfrak d:=V\cap\mathfrak h
      =\{\delta\in\operatorname{Der}_k(k_1):\delta\text{ extends to a derivation of }\mathbb K\}
  \]
  is a proper $k_1$-subspace of $V$.
  If $\delta_1,\delta_2\in\mathfrak d$ extend to derivations $D_1,D_2$ of $\mathbb K$, then $[D_1,D_2]$ and $D_1^3$ extend $[\delta_1,\delta_2]$ and $\delta_1^3$, so $\mathfrak d$ is a restricted Lie subalgebra of $V$.

  Every $D\in\mathscr G_\eta$ is a $k$-derivation of $K(X_1)=\mathbb K(t)$, and $D|_{k_1}=(D|_{\mathbb K})|_{k_1}$ with $D|_{\mathbb K}\in\operatorname{Der}_k(\mathbb K,\mathbb K(t))=\operatorname{Der}_k(\mathbb K)\otimes_{\mathbb K}\mathbb K(t)$.
  Hence $D|_{k_1}\in\mathfrak h\otimes_{\mathbb K}\mathbb K(t)$ for every $D\in\mathscr G_\eta$.
  Since $\mathscr G_\eta$ sits in $V\otimes_{k_1}K(X_1)$ via $D\mapsto D|_{k_1}$, this means
  \[
    W(t)\otimes_{k_1(t)}\mathbb K(t)=\mathscr G_\eta\subseteq\mathfrak h\otimes_{\mathbb K}\mathbb K(t).
  \]
  Let $w\in W(t)$ be arbitrary.
  After multiplying $w$ by a common denominator of its coordinates in a $k_1$-basis of $V$, we may write $w=\sum_i v_i t^i$ with $v_i\in V$.
  Its image modulo $\mathfrak h$ is zero, and transcendence of $t$ over $\mathbb K$ forces each $v_i$ to lie in $\mathfrak h$, hence in $\mathfrak d$.
  It follows that $\mathscr G_\eta=\mathscr G'_\eta\subseteq\mathfrak d\otimes_{k_1}\mathbb K(t)$.
  Let $k_1^{\mathfrak d}$ be the common constant field of $\mathfrak d$.
  Jacobson's correspondence \cite[pp.~561, 565]{Gerstenhaber-1964} gives
  \[
    [k_1:k_1^{\mathfrak d}]=3^{\dim_{k_1}\mathfrak d}<3^r=[k_1:k].
  \]
  Choose $a\in k_1^{\mathfrak d}\setminus k$.
  Every derivation in $\mathscr G$ annihilates $a$, so $a\in K(X_1)^{\mathscr G}=K(X)$, contradicting the algebraic closedness of $k$ in $K(X)$.
  This finishes the proof of this subcase.

  \par\smallskip
  \noindent\textbf{Case 2: $p=2$ and $\Delta_{\mathrm{hor}}$ is reduced.}
  Choose a horizontal component $\Gamma$ of $\Delta$, which has coefficient one.
  Put $h\colon\Gamma^\nu\to B$ and $e=\deg h\in\{1,2\}$.
  The curve $\Gamma$ is the strict transform of a component $C$ of $\mathfrak C$, and we write $\mathfrak C=C+C'$ with $C'\geq0$ not containing $C$.
  Adjunction gives
  \[
    0=\deg_{k_1}K_{C^\nu}
      +\deg_{k_1}\operatorname{Diff}_{C^\nu}+(C'\cdot C)_{k_1}.
  \]
  Thus $C^\nu$ has arithmetic genus at most one over its field of constants.
  In genus zero, $h^*L_B$ is trivial, so Lemma~\ref{lemma-finite-Cartier}\,\textup{(1)} again gives $2L_B\sim0$.
  In genus one, we have $\deg_{k_1} \mathrm{Diff}_{C^\nu} = C' \cdot C = 0$, so $C$ is regular, $X_1$ is regular along $C$, and $C'$ is disjoint from $C$.
  The minimal resolution $\mu$ is therefore an isomorphism near $C$, and we identify $\Gamma$ with $C$.
  If $(C^2)_{k_1}<0$, the argument as in Case~1 with $D=\phi(C)_{\mathrm{red}}$ trivializes $h^*L_B$, so Lemma~\ref{lemma-finite-Cartier}\,\textup{(1)} again gives $2L_B\sim0$.
  We may assume $(C^2)_{k_1}\geq0$, so $(\Gamma^2)_{k_1}\geq0$ and $\Gamma$ is nef on $\widetilde X_1$.

  We distinguish whether the generic conic $F$ is smooth or quasilinear.

  \par\smallskip
  \noindent\textbf{Subcase 2(a): The generic conic $F$ is smooth.}
  Put $\Delta_0:=\Delta-\Gamma\geq0$.
  Then
  \[
    -(K_{\widetilde X_1}+\Delta_0)\equiv\Gamma,\qquad
    \deg_{K(B)}\Delta_0|_F=2-e\in\{0,1\}.
  \]
  Thus $-(K_{\widetilde X_1}+\Delta_0)$ is nef.
  Since $F$ is smooth, the geometric generic fiber $(F_{\overline{K(B)}},\Delta_0|_{F_{\overline{K(B)}}})$ of $f\colon(\widetilde X_1,\Delta_0)\to B$ is $(\mathbb P^1,0)$ or $(\mathbb P^1,P)$ for one reduced point $P$, hence is $F$-pure.
  Lemma~\ref{lemm:Ejiri-base-curve}\textup{(1)} gives $p_a(B)\leq1$, and Lemma~\ref{lemm:geometric-genus-zero-torsion}\textup{(2)} gives $2L_B\sim0$.

  \par\smallskip
  \noindent\textbf{Subcase 2(b): The generic conic $F$ is quasilinear.}
  The conic $F$ has no $K(B)$-rational point, $e=2$ and the extension $K(C)/K(B)$ is purely inseparable. 
  Put $\ell=H^0(C,\mathcal O_C)$. Since $k_1$ is separably closed and $e=2$, either $\ell=k_1$ or $[\ell:k_1]=2$.

  First suppose that $\ell=k_1$.
  Since $C\to B$ is purely inseparable and $B$ has geometric genus zero, $C$ also has geometric genus zero.
  As $H^0(C,\mathcal O_C)=k_1$, Lemma~\ref{lemm:geometric-genus-zero-torsion}\,\textup{(2)} shows that the curve $\widehat C=(C\times_{k_1}k_1^{1/2})_{\mathrm{red}}^\nu$ has arithmetic genus zero.
  The composite $\widehat C\to C\to B$ has height one and the pullback of $L_B$ to $\widehat C$ is trivial, so Lemma~\ref{lemma-finite-Cartier}\textup{(2)} gives $2L_B\sim0$.

  Finally suppose that $[\ell:k_1]=2$, and write $\ell=k_1(\sqrt a)$ with $a\in k_1\setminus k_1^2=k_1\setminus k$.
  Then $K(C)=K(B)(\sqrt a)$.
  Since $F$ is quasilinear, it is defined by a quadratic form
  \[
    sx^2+ty^2+rz^2 \quad\text{with }s,t,r\in K(B).
  \]
  Since $p=2$, we have $(\lambda x+\mu y+\nu z)^2=\lambda^2x^2+\mu^2y^2+\nu^2z^2$, so every linear change of coordinates keeps the form a combination of $x^2,y^2,z^2$ only.
  After a linear change of coordinates, we may assume that $\{z=0\}$ is the line spanned by the degree-two point $C|_F$.
  Then $C|_F$ is the zero of $sx^2+ty^2$ on this line, so $K(B)(\sqrt{t/s})=K(C)=K(B)(\sqrt a)$.
  Hence $t/s=\alpha^2+\beta^2a$ for some $\alpha,\beta\in K(B)$, and a linear change of $x$ and $y$ gives
  \[
    F=\{x^2+ay^2+bz^2=0\},\qquad b\in K(B),\qquad C|_F=\{z=0\}.
  \]
  On $F\cap\{y\ne0\}$, put $u=x/y$ and $v=z/y$.
  At the generic point of $C$, the function $v$ is a uniformizer, and $b$, viewed as a function on $X_1$ via $f$, is regular there since $C$ dominates $B$.
  For a local section $\theta$ of $\mathscr G=T_{X_1/X}$, differentiating $u^2+a+bv^2=0$ gives
  \begin{equation}\label{eq:caseC-p2-theta}
    \theta(a)=v^2\theta(b),
  \end{equation}
  where $\theta(b)$ is regular along $C$.
  Thus $\theta(a)$ vanishes to order at least two along $C$, and Lemma~\ref{lemm:foliation-evaluation}, applied with $D=2C$, gives $\mathfrak C\sim2C+E$ for some $E\geq0$.
  This contradicts $\deg_{K(B)}((\mu^*\mathfrak C)|_F)\leq\deg_{K(B)}(\Delta|_F)=2$.
  This finishes the proof of this subcase and hence the whole proof.
\end{proof}

\subsection{The case \texorpdfstring{$-K_X$ nef and $K_X\not\equiv0$}{-K\_X nef and K\_X numerically nontrivial}}\label{sec:caseC-nef-nontrivial}
\begin{proposition}\label{prop:Alb-trivial-nef}
  Let $k$ be a field of characteristic $p > 0$, and let $X$ be a regular, geometrically non-normal projective $k$-surface such that $H^0(X,\mathcal O_X) = k$, $-K_X$ is nef and $K_X\not\equiv0$, and $\dim \Alb_{X/k} = 0$.
  Then every numerically trivial line bundle $L$ on $X$ satisfies the following:
  \begin{enumerate}[label=\textup{(\arabic*)}]
    \item If $p \ge 5$, then $L \sim 0$.
    \item If $p = 3$, then $3L \sim 0$.
    \item If $p = 2$, then $2^2L \sim 0$.
  \end{enumerate}
\end{proposition}

\begin{proof}
  We may assume that $k$ is $F$-finite and separably closed.
  Since $-K_X$ is nef and numerically nontrivial, the $K_X$-MMP ends with a Mori fiber space
  \[
    \begin{tikzcd}[column sep=large,row sep=small]
      X \ar[r,"\epsilon"] \ar[dr,"f"'] & Z \ar[d,"g"] \\
        & B\rlap{\,.}
    \end{tikzcd}
  \]
  The divisor $-K_Z$ is nef.

  Suppose first that $\dim B = 0$.
  By Lemma~\ref{lemm:descent-div}\,\textup{(1)}, there exists a numerically trivial line bundle $L_Z$ on the regular del Pezzo surface $Z$ such that $L \simeq \epsilon^*L_Z$.
  Then the conclusion follows from Proposition~\ref{prop-ell_F-del-Pezzo}.

  Suppose that $\dim B = 1$.
  By Lemma~\ref{lemm:descent-div}\textup{(2)}, there exists a numerically trivial line bundle $L_B$ on $B$ such that $L \simeq f^*L_B$.
  Since $\Alb_{X/k}$ is trivial, the curve $B$ has geometric genus zero.

  We divide the argument into the cases $p>2$ and $p=2$.

  \par\smallskip
  \noindent\textbf{Case~\textup{(a)}: $p>2$.}
  By \cite[Proposition~2.18]{Bernasconi-Tanaka-22}, the Mori fibration $g$ is smooth because $p>2$ and $k$ is separably closed.
  Its geometric generic fiber is $\mathbb P^1$, so positivity (Lemma~\ref{lemm:Ejiri-base-curve}\,\textup{(1)}) gives $p_a(B)\leq1$.
  If $p\geq5$, Lemma~\ref{lemm:geometric-genus-zero-torsion}\,\textup{(2)} shows $p_a(B)=0$ and hence $L_B\sim0$.
  If $p=3$, the same lemma gives $3L_B\sim0$.
  Pulling back by $f$ proves the required bounds for $L$.

  \par\smallskip
  \noindent\textbf{Case~\textup{(b)}: $p=2$.}
  If $K_Z^2>0$, then $-K_Z$ is nef and big, so Proposition~\ref{prop-ell_F-del-Pezzo} applied to $g^*L_B$ gives $4L\sim0$.
  We may therefore assume that $K_Z^2=0$.
  If the generic conic of $g$ is smooth, Lemma~\ref{lemm:Ejiri-base-curve}\,\textup{(1)} gives $p_a(B)\le1$, and Lemma~\ref{lemm:geometric-genus-zero-torsion}\,\textup{(2)} gives $2L_B\sim0$.
  Hence we may assume that the generic conic is quasilinear.

  Set $k_1=k^{1/2}$ and
  \[
    B_0:=\bigl(B\times_k k_1\bigr)_{\mathrm{red}}^\nu,
    \qquad
    Y:=\bigl(Z\times_k k_1\bigr)_{\mathrm{red}}^\nu,
  \]
  and let $\phi\colon Y\to Z$ be the natural morphism.
  We have $H^0(Y,\mathcal O_Y)=H^0(B_0,\mathcal O_{B_0})=k_1$.
  Let
  \[
    Y\xrightarrow{h}B_1\xrightarrow{\tau}B
  \]
  be the Stein factorization of $g\circ\phi$ and put $M:=\tau^*L_B$.
  The morphism $\tau$ has height at most one, so it suffices to prove $2M\sim0$.
  The curve $B_1$ has geometric genus zero and $H^0(B_1,\mathcal O_{B_1}) = k_1$.
  Since $Y$ is normal and dominates $(B\times_k k_1)_{\mathrm{red}}$, the morphism $g\circ\phi$ factors as $Y\to B_0\xrightarrow{\beta}B$.
  The Stein factorization of $Y\to B_0$ is then also a Stein factorization of $g\circ\phi$, so $\tau=\beta\circ\tau_0$ for a finite morphism $\tau_0\colon B_1\to B_0$.

  Let $\mu\colon S\to Y$ be a minimal resolution and set
  \[
    s:=h\circ\mu,
    \qquad N_Y:=\phi^*(-K_Z),
    \qquad N:=\mu^*N_Y.
  \]
  The morphisms introduced so far fit into the commutative diagram
  \[
    \begin{tikzcd}[column sep=large,row sep=large]
      S \ar[r,"\mu"] \ar[dr,"s"'] & Y \ar[rr,"\phi"] \ar[d,"h"] & & Z \ar[d,"g"] & X \ar[l,"\epsilon"'] \ar[dl,"f"] \\
      & B_1 \ar[r,"\tau_0"] \ar[rr,bend right=25,"\tau"'] & B_0 \ar[r,"\beta"] & B\rlap{\,.} &
    \end{tikzcd}
  \]
  The canonical base-change formula and the minimal-resolution argument in \eqref{eq:JJXM} give
  \[
    K_S+\Delta\sim-N
  \]
  for an effective integral divisor $\Delta$.
  The divisor $N$ is nef with $N^2=0$, and $N_Y$ is $h$-ample.

  Put $\mathbb K:=K(B)$ and $\mathbb K_0:=K(B_0)$.
  The generic conic $V$ of $g$ is quasilinear, so it is defined by an equation
  \begin{equation}\label{eq:caseC-nef-conic}
    c_1x^2+c_2y^2+c_3z^2=0 \quad\text{with }c_1,c_2,c_3\in\mathbb K.
  \end{equation}
  Suppose first that $V_{\mathbb K_0}$ is not regular.
  After a linear change of coordinates, its equation is $x^2=0$ or $x^2+a y^2=0$ with $a\notin \mathbb K_0^2$.
  Its reduced normalization is therefore $\mathbb P^1$ over $\mathbb K_0$ or $\mathbb K_0(\sqrt a)$, respectively, and the hyperplane bundle pulls back to $\mathcal O_{\mathbb P^1}(1)$.
  Thus the generic fiber $F$ of $s$ is $\mathbb P^1_{K(B_1)}$, and
  \[
    \deg N|_F=1,
    \qquad \deg\Delta|_F=1.
  \]
  Consequently, the geometric generic fiber $(F_{\overline{K(B_1)}},\Delta|_{F_{\overline{K(B_1)}}})$ of $s\colon(S,\Delta)\to B_1$ is $(\mathbb P^1,P)$ for one reduced point $P$, hence is $F$-pure.
  Lemma~\ref{lemm:Ejiri-base-curve}\,\textup{(1)} gives $p_a(B_1)\le1$, and Lemma~\ref{lemm:geometric-genus-zero-torsion}\,\textup{(2)} gives $2M\sim0$.

  It remains to treat the case where $V_{\mathbb K_0}$ is regular.
  Then $H^0(V_{\mathbb K_0},\mathcal O_{V_{\mathbb K_0}})=\mathbb K_0$, so $B_1=B_0$.
  We next show that $B$ is not smooth.
  Suppose to the contrary that $B$ is smooth.
  Then $\mathbb K$ is separable over $k(t)$ for some $t\in\mathbb K$, so $\mathbb K=\mathbb K^2k(t)$ and $t\notin\mathbb K^2k$, hence $[\mathbb K:\mathbb K^2k]=2$.
  Consequently, $c_1,c_2,c_3$ are linearly dependent over $\mathbb K^2k=\mathbb K_0^2$, so $V_{\mathbb K_0}$ has a $\mathbb K_0$-rational point.
  By \cite[Theorem~9.10\,(3)]{tanaka_invariants_2021}, the regular conic $V_{\mathbb K_0}$ is then isomorphic to $\mathbb P^1_{\mathbb K_0}$ and hence smooth, which is impossible because all partial derivatives of the polynomial in \eqref{eq:caseC-nef-conic} vanish.
  Hence $B$ is not smooth.

  By Theorem~\ref{th:geonon-normal} applied to $B$, there is an integral divisor $D_B> 0$ on $B_1=B_0$ such that
  \[
    K_{B_1}+D_B\sim\tau^*K_B.
  \]
  The conic fibration $g$ is flat and Gorenstein, so $W:=Z\times_B B_1$ is Gorenstein and is integral because its generic fiber is integral.
  Its normalization is $Y$, since both have function field $K(Z)k_1$.
  Relative dualizing base change and the conductor formula therefore give
  \[
    K_Y+C_{\mathrm{nor}}+h^*D_B\sim\phi^*K_Z,
    \qquad C_{\mathrm{nor}}\ge0.
  \]
  Since $K_Y+C_{\mathrm{nor}}\sim\phi^*K_Z-h^*D_B$ is Cartier and $\mu$ is the minimal resolution, we have $\mu^*(K_Y+C_{\mathrm{nor}})\sim K_S+\Delta'$ with $\Delta'\geq0$ integral, so we may take $\Delta:=\Delta'+s^*D_B$, and then
  \[
    K_S+\Delta\sim-N,
    \qquad \Delta\ge s^*D_B.
  \]
  In particular, $\Delta\cdot N>0$ because $N_Y$ is $h$-ample, and hence $K_S\cdot N<0$.

  Since $(K_S-mN)\cdot N<0$, we have $h^2(S,mN)=0$, so Riemann--Roch gives
  \[
    h^0(S,mN)\geq\chi(S,\mathcal O_S(mN))
      =\chi(S,\mathcal O_S)-\frac m2K_S\cdot N\to\infty.
  \]
  For $m$ sufficiently large, the moving part $A$ of $|mN|$ satisfies $N\cdot A=0$.
  The Hodge index theorem and movability give $A^2=0$, so $|A|$ has no base points.
  Its Stein factorization defines a morphism $q\colon S\to D$ to a regular projective curve, with $N$ numerically proportional to a fiber of $q$.
  The generic fiber $C$ is regular with $H^0(C,\mathcal O_C)=K(D)$, and
  \[
    \deg K_C=-\deg\Delta|_C<0.
  \]
  Thus $p_a(C)=0$ and $\deg\Delta|_C=2$.
  Running a $K_S$-MMP over $D$ and applying Lemma~\ref{lemm:descent-div}\textup{(2)}, we obtain a numerically trivial line bundle $M_D$ on $D$ such that
  \[
    s^*M\simeq q^*M_D.
  \]

  Since $(s^*D_B)\cdot N>0$ and $N$ is numerically proportional to a fiber of $q$, the divisor $s^*D_B$ has a $q$-horizontal prime component $\Gamma$.
  The morphisms used below fit into the commutative diagram
  \[
    \begin{tikzcd}[column sep=large,row sep=large]
      \Gamma^\nu \ar[r] \ar[dr] & S \ar[r,"s"] \ar[d,"q"] & B_1\rlap{\,,} \\
      & D &
    \end{tikzcd}
  \]
  where $\Gamma^\nu\to S$ is the normalization of $\Gamma$ followed by the inclusion.
  Since $\Delta\geq s^*D_B$, the finite morphism $\Gamma^\nu\to D$ has degree
  \[
    1\leq d=\deg_{K(D)}(\Gamma|_C)\leq\deg_{K(D)}(\Delta|_C)=2.
  \]
  The curve $\Gamma$ is $s$-vertical, so $s^*M|_{\Gamma^\nu}$ is trivial, being pulled back from a closed point of $B_1$.
  Restricting $s^*M\simeq q^*M_D$ to $\Gamma^\nu$ and applying Lemma~\ref{lemma-finite-Cartier}\,\textup{(1)} gives $dM_D\sim0$, hence $2M_D\sim0$.
  Since $s_*\mathcal O_S=\mathcal O_{B_1}$, pullback is injective and $2M\sim0$.
  This finishes the proof of this subcase and thus the whole proof.
\end{proof}

\section{Twisted covers and triviality in large characteristic}\label{sec:case-C-p-ge7}
In this section, we prove that every numerically trivial line bundle on a surface as in Theorem~\ref{thm:main}\textup{(C)} is trivial when $p\ge7$.

\subsection{The \texorpdfstring{$t$}{t}-twisted \texorpdfstring{$\mu_p$}{mu-p}-cover}\label{sec:generically-twisted-mu-p-cover}
Let $k$ be a field of characteristic $p>0$, and let $X$ be a $k$-scheme.
Let $L$ be a line bundle on $X$ such that $L^{\otimes p}\simeq\mathcal O_X$. We fix an isomorphism
\[
  \tau\colon L^{\otimes p}\xrightarrow{\sim}\mathcal O_X.
\]
Let $t$ be transcendental over $k$, set $K=k(t)$, and write
$X_K=X\times_k K$.  Denote the pullbacks of $L$ and $\tau$ to
$X_K$ by $L_K$ and $\tau_K$, respectively, and define the trivialization
\[
  \tau_t\colon L_K^{\otimes p}\xrightarrow{\sim}\mathcal O_{X_K},\qquad v\longmapsto t^{-1}\tau_K(v).
\]
Regarding $\tau_t$ as a nowhere vanishing section of $L_K^{-p}=\mathcal H\!om(L_K^{\otimes p},\mathcal O_{X_K})$, let $\lambda_t\colon L_K^{-p}\xrightarrow{\sim}\mathcal O_{X_K}$ be the isomorphism with $\lambda_t(\tau_t)=1$.
We define an $\mathcal O_{X_K}$-algebra
\[
  \mathcal A_{L,t}:=\bigoplus_{i=0}^{p-1}L_K^{-i}
\]
by taking the multiplication to be the usual tensor product on local sections of degrees $i$ and $j$ with $i+j<p$, and the map
\[
  L_K^{-i}\otimes L_K^{-j}
    \simeq
  L_K^{-(i+j-p)}\otimes L_K^{-p}
    \xrightarrow{\,1\otimes\lambda_t\,}
  L_K^{-(i+j-p)}
\]
when $i+j\geq p$.
Let
\[
  \pi\colon\widehat Y:=\operatorname{Spec}_{X_K}\mathcal A_{L,t}\longrightarrow X_K.
\]
The $\mathbb Z/p\mathbb Z$-grading induces a $\mu_p$-action on
$\widehat Y$, and $\pi$ is a $\mu_p$-torsor.  In
particular, it is finite flat of degree $p$.
We call $\pi$ the \emph{$t$-twisted $\mu_p$-cover} associated with $(L,\tau)$.

\smallskip
To describe it locally,
let $U=\operatorname{Spec}R\subset X$ be an open subset on which $L$
is free, and choose a basis $e$ of $L|_U$.  Its dual basis is the
section
\[
  z=e^\vee\in H^0(U,L^{-1}),\qquad z(e)=1.
\]
Thus $z(ae)=a$ for every local function $a$, and $z^{\otimes i}$ is a
basis of $L^{-i}|_U$.  Write
\[
  \tau(e^{\otimes p})=u\in R^\times.
\]
After passing to $U_K$, we have $\tau_t(e^{\otimes p})=t^{-1}u$, that is,
\[
  \tau_t=t^{-1}u z^{\otimes p}
\]
as a section of $L_K^{-p}$.
Consequently, $\lambda_t(z^{\otimes p})=tu^{-1}$, so the degree-one
element $z\in L_K^{-1}$ satisfies $z^p=tu^{-1}$ in the covering
algebra.  Hence the cover over $U_K$ is
\[
  \operatorname{Spec}
    \bigl((R\otimes_kK)[z]/(z^p-tu^{-1})\bigr).
\]

\subsection{Triviality in characteristic \texorpdfstring{$p\ge7$}{p >= 7}}\label{sec:triviality-p-ge7}%
\begin{proposition}
  \label{prop:generically-twisted-mu-p-cover}
  Let $k$ be a field of characteristic $p>0$, and let $X$ be a regular projective $k$-variety such that $H^0(X,\mathcal O_X)=k$.
  Let $L$ be a line bundle on $X$ of exact order $p$, and let
  \[
    \pi\colon \widehat Y
      =\operatorname{Spec}_{X_K}\mathcal A_{L,t}\longrightarrow X_K
  \]
  be the $t$-twisted $\mu_p$-cover associated with $(L,\tau)$ for an isomorphism $\tau\colon L^{\otimes p}\xrightarrow{\sim}\mathcal O_X$.
  Then
  \begin{enumerate}[label=\textup{(\arabic*)}]
    \item $\widehat Y$ is a regular projective variety over $K$ with $H^0(\widehat Y,\mathcal O_{\widehat Y})=K$,
    \item $\omega_{\widehat Y}\simeq\pi^*\omega_{X_K}$ and $\pi^*L_K\simeq\mathcal O_{\widehat Y}$.
  \end{enumerate}
  Assume furthermore that $X_{\overline{k}}$ is integral and $\nu^*L_{\overline{k}}\simeq\mathcal O_Z$, where $\nu\colon Z=X_{\overline{k}}^\nu\to X_{\overline{k}}$ is the normalization.
  Then
  \begin{enumerate}[label=\textup{(\arabic*)},resume]
    \item $\widehat Y$ is geometrically non-reduced over $K$.
  \end{enumerate}
\end{proposition}

\begin{proof}
  (1) The purely transcendental extension $k\subset K$ is separable, so $X_K$ is regular.
  Let $U=\operatorname{Spec}R\subset X$ be an affine open on which $L$ is free, and let $e$, $z=e^\vee$, and $u=\tau(e^{\otimes p})\in R^\times$ be as in Subsection~\ref{sec:generically-twisted-mu-p-cover}.
  By the local description,
  \[
    \Gamma(U_K,\mathcal A_{L,t}) =\bigoplus_{i=0}^{p-1}(R\otimes_kK)z^i \simeq (R\otimes_kK)[z]/(z^p-tu^{-1}).
  \]
  Since $R\otimes_kK$ is the localization of $R[t]$ at $k[t]\setminus\{0\}$, the coordinate ring $(R\otimes_kK)[z]/(z^p-tu^{-1})$ of $\pi^{-1}(U_K)$ is a localization of $R[t,z]/(z^p-tu^{-1})$.
  The latter ring is isomorphic to $R[z]$ via $t\mapsto uz^p$, which is regular since $U$ is open in the regular variety $X$, so $\pi^{-1}(U_K)$ is regular.
  These inverse images cover $\widehat Y$, proving its regularity.
  Finally, $\pi$ is finite, hence projective, and $X_K$ is projective over $K$.
  Their composition shows that $\widehat Y$ is projective over $K$.

  Next, for $1\le i\le p-1$, the line bundle $L^{-i}$ is nontrivial and $(L^{-i})^{\otimes p}\simeq\mathcal O_X$, so $H^0(X,L^{-i})=0$.
  Therefore $H^0(X_K,L_K^{-i}) = 0$.
  Since $\pi$ is affine and $\pi_*\mathcal O_{\widehat Y}=\mathcal A_{L,t}$, we obtain \[
    H^0(\widehat Y,\mathcal O_{\widehat Y}) =\bigoplus_{i=0}^{p-1} H^0(X_K,L_K^{-i})=K.
  \]
  Thus $\widehat Y$ is connected, hence integral because it is regular.

  \smallskip
  (2) We first compute the canonical sheaf of $\widehat Y$.
  Multiplication followed by projection onto the degree-zero summand defines an $\mathcal O_{X_K}$-bilinear pairing
  \[
    \mathcal A_{L,t}\otimes\mathcal A_{L,t}
      \longrightarrow\mathcal O_{X_K},
    \qquad a\otimes b\longmapsto\operatorname{pr}_0(ab).
  \]
  This pairing is perfect, because $\lambda_t\colon L_K^{-i}\otimes L_K^{-(p-i)}=L_K^{-p}\to\mathcal O_{X_K}$ is an isomorphism for $1\le i<p$.
  Therefore the induced map \[
    \mathcal A_{L,t}\longrightarrow
      \mathcal H\!om_{X_K}(\mathcal A_{L,t},\mathcal O_{X_K}),
    \qquad
    a\longmapsto\bigl(b\mapsto\operatorname{pr}_0(ab)\bigr),
  \]
  is an isomorphism of $\mathcal A_{L,t}$-modules, where $\mathcal A_{L,t}$ acts on the target by precomposition with multiplication.
  Since $\pi\colon\widehat Y\to X_K$ is a finite flat morphism of regular varieties, duality for finite morphisms gives
  \[
    \pi_*\omega_{\widehat Y}
      \simeq
    \mathcal H\!om_{X_K}(\mathcal A_{L,t},\omega_{X_K}).
  \]
  The algebra $\mathcal A_{L,t}$ is locally free over $\mathcal O_{X_K}$, so its self-duality gives
  \begin{align*}
    \mathcal H\!om_{X_K}(\mathcal A_{L,t},\omega_{X_K})
      &\simeq
        \mathcal H\!om_{X_K}(\mathcal A_{L,t},\mathcal O_{X_K})
          \otimes\omega_{X_K}\\
      &\simeq \mathcal A_{L,t}\otimes\omega_{X_K}
       \simeq \pi_*\pi^*\omega_{X_K}.
  \end{align*}
  These are isomorphisms of $\mathcal A_{L,t}$-modules.
  Since $\pi_*$ identifies quasi-coherent $\mathcal O_{\widehat Y}$-modules with quasi-coherent $\mathcal A_{L,t}$-modules, we obtain
  \[
    \omega_{\widehat Y}\simeq\pi^*\omega_{X_K}.
  \]

  Next, we trivialize $\pi^*L_K$.
  Under the adjunction $\operatorname{Hom}_{\mathcal O_{\widehat Y}}(\pi^*L_K^{-1},\mathcal O_{\widehat Y})=\operatorname{Hom}_{\mathcal O_{X_K}}(L_K^{-1},\pi_*\mathcal O_{\widehat Y})$, the inclusion of the degree-one summand $L_K^{-1}\hookrightarrow\mathcal A_{L,t}=\pi_*\mathcal O_{\widehat Y}$ corresponds to the homomorphism
  \begin{equation}\label{eq:pullback-Linv}
    \pi^*L_K^{-1}\longrightarrow \mathcal O_{\widehat Y},\qquad a\otimes w\longmapsto aw,
  \end{equation}
  where $a$ is a local function on $\widehat Y$ and $w$ is a local section of $L_K^{-1}$, viewed as a function on $\widehat Y$.
  Over $\pi^{-1}(U_K)$, the line bundle $\pi^*L_K^{-1}$ has the basis $1\otimes z$, where $z=e^\vee$, and this map sends it to the element $z$ of the covering ring $(R\otimes_kK)[z]/(z^p-tu^{-1})$.
  The relation $z^p=tu^{-1}$ shows that $z$ is a unit, so \eqref{eq:pullback-Linv} is an isomorphism over $\pi^{-1}(U_K)$.
  These affine opens cover $\widehat Y$, so $\pi^*L_K^{-1}\simeq\mathcal O_{\widehat Y}$.
  Dualizing gives $\pi^*L_K\simeq\mathcal O_{\widehat Y}$.

  \smallskip
  (3) Let $\Omega$ be an algebraic closure of $K$, choose a $k$-embedding $\overline{k}\hookrightarrow\Omega$, and identify $\overline{k}$ with its image.
  By assumption, $X_\Omega$ is integral, and the base change $\nu_\Omega\colon Z_\Omega\to X_\Omega$ is its normalization.

  Fix an isomorphism $\alpha\colon\nu^*L_{\overline{k}}\xrightarrow{\sim}\mathcal O_Z$, and denote its base change by $\alpha_\Omega\colon L_\Omega|_{Z_\Omega}\xrightarrow{\sim}\mathcal O_{Z_\Omega}$.
  Set $s_Z:=\alpha_\Omega^{-1}(1)$.
  Let $\tau_Z\colon(L_\Omega|_{Z_\Omega})^{\otimes p}\xrightarrow{\sim}\mathcal O_{Z_\Omega}$ be the pullback of $\tau_t$ to $Z_\Omega$, and set
  \[
    c:=\tau_Z(s_Z^{\otimes p}) \in H^0(Z_\Omega,\mathcal O_{Z_\Omega})^\times =\Omega^\times.
  \]
  Let $F=\Omega(Z_\Omega)=\Omega(X_\Omega)$ be the common function field.
  At the generic point, the element $w=s_Z^\vee$ satisfies $w^p=c^{-1}$, so the generic covering algebra is \[
    F[w]/(w^p-c^{-1}).
  \]
  Since $\Omega$ is algebraically closed, we have $c^{-1}=a^p$ for some $a\in\Omega^\times$, so the generic covering algebra $F[w]/((w-a)^p)$ is non-reduced.
  Hence $(\widehat Y)_\Omega$ is non-reduced, that is, $\widehat Y$ is geometrically non-reduced over $K$.
\end{proof}

\begin{corollary}\label{cor:case-C-p-ge7}
  Let $k$ be a field of characteristic $p\ge7$, and let $X$ be a regular, geometrically non-normal projective $k$-surface such that $H^0(X,\mathcal O_X)=k$, $K_X\equiv0$. 
  Then every numerically trivial line bundle $L$ on $X$ is trivial.
\end{corollary}

\begin{proof}
  We may assume that $k$ is separably closed.
  Since $-K_X$ is nef and $p\ge7$, \cite[Theorem~1.2]{WZ26} shows that $X$ is geometrically integral.
  We have $\dim\Alb_{X/k}=0$ by Proposition~\ref{prop:a_X-surjective} since $p\geq5$.
  It follows from Proposition~\ref{prop:Alb-trivial} that $L^{\otimes p}\simeq\mathcal O_X$.
  Moreover, \eqref{eq:goal-C} in the proof of Proposition~\ref{prop:Alb-trivial} gives $\phi^*L\simeq\mathcal O_Z$ for the natural morphism $\phi\colon Z:=X_{\bar k}^\nu\to X$.
  If $L$ were nontrivial, it would have exact order $p$.
  Proposition~\ref{prop:generically-twisted-mu-p-cover} would then produce a regular projective surface $\widehat Y$ over $K=k(t)$ such that $H^0(\widehat Y,\mathcal O_{\widehat Y})=K$
  and $K_{\widehat Y}\equiv\pi^*K_{X_K}\equiv0$, while $\widehat Y$ is geometrically non-reduced.
  This contradicts \cite[Theorem~1.2]{WZ26}.
  Hence $L\simeq\mathcal O_X$.
\end{proof}

\section{Examples}\label{sec:examples}
\subsection{Diagonal quotients in characteristics two and three}\label{sec:caseB-examples}%
By Proposition~\ref{prop:a_X-surjective}\,\textup{(4)}, every surface in Case~\textup{(B)} is essentially a diagonal quotient $(E\times C)/G$, where $E$ is the Albanese variety, $C$ is an Albanese fiber, and $G\subset E$ is a finite subgroup scheme. 
Table~\ref{tab:caseB-examples} lists four such examples $X=(C\times_kE)/G$.
They show that the bounds $12$ and $6$ in Case~\textup{(B)} are optimal for $p=2$ and $p=3$, respectively.

\par\medskip\noindent
\begin{minipage}{\linewidth}
  \centering
  \begin{tabular}{cccccc}
    \hline
    Example & $p$ & $k$ & $C$ & $G$ & $\ord(K_X)$\\
    \hline
    \ref{ex:caseB-char2-order4} & $2$ & $\mathbf F_2(a,b)$ & $y^2=ax^4+bz^4$ & $\mu_4$, non-reduced & $4$\\
    \ref{ex:caseB-char2-order3} & $2$ & $\mathbf F_2(s)$ & $y^2z=x^3+sz^3$ & $\mu_3$, \'{e}tale & $3$\\
    \ref{ex:caseB-order2} & $3$ & $\mathbf F_3(s)$ & $y^2z=x^3+sz^3$ & $\mu_2$, \'{e}tale & $2$\\
    \ref{ex:caseB-order3} & $3$ & $\mathbf F_3(s,t)$ & $sx^3+ty^3+z^3=0$ & $\mu_3$, non-reduced & $3$\\
    \hline
  \end{tabular}
  \captionof{table}{Examples in Case~\textup{(B)}; $C\subset\mathbb P^2_k$ except in Example~\ref{ex:caseB-char2-order4}, where $C\subset\mathbb P_k(1,2,1)$.}
  \label{tab:caseB-examples}
\end{minipage}
\par\medskip

\begin{lemma}\label{lem:canonical-character-descent}
  Let $G$ be a finite $k$-group scheme that is commutative, \'{e}tale, or linearly reductive, and let $q\colon Y\to X$ be a $G$-torsor between projective Gorenstein $k$-schemes such that $H^0(Y,\mathcal O_Y)=k$ and $\omega_Y\simeq\mathcal O_Y$.
  Then $G$ acts on the one-dimensional $k$-vector space $H^0(Y,\omega_Y)$ through a character $\chi\colon G\to\mathbb G_m$, and
  \[
    \ord(K_X)=\ord(\chi),
  \]
  where $\ord(\chi)$ is the least positive integer $m$ such that $\chi^m=1$.
\end{lemma}

\begin{proof}
  Fix a generator $\eta$ of $\omega_Y$, so that $H^0(Y,\omega_Y)=k\eta$, and put $Y_T:=Y\times_kT$ for a $k$-scheme $T$.
  For every $T$ and every $\lambda\in G(T)$, the pullback $\lambda^*\eta$ is again a generator of $\omega_{Y_T/T}$, so $\lambda^*\eta=\chi(\lambda)\eta$ for some $\chi(\lambda)\in H^0(Y_T,\mathcal O_{Y_T}^\times)=\Gamma(T,\mathcal O_T)^\times$.
  This assignment is functorial in $T$ and satisfies $\chi(\lambda\mu)=\chi(\lambda)\chi(\mu)$, so it defines a character $\chi\colon G\to\mathbb G_m$.

  By fpqc descent 
  along the $G$-torsor $q$, a line bundle on $X$ is the same as a $G$-equivariant line bundle on $Y$.
  We first show that $q^*\omega_X\simeq\omega_Y$ as $G$-equivariant line bundles.
  By Grothendieck duality \cite[Tag~\href{https://stacks.math.columbia.edu/tag/0C10}{0C10}]{StacksProject}, the canonical isomorphism $\omega_Y\simeq q^*\omega_X\otimes\omega_{Y/X}$ is $G$-equivariant.
  Thus it suffices to show that $\omega_{Y/X}$ has a $G$-invariant generator.
  We only treat the case $G = \mu_n$, which suffices for the applications below, and refer to \cite[Tags~\href{https://stacks.math.columbia.edu/tag/0FWI}{0FWI} and~\href{https://stacks.math.columbia.edu/tag/0ATX}{0ATX}]{StacksProject} if $G$ is \'{e}tale, to \cite[Lemma~4.3.1(ii)]{Brion-equivariantly-normal} if $G$ is linearly reductive, and to \cite[\S\S4.1.4--4.1.5]{Vial-2026} if $G$ is commutative.

  Suppose $G = \mu_n$, and write $q_*\mathcal O_Y = \bigoplus_{i\in\mathbb Z/n\mathbb Z}\mathcal L_i$ for the grading, where $\mathcal L_0 = \mathcal O_X$.
  Let $\varepsilon\colon q_*\mathcal O_Y\to\mathcal O_X$ be the projection onto degree zero.
  As in the proof of Proposition~\ref{prop:generically-twisted-mu-p-cover}\,\textup{(2)}, the torsor grading makes the pairing $(b,c)\mapsto\varepsilon(bc)$ perfect.
  By duality, $q_*\omega_{Y/X}\simeq\mathcal{H}om_{\mathcal O_X}(q_*\mathcal O_Y,\mathcal O_X)$ as $q_*\mathcal O_Y$-modules, where $(b\cdot\varphi)(c) = \varphi(bc)$.
  Then the $q_*\mathcal O_Y$-linear map $q_*\mathcal O_Y\to q_*\omega_{Y/X}$ sending $1$ to $\varepsilon$ is an isomorphism, so $\varepsilon$ generates $\omega_{Y/X}$.
  Write $\varepsilon_T$ for the base change of $\varepsilon$ to a $k$-scheme $T$.
  For every $\lambda\in G(T)$ and every local section $\sum_i b_i$ of the base-changed graded algebra, where $b_i$ has degree $i$, we have $\lambda^*b_i = \lambda^i b_i$ and hence
  \[
    \varepsilon_T\!\Bigl(\lambda^*\sum_i b_i\Bigr)
    = \varepsilon_T\!\Bigl(\sum_i\lambda^i b_i\Big)
    = b_0
    = \varepsilon_T\!\Bigl(\sum_i b_i\Bigr).
  \]
  Thus $\varepsilon$ is a $G$-invariant generator of $\omega_{Y/X}$, so $q^*\omega_X\simeq\omega_Y$ as $G$-equivariant line bundles.

  Consequently, for every $m\ge1$, the line bundle $\omega_X^{\otimes m}$ is trivial if and only if $\omega_Y^{\otimes m}\simeq\mathcal O_Y$ as $G$-equivariant line bundles, that is, if and only if $\omega_Y^{\otimes m}$ has a $G$-invariant generator.
  Every generator of $\omega_Y^{\otimes m}$ is a constant multiple of $\eta^{\otimes m}$, so such an invariant generator exists if and only if $\chi^m=1$.
  This proves the assertion.
\end{proof}

\medskip
For a scheme $Y$ over a field $k$ and a point $y\in Y(k)$, write $y_S\colon S\to Y_S$ for the induced section over a $k$-scheme $S$.
The \emph{stabilizer} $\operatorname{Stab}(y)\subset\operatorname{Aut}_k(Y)$ is the subgroup functor defined by
\[
  \operatorname{Stab}(y)(S)
  :=\{\varphi\in\operatorname{Aut}_S(Y_S)\mid\varphi\circ y_S=y_S\}.
\]
Given an action $\alpha\colon H\to\operatorname{Aut}_k(Y)$ of a $k$-group scheme $H$, the \emph{fixed point functor} $Y^H\subset Y$ is defined on each $k$-scheme $S$ by
\[
  Y^H(S):=\{x\in Y(S)\mid h\cdot x_{S'}=x_{S'}
  \text{ for every }S'\to S\text{ and every }h\in H(S')\},
\]
where $x_{S'}$ denotes the pullback of $x$ to $S'$.
The action fixes $y$ scheme-theoretically if $y\in Y^H(k)$, or equivalently, if $\alpha(H)\subseteq\operatorname{Stab}(y)$; see \cite[Definition~2.2.5 and Theorem~2.2.6]{Brion-structure-2017}.

\begin{lemma}[Regular geometrically non-normal diagonal quotients]\label{prop:diagonal-quotient-nonnormal}
  Let $C$ be a regular curve over $k$.
  Let $G:=\mu_n$ act on $C$, and let $E\to B$ be a $G$-torsor with $E$ a smooth curve over $k$.
  Then the quotient $X:=(C\times_kE)/G$ exists as a $k$-scheme.
  Write $\alpha\colon G_{\bar k}\to\operatorname{Aut}(C_{\bar k})$ for the induced action.
  \begin{enumerate}[label=\textup{(\arabic*)}]
    \item The scheme $X$ is regular.
      If $C$ is geometrically reduced, then $X$ is geometrically reduced.
    \item Assume that $C$ is geometrically reduced and that $C_{\bar k}$ has a unique singular point $c$.
      If
      \[
        \alpha(G_{\bar k})\subseteq\operatorname{Stab}(c),
      \]
      then $X$ is geometrically non-normal.
  \end{enumerate}
\end{lemma}

\begin{proof}
  The diagonal $G$-action is free because $E\to B$ is a $G$-torsor.
  By \cite[Tag~\href{https://stacks.math.columbia.edu/tag/07S7}{07S7}]{StacksProject}, the quotient $X$ exists as a scheme and the quotient map $q\colon C\times_kE\to X$ is a $G$-torsor, hence finite and faithfully flat.
  Since $C$ is regular and $E$ is smooth over $k$, the product $C\times_kE$ is regular.
  Since $q$ is faithfully flat, $X$ is regular by descent.
  If $C$ is geometrically reduced, then $C_{\bar k}\times_{\bar k}E_{\bar k}$ is reduced, and reducedness descends to $X_{\bar k}$.
  This proves \textup{(1)}.

  For \textup{(2)}, work over $\bar k$ and write $\overline C:=C_{\bar k}$, $\overline E:=E_{\bar k}$, $\overline B:=B_{\bar k}$, $\overline G:=G_{\bar k}$, and $\overline X:=X_{\bar k}$.
  The fixed point $c$ gives a $\overline G$-equivariant section $\overline E\to\overline C\times_{\bar k}\overline E$, which descends to a section
  \[
    s\colon\overline B\longrightarrow\overline X
  \]
  of the natural morphism $f\colon\overline X\to\overline B$.
  The curve $\overline B$ is smooth, since geometric regularity descends through the torsor $\overline E\to\overline B$ by \cite[Tag~\href{https://stacks.math.columbia.edu/tag/07NH}{07NH}]{StacksProject}.
  The two projections induce the following diagram:
  \[
    \begin{tikzcd}
      \overline C \times_{\bar k} \overline E \rar["q_{\bar k}"] & \overline X \rar \dar["f"'] & \overline C/\overline G \\
      & \overline B \ar[u,bend right,"s"']
    \end{tikzcd}
  \]
  Let $\eta$ be the generic point of a component of $\overline B$, put $\ell:=\kappa(\eta)$, and let $\xi:=s(\eta)\in\overline X_\eta(\ell)$.
  Let $\Omega$ be an algebraic closure of $\ell$ and choose a point $e\in\overline E_\eta(\Omega)$.
  The point $e$ trivializes the torsor $(\overline E_\eta)_\Omega$, so the map $x\mapsto(x,e)$ followed by the quotient map gives an isomorphism $\overline C_\Omega\simeq(\overline X_\eta)_\Omega$.
  By the construction of $s$, this isomorphism sends $c_\Omega$ to $\xi_\Omega$.
  Thus $\overline X_\eta$ is not smooth at $\xi$.
  Since $\xi$ is $\ell$-rational, its local ring is not regular by \cite[Tag~\href{https://stacks.math.columbia.edu/tag/00TV}{00TV}]{StacksProject}.
  So $\overline X$ is not normal at $\xi$.

\end{proof}

\begin{remark}\label{rmk:diagonal-quotient-cusp}
  Under the assumptions on $C$ in Lemma~\ref{prop:diagonal-quotient-nonnormal}\textup{(2)}, every element of $G_{\bar k}(\bar k)$ fixes $c$, since automorphisms preserve singular points and $c$ is the only singular point.
  The stabilizer condition requires the orbit morphism $G_{\bar k}\to C_{\bar k}$, $g\mapsto g\cdot c$, to be constant with value $c$ as a morphism of schemes.
  If $G$ is \'{e}tale, then $G_{\bar k}$ is a disjoint union of copies of $\operatorname{Spec}\bar k$, so fixing $c$ on geometric points is enough to verify this condition.
  Thus, when $G$ is \'{e}tale, Lemma~\ref{prop:diagonal-quotient-nonnormal}\textup{(2)} applies as soon as $C$ is geometrically reduced and $C_{\bar k}$ has a unique singular point.
  For non-reduced $G$, the orbit morphism may be nonconstant even though every geometric point maps to $c$, so the stabilizer condition must be checked separately.
\end{remark}

\begin{example}[$p=2$ and $\ord(K_X)=4$]\label{ex:caseB-char2-order4}%
  Put $k:=\mathbf F_2(a,b)$, where $a,b$ are algebraically independent over $\mathbf F_2$.
  Let $C:=V(y^2-ax^4-bz^4)\subset\mathbb P_k(1,2,1)$.
  The curve $C$ is regular and integral, with $H^0(C,\mathcal O_C)=k$, $p_a(C)=1$, and $\omega_C\simeq\mathcal O_C$.
  Let $\mu_4$ act on $C$ by
  \[
    \lambda\cdot[x:y:z]=[\lambda x:y:z],
    \qquad \lambda^4=1.
  \]
  On $U:=C\cap\{z\ne0\}$, put $\xi:=x/z$, $\eta:=y/z^2$, $f:=\eta^2-a\xi^4-b$, and $\mathcal I:=(f)$.
  By adjunction,
  \[
    \omega_C|_U\simeq
    \left.\Omega^2_{\mathbb A^2_k/k}\right|_U
    \otimes_{\mathcal O_U}(\mathcal I/\mathcal I^2)^\vee.
  \]
  A global generator $\tau$ is represented on $U$ by $(d\xi\wedge d\eta)\otimes f^\vee$, where $f^\vee(f\bmod\mathcal I^2)=1$.
  The action sends $\xi$ to $\lambda\xi$ and fixes $\eta$, $f$, and $f^\vee$, so
  \[
    \lambda^*\tau
    =(\lambda\,d\xi\wedge d\eta)\otimes f^\vee
    =\lambda\tau.
  \]

  Choose an elliptic curve $E$ over $k$ together with an embedding $\mu_4\hookrightarrow E$ (for example, $E\colon y^2+xy=x^3+1$).
  Let $\mu_4$ act on $E$ by translation and set
  \[
    X:=(C\times_kE)/\mu_4.
  \]
  We have $H^0(X,\mathcal O_X)=k$.
  The quotient map is a $\mu_4$-torsor, $X$ is projective, and Lemma~\ref{prop:diagonal-quotient-nonnormal}\textup{(1)} gives regularity.

  Put $D:=E/\mu_4$.
  Choose a nonempty affine open $\operatorname{Spec}A\subset D$ and a unit $u\in A^\times$ such that the restriction of the torsor $E\to D$ has the form
  \[
    E_A:=E\times_D\operatorname{Spec}A
    \simeq\operatorname{Spec}\bigl(A[r]/(r^4-u)\bigr)
    \longrightarrow\operatorname{Spec}A.
  \]
  In this presentation, $\mu_4$ fixes $A$ and acts by $\lambda^*r=\lambda r$.
  Consider the diagonal $\mu_4$-action on $U\times_k E_A$, where $U=C\cap\{z\ne0\}$.
  Since $r$ is invertible, $w:=\xi/r$ is an invariant function, and the invariant subring of the coordinate ring of $U\times_k E_A$ is
  \[
      R_0:=\bigl(A[\xi,\eta,r]/(\eta^2-a\xi^4-b,\ r^4-u)\bigr)^{\mu_4} 
      \simeq A[w,\eta]/(\eta^2-auw^4-b).
  \]
  Thus $\operatorname{Spec}R_0$ is an affine open subset of $X$.
  Choose $\alpha,\beta\in\overline{k}$ with $\alpha^2=a$ and $\beta^2=b$.
  After base change to $\overline{k}$, this equation becomes
  \[
    (\eta+\beta)^2=\alpha^2u w^4.
  \]
  Put $v:=\eta+\beta$.
  Since $E_{\overline{k}}$ is reduced, $u\in K(D_{\overline{k}})$ is not a square, and thus $R_0\otimes_k\overline{k}$ is an integral domain.
  At the height-one prime $\mathfrak p:=(w,v)$ of $R_0\otimes_k\overline{k}$, the local ring is
  \[
    L:=(R_0\otimes_k\overline{k})_{\mathfrak p}
    \simeq K(D_{\overline{k}})[w,v]_{(w,v)}/(v^2-\alpha^2u w^4).
  \]
  The ring $L$ has dimension one and maximal ideal $\mathfrak m=(w,v)L$.
  Since the defining equation has no linear terms, the classes of $w$ and $v$ form a $K(D_{\overline{k}})$-basis of $\mathfrak m/\mathfrak m^2$.
  Thus $L$ has embedding dimension two and is not regular, so $X$ is geometrically non-normal.

  If $\sigma$ is an invariant differential on $E$, then $\tau\otimes\sigma$ generates $\omega_{C\times_kE}$ and satisfies
  \[
    \lambda^*(\tau\otimes\sigma)=\lambda(\tau\otimes\sigma).
  \]
  Thus the character is $\chi\colon\mu_4\to\mathbb G_m$, given by $\chi(\lambda)=\lambda$.
  This character has exact order four.
  By Lemma~\ref{lem:canonical-character-descent}, we have $\ord(K_X)=4$, so $X$ satisfies the hypotheses of Case~\textup{(B)}.
\end{example}

\begin{example}[$p=2$ and $\ord(K_X)=3$]\label{ex:caseB-char2-order3}%
  Put $k:=\mathbf F_2(s)$, and let
  \[
    C:=\Proj k[x,y,z]/(y^2z-x^3-sz^3).
  \]
  The curve $C$ is a regular plane cubic with $H^0(C,\mathcal O_C)=k$ and $\omega_C\simeq\mathcal O_C$.
  Let $G:=\mu_3$ act on $C$ by
  \[
    \lambda\cdot[x:y:z]=[\lambda x:y:z].
  \]
  Choose an elliptic curve $E$ over $k$ with an embedding $\mu_3\hookrightarrow E$ (for example, $E\colon y^2+y=x^3$), and let $G$ act on $E$ by translation.
  Set
  \[
    X:=(C\times_kE)/G.
  \]
  As in Example~\ref{ex:caseB-char2-order4}, the quotient $X$ is projective with $H^0(X,\mathcal O_X)=k$, and the projection induces a nonconstant morphism $X\to E/G$.
  By Lemma~\ref{prop:diagonal-quotient-nonnormal}\textup{(1)}, $X$ is regular.

  Choose $\beta\in\overline{k}$ with $\beta^2=s$.
  Then $C_{\overline{k}}$ has affine equation $(y+\beta)^2=x^3$, so it is reduced with a unique cusp at $c=(0,\beta)$.
  This cusp is fixed scheme-theoretically by $G_{\overline{k}}$, so $X$ is geometrically non-normal by Lemma~\ref{prop:diagonal-quotient-nonnormal}\textup{(2)}.

  By the same adjunction calculation as in Example~\ref{ex:caseB-char2-order4}, the character on $H^0(C,\omega_C)$ is $\lambda\mapsto\lambda$, of exact order three.
  Tensoring with an invariant differential on $E$ preserves this character, so by Lemma~\ref{lem:canonical-character-descent} we have $\ord(K_X)=3$.
\end{example}

\begin{example}[$p=3$ and $\ord(K_X)=2$]\label{ex:caseB-order2}%
  Let $k:=\mathbf F_3(s)$ and let
  \[
    C:=\Proj k[x,y,z]/(y^2z-x^3-sz^3).
  \]
  The curve $C$ is a regular plane cubic with $H^0(C,\mathcal O_C)=k$ and $\omega_C\simeq\mathcal O_C$.
  Let $G:=\mu_2$ act on $C$ by
  \[
    \lambda\cdot[x:y:z]=[x:\lambda y:z].
  \]
  Choose an elliptic curve $E$ over $k$ with an embedding $\mu_2\hookrightarrow E$, and let $G$ act on $E$ by translation.
  Set
  \[
    X:=(C\times_kE)/G.
  \]
  As in Example~\ref{ex:caseB-char2-order4}, the quotient $X$ is projective with $H^0(X,\mathcal O_X)=k$.
  By Lemma~\ref{prop:diagonal-quotient-nonnormal}\textup{(1)}, $X$ is regular.

  If $\alpha^3=s$ in $\overline{k}$, then $C_{\overline{k}}$ has affine equation $y^2=(x+\alpha)^3$, so it is reduced with a unique cusp at $c=(-\alpha,0)$.
  This cusp is fixed scheme-theoretically by $G_{\overline{k}}$, so $X$ is geometrically non-normal by Lemma~\ref{prop:diagonal-quotient-nonnormal}\textup{(2)}.

  By the same adjunction calculation as in Example~\ref{ex:caseB-char2-order4}, the character on $H^0(C,\omega_C)$ is $\lambda\mapsto\lambda$, of exact order two.
  Tensoring with an invariant differential on $E$ preserves this character, so by Lemma~\ref{lem:canonical-character-descent} we have $\ord(K_X)=2$.
  Thus $X$ satisfies the hypotheses of Case~\textup{(B)}.
\end{example}

\begin{example}[$p=3$ and $\ord(K_X)=3$]\label{ex:caseB-order3}%
  Let $k:=\mathbf F_3(s,t)$ and let
  \[
    C:=\Proj k[x,y,z]/(sx^3+ty^3+z^3).
  \]
  The curve $C$ is a regular integral plane cubic with $H^0(C,\mathcal O_C)=k$ and $\omega_C\simeq\mathcal O_C$.
  Let $G:=\mu_3$ act on $C$ by
  \[
    \lambda\cdot[x:y:z]=[\lambda x:y:z].
  \]
  Choose an elliptic curve $E$ over $k$ with an embedding $\mu_3\hookrightarrow E$, and let $G$ act on $E$ by translation.
  Set
  \[
    X:=(C\times_kE)/G.
  \]
  As in Example~\ref{ex:caseB-char2-order4}, the quotient $X$ is projective with $H^0(X,\mathcal O_X)=k$.
  By Lemma~\ref{prop:diagonal-quotient-nonnormal}\textup{(1)}, $X$ is regular.

  Put $D:=E/G$.
  As in Example~\ref{ex:caseB-char2-order4}, cover $D$ by nonempty affine opens $\operatorname{Spec}A$ on which the torsor has the form
  \[
    E_A:=E\times_D\operatorname{Spec}A
    \simeq\operatorname{Spec}\bigl(A[r]/(r^3-u)\bigr),
    \qquad u\in A^\times.
  \]
  The action fixes $A$ and satisfies $\lambda^*r=\lambda r$.
  For $U:=C\cap\{z\ne0\}$, put $v:=y/z$ and $w:=x/(rz)$ on $U\times_k E_A$.
  The invariant subring of the coordinate ring of $U\times_k E_A$ under the diagonal $G$-action is
  \[
    R_0:=A[v,w]/(su w^3+tv^3+1).
  \]
  Similarly, on $V:=C\cap\{y\ne0\}$, the coordinates $v:=z/y$ and $w:=x/(ry)$ give the invariant ring $R_1:=A[v,w]/(su w^3+t+v^3)$ for $V\times_k E_A$.
  The spectra of $R_0$ and $R_1$ cover $X\times_D\operatorname{Spec}A$.
  Choose $\alpha,\beta\in\overline{k}$ with $\alpha^3=s$ and $\beta^3=t$.
  After base change to $\overline{k}$, the defining equations of $R_0$ and $R_1$ become, respectively,
  \[
    \alpha^3u w^3+(\beta v+1)^3=0,
    \qquad
    \alpha^3u w^3+(v+\beta)^3=0.
  \]
  Since $E_{\overline{k}}$ is integral, $u\notin K(D_{\overline{k}})^3$, so both $R_0\otimes_k\overline{k}$ and $R_1\otimes_k\overline{k}$ are integral domains.  Hence $X$ is geometrically reduced.
  As in Example~\ref{ex:caseB-char2-order4}, the one-dimensional local ring $(R_0\otimes_k\overline{k})_{(w,\beta v+1)}$ has embedding dimension two, since its defining equation has no linear terms.
  Thus $X$ is geometrically non-normal.

  By the same adjunction calculation as in Example~\ref{ex:caseB-char2-order4}, the character on $H^0(C,\omega_C)$ is $\lambda\mapsto\lambda$, of exact order three.
  Tensoring with an invariant differential on $E$ preserves this character, so by Lemma~\ref{lem:canonical-character-descent} we have $\ord(K_X)=3$.
\end{example}

\subsection{A foliation quotient in characteristic five}\label{sec:caseC-order-five-example}%
The following example constructs a geometrically non-normal regular projective surface $Y$ in characteristic $5$ with $\ord(K_Y) = 5$.

\begin{example}\label{ex:case-C-p5}
  We first recall the surface \(S\) constructed in \cite[\S6.1]{WZ26}.
  Let
  \[
    \ell=\mathbf F_5(s_1,s_2),\qquad
    k=\ell(c_1,c_2),\qquad c_i^5=s_i.
  \]
  The construction starts with a regular projective toric \(k\)-surface \(\widetilde X\) carrying a rank-one foliation \(\mathcal D\subset T_{\widetilde X/\ell}\).
  Let \(\sigma\colon\widetilde X\to X\) be the contraction of two toric curves \(E_1\) and \(E_2\).
  Then \(X\) is a normal projective toric surface with a unique cyclic quotient singularity \(P\) of type \(\frac15(1,2)\), and \(\mathcal D\) descends to a foliation \(\mathcal F\) on \(X\).
  Let \(\pi\colon X\to S:=X/\mathcal F\) be the quotient and put \(s_0:=\pi(P)\).
  By \cite[Theorem~6.1]{WZ26}, \(S\) is a regular projective surface, geometrically non-reduced over \(\ell\), with
  \[
    H^0(S,\mathcal O_S)=\ell,\qquad K_S\equiv0.
  \]

  We will show that \(\omega_S\simeq\mathcal O_S\) and construct a nowhere-vanishing \(\ell\)-derivation \(\vartheta\) on \(S\) satisfying \(\vartheta^{[5]}=\vartheta\).
  The associated free \(\mu_5\)-action defines a quotient \(\rho\colon S\to Y:=S/\vartheta\).
  We will prove that \(Y\) is regular, projective, and geometrically non-normal over \(\ell\), with \(H^0(Y,\mathcal O_Y)=\ell\) and \(K_Y\) of exact order \(5\).
  The construction is summarized by the diagram
  \[
    \begin{tikzcd}[column sep=large,row sep=huge]
      (\widetilde X,\mathcal D)
      \arrow[d,"\sigma" left,"\substack{\text{birational contraction}\\\text{of foliated pairs}}" right] \\
      (X,\mathcal F)
      \arrow[r,"\pi"] &
      S:=X/\mathcal F
      \arrow[r,"\rho"] &
      Y:=S/\vartheta.
    \end{tikzcd}
  \]

  \par\smallskip
  \noindent\textbf{Step 1: Triviality of $\omega_S$.}
  On the affine chart \(V_1=\operatorname{Spec}k[t,w]\subset\widetilde X\), the restriction \(\sigma|_{V_1}\colon V_1\to\sigma(V_1)\) is an isomorphism onto an open subset of \(X\), so we also write \(V_1\) for its image in \(X\).
  In the function field \(K=k(t,w)\), the foliation is
  \[
    \mathcal D=T_{\widetilde X/\ell}\cap K\delta,
    \qquad
    \text{where }
    \delta=\partial_{c_1}+t\,\partial_{c_2}+w\,\partial_t
    \in\operatorname{Der}_\ell(K).
  \]
  Put \(L:=\ker\delta\), so \([K:L]=5\).
  Set \(R:=\ker(\delta|_{k[t,w]}) \subset k[t,w]\) to be the coordinate ring of \(\pi(V_1)\).
  The functions
  \[
    A:=t-c_1w,\qquad
    B:=c_2-c_1t+\tfrac12 c_1^2w,\qquad
    W:=w
  \]
  are \(\delta\)-invariant.
  Since \(k[t,w]=\bigoplus_{i=0}^{4}\ell[A,B,W]c_1^i\) and \(\delta(c_1)=1\), the functions \(A,B,W\) generate \(R\) as an \(\ell\)-algebra, with presentation
  \[
    R\simeq\ell[A,B,W]/(F),\qquad
    F=B^5-s_2+s_1A^5+\tfrac12 s_1^2W^5 .
  \]

  Adjunction for \(\operatorname{Spec}R\subset\mathbb A^3_\ell\) gives a generator
  \[
    \eta_0:= (dA\wedge dB\wedge dW)\otimes F^\vee
  \]
  of \(\omega_S\) on \(\pi(V_1)\).

  The boundary \(\widetilde X\setminus V_1\) consists of \(E_1,E_2,G_1,G_2\).
  Writing \(\Gamma_i:=\pi(\sigma(G_i))\), we therefore have \(\operatorname{div}(\eta_0)=a\Gamma_1+b\Gamma_2\) for some \(a,b\in\mathbb Z\).
  To prove that \(\eta_0\) extends to a global generator of \(\omega_S\), it remains to show that \(a=b=0\).
  Put \(H_i:=\sigma(G_i)\).
  We use the notation \((D\cdot C)_k\) for Mumford's numerical intersection pairing over \(k\).
  The toric intersection numbers in $\widetilde X$ are
  \[
    \begin{gathered}
      (E_1\cdot E_1)_k=-3,\quad (E_2\cdot E_2)_k=-2,\quad
      (G_1\cdot G_1)_k=-2,\quad (G_2\cdot G_2)_k=-1,\\
      (E_1\cdot E_2)_k=(E_2\cdot G_2)_k=(G_2\cdot G_1)_k=1,
    \end{gathered}
  \]
  and all other intersections among these four curves are zero.
  The pullbacks are
  \[
    \sigma^*H_1=G_1,\qquad
    \sigma^*H_2=G_2+\tfrac15E_1+\tfrac35E_2.
  \]
  Therefore
  \[
    (H_1\cdot H_1)_k=-2,\qquad (H_1\cdot H_2)_k=1,\qquad
    (H_2\cdot H_2)_k=-1+\tfrac35=-\tfrac25.
  \]
  Let \(e_i>0\) be defined by \(\pi^*\Gamma_i=e_iH_i\).
  Explicitly, $\pi^*\operatorname{div}(\eta_0)=ae_1H_1+be_2H_2$.
  Since \(\operatorname{div}(\eta_0)\equiv K_S\equiv0\), this divisor is numerically trivial. 
  Since the intersection matrix
  \(\bigl((H_i\cdot H_j)_k\bigr)_{1\le i,j\le2}\) is invertible, we obtain \(a=b=0\).

  \par\smallskip
  \noindent\textbf{Step 2: A free $\mu_5$-action on $S$.}
  Define \(\vartheta\in\operatorname{Der}_{\ell}(K)\) by
  \[
    \vartheta(c_1)=-c_1,\quad \vartheta(c_2)=c_2,\quad
    \vartheta(t)=2t,\quad \vartheta(w)=3w.
  \]
  This satisfies \(\vartheta^{[5]}=\vartheta\) and \([\vartheta,\delta]=\delta\).
  Moreover, \(\vartheta\) is diagonal on monomials in \(t,w\), so it preserves \(\mathcal O_X\) and \(L\).
  Here \(\vartheta(L)\subseteq L\) follows from \([\vartheta,\delta]=\delta\): if \(f\in L\), then
  \[
    \delta(\vartheta f)=\vartheta(\delta f)-[\vartheta,\delta](f)=0.
  \]
  It therefore induces a global derivation \(\vartheta\in H^0(S,T_{S/\ell})\), again denoted by \(\vartheta\), with \(\vartheta^{[5]}=\vartheta\).
  We claim that \(\vartheta\) vanishes nowhere on \(S\).
  For this, recall the six affine charts of \(\widetilde X\) and the corresponding local generators \(\partial\) of \(\mathcal D\):
  \[
    {
    \renewcommand{\arraystretch}{1.2}
    \begin{array}{c|c|c|cccc}
      \text{chart} & (y_1,y_2) & \partial & \partial(c_1) & \partial(c_2) & \partial(y_1) & \partial(y_2) \\
      \hline
      V_1 & (t,w) & \delta & 1 & t & w & 0\\
      V_2 & (t,v) & v\delta & v & tv & 1 & 0\\
      V_3 & (t',\tilde v) & t'^2\tilde v\,\delta & t'^2\tilde v & t'\tilde v & -t' & 3\tilde v\\
      V_4 & (t',\eta) & t'\delta & t' & 1 & -t'\eta & 3\eta^2\\
      V_5 & (\tilde w,\rho) & \tilde w^2\rho^2\delta & \tilde w^2\rho^2 & \rho & 2 & 0\\
      V_6 & (\xi,\tau) & \xi^2\tau\,\delta & \xi^2\tau & 1 & 2 & 0
    \end{array}
    }
  \]
  where \(t'=1/t\), \(v=1/w\), \(\tilde w=w/t^3\), \(\tilde v=1/\tilde w\), \(\eta=w/t^2\), \(\xi=t^2/w\), and \(\rho=t^5/w^2=1/\tau\).

  \smallskip
  We check the vanishing of \(\vartheta\) separately at points
  \(s:=\pi(x)\) for \(x\in X\setminus\{P\}\), and at \(s_0 = \pi(P)\).
  For \(x\in X\setminus\{P\}\), by the local structure in \cite{WZ26}, \(\mathcal D\) is a subbundle at \(x\) and \(\mathcal O_{X,x}=\bigoplus_{i=0}^{4}\mathcal O_{S,s}g^i\) for a local generator \(\partial\) of \(\mathcal D\) normalized by \(\partial(g)=1\).
  If \(\vartheta\) vanished at \(s\), then \(\vartheta(a_i)\in\mathfrak m_s\) for every \(a_i\in\mathcal O_{S,s}\).
  Every $f\in\mathcal O_{X,x}$ can be expressed as \(f=\sum_i a_i g^i\). 
  Using \(\partial(g^i)=i g^{i-1}\), one sees that \(\vartheta-\vartheta(g)\partial\) would map \(\mathcal O_{X,x}\) into \(\mathfrak m_x\), so
  \[
    \vartheta\equiv\vartheta(g)\,\partial\pmod{\mathfrak m_x T_{\widetilde X/\ell}}
  .
  \]
  This relation does not hold in any of the six charts in the above table.
  For example, on \(V_1\), the congruences \(-c_1\equiv\vartheta(g)\), \(c_2\equiv\vartheta(g)\,t\), \(2t\equiv\vartheta(g)\,w\), and \(3w\equiv0\) force \(w(x)=0\), then \(t(x)=0\), then \(c_2(x)=0\), which is absurd.
  Thus $\vartheta$ does not vanish at $s$.
  At \(s_0\), recall from \cite{WZ26} that, 
  \[
    \mathcal O_{X,P}=k\bigl[U^aT^b:\ 0\le2b\le5a\bigr]_{\mathfrak m},\qquad
    \partial=T\partial_T+U(T\partial_{c_1}+T^2\partial_{c_2}),
  \]
  where \(U=1/w\), \(T=t\), and \(\mathfrak m\) is the maximal ideal of the point \(P\).
  The invariant ring \(\mathcal O_{S,s_0}=\ker(\partial|_{\mathcal O_{X,P}})\) has residue field \(k\).
  The element \(q:=c_1-t/w\in\mathcal O_{S,s_0}\) is a unit with residue \(c_1\) and satisfies \(\vartheta(q)=-q\), so \(\vartheta\) does not vanish at \(s_0\) either.

  Since \(\vartheta^{[5]}=\vartheta\) and \(\vartheta\) vanishes nowhere, the associated \(\mu_5\)-action on \(S\) is free.

  \par\smallskip
  \noindent\textbf{Step 3: The quotient $Y$ and its canonical order.}
  Let \(\rho\colon S\to Y:=S/\vartheta\) be the quotient.
  Then \(\rho\) is a finite flat \(\mu_5\)-torsor and \(Y\) is regular and projective.
  On \(\pi(V_1)=\operatorname{Spec}R\), the derivation is diagonal:
  \[
    \vartheta(A)=2A,\qquad \vartheta(B)=B,\qquad \vartheta(W)=3W.
  \]
  Let \(T\) be an \(\ell\)-scheme and let \(\lambda\in\mu_5(T)\).
  The associated \(\mu_5\)-action is characterized on \(R_T\) by
  \[
    \lambda^*A=\lambda^2A,\qquad \lambda^*B=\lambda B,\qquad
    \lambda^*W=\lambda^3W,
  \]
  corresponding to the displayed formulas for \(\vartheta\).
  Using the formula
  \(\eta_0=(dA\wedge dB\wedge dW)\otimes F^\vee\) from Step~1 and the invariance of \(F\), we obtain
  \[
    \lambda^*\eta_0=\lambda^{2+1+3}\eta_0=\lambda\eta_0
  \]
  for every such \(T\) and \(\lambda\).
  By Lemma~\ref{lem:canonical-character-descent}, we have \(\ord(K_Y)=5\).

  \smallskip
  This gives a projective regular surface \(Y\) over \(\ell\) with \(H^0(Y,\mathcal O_Y)=\ell\), \(K_Y\equiv0\), and \(\ord(K_Y)=5\).
  Moreover, \(Y\) is geometrically non-normal and its Albanese variety is trivial: otherwise Theorem~\ref{thm:main}\,(A) or (B) would give \(12K_Y\sim0\), contradicting \(\ord(K_Y)=5\).
  Hence \(Y\) falls in case~(C), and the statement of Theorem~\ref{thm:main}\,(C) is sharp for \(p=5\).
\end{example}

\newcommand{\etalchar}[1]{$^{#1}$}

\end{document}